\documentclass[12pt,a4paper]{amsart}

\usepackage{amsfonts}
\usepackage{amsmath}
\usepackage{amssymb}
\usepackage{amsthm}
\usepackage{mathrsfs}
\usepackage{color}

\usepackage{graphicx}
\usepackage[all]{xy}

\usepackage{enumerate}

\newcommand{\f}{\varphi}

\newcommand{\aA}{{\mathcal{A}}}
\newcommand{\bB}{{\mathcal{B}}}
\newcommand{\cC}{{\mathcal{C}}}
\newcommand{\dD}{{\mathcal{D}}}
\newcommand{\eE}{{\mathcal{E}}}
\newcommand{\fF}{{\mathcal{F}}}

\newcommand{\hH}{{\mathcal{H}}}

\newcommand{\lL}{{\mathcal{L}}}
\newcommand{\oO}{{\mathcal{O}}}

\newcommand{\tT}{{\mathcal{T}}}

\newcommand{\lle}{\mbox{\raisebox{0.25ex}{${\scriptscriptstyle\le}$}}}
\newcommand{\gge}{\mbox{\raisebox{0.25ex}{${\scriptscriptstyle\ge}$}}}

\newcommand{\tr}{{\mbox{$t$-struc}\-tu\-r}}

\newcommand{\bcdot}{{\mbox{\boldmath{$\cdot$}}}}

\newcommand{\Hom}{\mathop{\textrm{Hom}}\nolimits}
\newcommand{\Ext}{\mathop{\textrm{Ext}}\nolimits}

\newcommand{\Coh}{\mathop{\textrm{Coh}}\nolimits}
\newcommand{\hHom}{\mathop{\mathcal{H}\textrm{om}}\nolimits}
\newcommand{\pPerf}{\mathop{\mathfrak{P}\textrm{erf}}\nolimits}
\newcommand{\Conn}{\mathop{\textrm{Conn}}\nolimits}

\newcommand{\Dec}{\mathop{\textrm{Dec}}\nolimits}
\newcommand{\Irr}{\mathop{\textrm{Irr}}\nolimits}
\newcommand{\Ex}{\mathop{\textrm{Ex}}\nolimits}
\newcommand{\Spec}{\mathop{\textrm{Spec}}\nolimits}

\newcommand{\Perf}{\mathop{\textrm{Perf}}\nolimits}

\newcommand{\Id}{\mathop{\textrm{Id}}\nolimits}

\newcommand{\wt}[1]{\widetilde{#1}}
\newcommand{\ol}[1]{\overline{#1}}
\newcommand{\di}[1]{\omega_{#1}|_{D_{#1}}}

\newtheorem{LEM}{Lemma}[section]

\newtheorem{THM}[LEM]{Theorem}
\newtheorem{PROP}[LEM]{Proposition}
\newtheorem{COR}[LEM]{Corollary}

\theoremstyle{definition}
\newtheorem{EXM}[LEM]{Example}
\newtheorem{REM}[LEM]{Remark}
\newtheorem{DEF}[LEM]{Definition}

\begin{document}
	
\title{Canonical tilting relative generators}
\author{Agnieszka Bodzenta}
\author{Alexey Bondal}
\date{\today}

\address{Agnieszka Bodzenta\\
	School of Mathematics, 
	University of Edinburgh \\ Peter Guthrie Tait Road \\ Edinburgh EH9 3FD,
	U. K.} \email{A.Bodzenta@ed.ac.uk}

\address{Alexey Bondal\\
	Steklov Mathematical Institute of Russian Academy of Sciences, Moscow, Russia, and
	Kavli Institute for the Physics and Mathematics of the Universe (WPI), The University of Tokyo, Kashiwa, Chiba 277-8583, Japan,  and National Research University Higher School of Economics, Russian Federation} \email{bondal@mi.ras.ru}

\begin{abstract}
	Given a relatively projective birational morphism $f\colon X\to Y$ of smooth algebraic spaces with dimension of fibers bounded by 1, we construct tilting relative (over $Y$) generators $T_{X,f}$ and $S_{X,f}$ in $\dD^b(X)$. We develop a piece of general theory of strict admissible lattice filtrations in triangulated categories and show that $\dD^b(X)$ has such a filtration $\lL$ where the lattice is the set of all birational decompositions $f \colon X \xrightarrow{g} Z \xrightarrow{h} Y$ with smooth $Z$. The \tr es related to $T_{X,f}$ and $S_{X,f}$ are proved to be glued via filtrations left and right dual to $\lL$. We realise all such $Z$ as the fine moduli spaces of simple quotients of $\oO_X$ in the heart of the \tr e for which $S_{X,g}$ is a relative projective generator over $Y$. This implements the program of interpreting relevant smooth contractions of $X$ in terms of a suitable system of \tr es on $\dD^b(X)$.
\end{abstract}	

\maketitle
\tableofcontents

\section*{\textbf{Introduction}}

This paper is devoted to the categorical study of relatively projective birational morphisms $f\colon X\to Y$ between smooth algebraic spaces with the dimension of fibres  bounded by 1. 
According to a theorem of V. Danilov such a morphism has a decomposition into a sequence of blow-ups with smooth centers of codimension 2. 
Our goal is to find a categorical interpretation for $f$ and for all possible intermediate contractions in terms of transformations of \tr es in the bounded derived category $\dD^b(X)$ of coherent sheaves on $X$.

Recall that T. Bridgeland, in his approach to proving the derived flop conjecture (see \cite{BO}) in dimension 3, introduced in \cite{Br1} a series of \tr es in $\dD^b(X)$ related to a birational morphism $f\colon X\to Y$ of projective varieties with fibers of dimension bounded by 1. The \tr es, with hearts $\ ^p{\rm Per}(X/Y)$, depended on an integer parameter $p\in \mathbb{Z}$. Under the assumption that $f$ was a flopping contraction, he used these \tr es to define the flopped variety as a moduli space of so-called {\em point objects} in $\ ^{-1}{\rm Per}(X/Y)$. 

In our setting of \emph{divisorial contractions} instead of flopping contractions, we construct a system of \tr es with nice properties and interpret all possible intermediate smooth contractions between $X$ and $Y$ as the fine moduli spaces of simple quotients of $\mathcal {O}_X$ in the hearts of those \tr es.

We study the partially ordered set ${\rm Dec}(f)$ of all decompositions for $f$ into two birational morphisms with a smooth intermediate space. We prove that it is a distributive lattice and identify it with the lattice of lower ideals in a poset ${\rm Conn}(f)$, which is a subposet in ${\rm Dec}(f)$ (see Corollary \ref{cor_Dec_id_in_Con}). We provide with various descriptions of $\Conn(f)$ and find a one-to-one correspondence between the elements of poset  ${\rm Conn}(f)$ and the set ${\rm Irr}(f)$ of irreducible components of the exceptional divisor for $f$, thus inducing a poset structure on $\Irr(f)$. This partial order is important from the categorical viewpoint, though it is very far from the incidence relation of irreducible components of the exceptional divisor even for the case of smooth surface contractions. One can assign to every element $g\in {\rm Dec}(f)$ the subset in ${\rm Irr}(f)$ of irreducible components contracted by $g$. Then the induced partial order on ${\rm Irr}(f)$ allows to identify ${\rm Dec}(f)$ with the lower ideals in poset ${\rm Irr}(f)$.

We consider the general set-up of a morphism $X \to Y$ of quasi-compact quasi-separated algebraic spaces and substacks $\mathfrak{C}$ in $\dD(X)$ over $Y$. We prove that  a relative generator $T$ in $\mathfrak{C}$ induces an equivalence of $\mathfrak{C}$ with the stack of perfect modules over the relative endomorphism algebra of $T$ (Theorem \ref{thm_relative_tilting}). Various partial versions of this statement are scattered in the literature, cf. \cite{VdB}, \cite{SpVdB}, \cite{KruPloSos}.

If $X$ is smooth, one can assign two \tr es in $\mathfrak{C}$ to a \emph{tilting} relative  generator $T$, the one when the object is relatively projective in the heart of the \tr e, and when it is relatively injective (\ref{eqtn_t-str_on_stack}), which we dubbed $T$-{\em projective} and, respectively, $T$-{\em injective} \tr e. The above equivalence is $t$-exact for the first \tr e. 

We construct for our $f$ a tilting relative  generator in $\dD^b(X)$. It has a surprisingly simple canonical form and, remarkably, it is just a direct sum of discrepancy sheaves:
$$
T_{X,f} = \omega_X \oplus \bigoplus_{g\in \Conn(f)}\omega_{X}|_{D_g},
$$
with $D_g$ the discrepancy divisor for $g\colon X \to Z$ in $\Conn(f)$, i.e. $\omega_g=g^!(\oO_Z) = \oO_X(D_g)$.

We hope that explicit dependence of this generator on the relative canonical classes  will shed light on the mystery of the conjectures that the behaviour of the derived categories under birational transformations in the Minimal Model Program, namely divisorial contractions, flops, and flips, is controlled by the canonical classes of the varieties involved \cite{BonOrl}, \cite{BO}, \cite{Kaw}.

By applying the relative duality functor $\mathscr{D}_f(-)= R\hHom_X(-,\omega_f)$, we obtain another  tilting relative generator
$$
S_{X,f}=\oO_X \oplus \bigoplus_{g\in \Conn(f)}\omega_g|_{D_g}[-1].
$$

Then we study the four \tr es in $\dD^b(X)$ where either $T_{X,f}$ or $S_{X,f}$ are relatively projective or injective.

In \cite{VdB}, M. Van den Bergh found that, if  $Y$ is the spectrum of a complete local ring, then $\ ^{-1}{\rm Per}(X/Y)$ and $\ ^{0}{\rm Per}(X/Y)$ have  projective generators,
$P=M$ or $P=M^* $ respectively, and this allowed him  to identify these two  hearts of the \tr es with the category of modules over the algebra $A_P={\rm End}P$ and to construct a derived equivalence $\dD(X)\simeq \dD({\rm mod}-A_P)$.  Some examples of this sort of equivalences were already known by that time (for instance, to the second-named author of this paper and D. Orlov in the study of intersection of quadrics by means of the sheaf of Clifford algebras, \textit{cf.} \cite{BO}).  This approach paved the way to interpreting birational geometry via non-commutative resolutions, which includes derived Mac-Kay correspondence as a very particular case. 

Note that Van den Bergh's construction of projective generators was inherently non-canonical, which implied extra conditions for gluing a relative projective generator along the base $Y$, for the case when $Y$ is not the spectrum of a complete local ring.

In order to understand the gluing properties of our \tr es, we develop a piece of general theory of $\lL$-filtrations in triangulated categories, where $\lL$ is a lattice. In particular, we introduce the notion of {\rm strict} admissible $\lL$-filtration and show that the strictness is preserved under the transit to the dual $\lL$-filtration. Under additional assumption that $\lL$ is a distributive lattice, Theorem \ref{thm_gluing_via_poset} claims that, given \tr es on all minimal subquotients of a strict $\lL$-filtration, one can construct a unique \tr e with nice gluing properties on the whole category by iterating Beilinson-Bernstein-Deligne gluing procedure. Since this part of the work is of abstract general character and can be of independent interest, we put it at the beginning of the paper. 

The {\em null-category} $\cC_f$ of the morphism $f$, whose importance for construction of spherical functors related to flops was emphasised in \cite{BodBon}, is proven to have a strict admissible ${\rm Dec}(f)$-filtration. The standard \tr e restricts to $\cC_f$ and the resulting \tr e on $\cC_f$ is glued via this filtration from the standard \tr es on categories $\dD^b(\bB_g)$, where $\bB_g$'s are the centers of the intermediate blow-ups that lead to $f$.

Objects  $T_f=\bigoplus_{g\in \Conn(f)} \omega_X|_{D_g}$ and $S_f=\bigoplus_{g\in \Conn(f)}\omega_g|_{D_g}[-1]$ are tilting generators in $\cC_f$.
We prove that the $S_f$-projective \tr e on the null category $\cC_f$ is glued via a $\Dec(f)^{op}$-filtration which to an element $(g\colon X \to Z, h\colon Z \to Y)$ in $\Dec(f)$ assigns subcategory $g^! \cC_h$. The $S_{X,f}$-projective \tr e is glued via the $\Dec(f)$-filtration extended from $\cC_f$ to $\dD^b(X)$ by adding one element to the filtration, $\dD^b(X)$ itself. 
In particular, functor $f^!:\dD^b(Y)\to \dD^b(X)$ is t-exact for the standard \tr e on $\dD^b(Y)$ and for the $S_{X,f}$-projective \tr e on $\dD^b(X)$.
	
Remarkably, the $S_{X,f}$-projective \tr e is also glued via recollement
\[
\xymatrix{\cC_f \ar[rr]|{\iota_{f*}} && \dD^b(X) \ar@<-2ex>[ll]|{\iota_f^*} \ar@<2ex>[ll]|{\iota_f^!} \ar[rr]|{Rf_*} && \dD^b(Y) \ar@<-2ex>[ll]|{Lf^*} \ar@<2ex>[ll]|{f^!} }
\]
from the $S_f$-projective \tr e on $\cC_f$ and the standard \tr e on $\dD^b(Y)$ .
Similar dual facts hold for the $T_{X,f}$-injective \tr e (see Subsection \ref{ssec_gluing_prop_for_t_str}).

Then we assign the $S_{X,g}$-projective \tr e on $\dD(X)$ to every element $(g,h)\in {\rm Dec}(f)$ as above. It is instructive to understand how they are interrelated.
To this end, we assign also a \tr e to every pair of compatible elements $(g,h)\ge (g',h')$ in ${\rm Dec}(f)$ and prove that this \tr e is related  with the $S_{X,g}$-projective \tr e and the $S_{X,g'}$-projective \tr es by one tilt in a torsion pair (see Subsection \ref{ssec_Def_f_syst_of_trs}). 

Finally, we recover $Y$ as a fine moduli space of objects in the heart of the $S_{X,g}$-projective \tr e. To be more precise, since the \tr e is glued via two 'opposite' recollements from the given \tr es on $\cC_g$ and $\dD^b(Z)$, all simple objects in its heart are either simple in $\cC_g$ or isomorphic to $g^!\oO_z$, for closed points $z\in Z$. 
Moreover, $g^!\oO_z$ are the only simple quotients of $\oO_X$ in the heart of the \tr e under consideration. 

This suggests to consider a moduli functor of simple quotients of $\oO_X$ by mimicking the Hilbert functor of $0$-dimensional subschemes. Theorem \ref{thm_Z_is_moduli_space} claims that $Z$ represents the functor, i.e. it is the fine moduli space of simple quotients of $\oO_X$ in the heart of the  $S_{X,g}$-projective \tr e. 

Here is a couple of remarks concerning our approach. First, in contrast to Bridgeland's set-up, we consider everything relatively over $Y$, including the functor of points, which allows us to avoid the assumption on (quasi) projectiveness of our algebraic spaces and some complications in proving the existence of the fine moduli space. Second, the $t$-exactness of functors $Rg_*$ and $g^!$ for our choice of \tr e allows us to show that there is a one-to-one correspondence between families of simple quotients of $\oO_X$ in the heart of $S_{X,g}$-projective \tr e and families of skyscrapers on $Z$.  This justifies our choice of the \tr e related to $g\in \Dec(f)$. For the case when $g$ is the blow-up of a smooth locus of codimension 2, our \tr e coincides with the one constructed by Bridgeland for $p=1$, but it is neither of his \tr es for more involved $g$.

In \cite{Toda} Yu.Toda, in his approach to Minimal Model Program (MMP) of birational geometry via variations of stability conditions,  considers a birational morphism of smooth projective surfaces $f\colon X\to Y$. He endows  the quotient category, $\dD^b(Y)$, with a \tr e corresponding to an ample class $\omega$ and glues it with the \tr e on $\cC_f \otimes \omega_f^{\vee}$ to get a \tr e on $\dD^b(X)$. One can check that his \tr e on $\cC_f\otimes \omega_f^\vee$ coincides for the case of surfaces with our $S_f$-projective \tr e after $\omega_f$-twist (though he constructs it in a very different way). In \cite{Toda} it is proved that the \tr e can be accompanied with the central charge to give a stability condition on $\dD^b(X)$, which allows to reconstruct $Y$ as the moduli space of stable point objects in the heart. Our approach is probably more direct, as we don't need stability conditions for reconstruction of the contracted variety.


The system of \tr es constructed in this paper and the corresponding moduli interpretation of birational contractions opens a new perspective for the categorical interpretation of the MMP by means of transformations of \tr es.

\textbf{Acknowledgement}
We would like to thank Alexander Efimov, Benjamin Hennion, Alexander Kuznetsov, \v{S}pela \v{S}penko, Yukinobu Toda, Michel Van den Bergh and Michael Wemyss  for many useful remarks. This work was partially supported by
World Premier International Research Center Initiative (WPI Initiative), MEXT, Japan. A. Bodzenta was partially supported by Polish National Science Centre grant nr 2013/11/N/ST1/03208 and by the EPSRC grant EP/K021400/1. A. Bondal was partially supported by the Russian Academic Excellence Project '5-100'. The reported study was partially supported by RFBR, research projects 14-01-00416 and 15-51-50045. This work is supported by the Program of the Presidium of
the Russian Academy of Sciences 01 'Fundamental Mathematics and
its Applications' under grant PRAS-18-01.


\section{\textbf{Strict lattice filtrations and gluing of \tr es}}

\subsection{Strict lattice filtrations on categories}\label{ssec_strict_poset_filt}

A full subcategory $\dD_0 \subset \dD$ of a triangulated category is \emph{left}, respectively \emph{right, admissible} if the inclusion functor $\iota_{0*} \colon \dD_0 \to \dD$ has left adjoint $\iota_0^* \colon \dD \to \dD_0$, respectively right adjoint $\iota_0^! \colon \dD \to \dD_0$, and it is \emph{admissible} if it is both left and right admissible \cite{B}.

If $\dD_0\subset \dD$ is left admissible then ${}^{\perp}\dD_0 := \{ D\in \dD\,|\, \Hom(D, \dD_0) = 0\}$ is right admissible \cite[Lemma 3.1]{B} and category $\dD$ admits a semi-orthogonal decomposition (or simply SOD) $\dD = \langle \dD_0, {}^\perp \dD_0\rangle$. Similarly, for a right admissible subcategory $\dD_0 \subset \dD$, category $\dD_0^{\perp} := \{ D\in \dD\,|\, \Hom(\dD_0, D) = 0\}$ is left admissible and $\dD$ admits an SOD $\dD = \langle \dD_0^\perp, \dD_0 \rangle$. 
Conversely, for any SOD $\dD = \langle \dD_0, \dD_1\rangle$, category $\dD_0 \subset \dD$ is left admissible while $\dD_1\subset \dD$ is right admissible.

A \emph{recollement} is a diagram of triangulated categories with exact functors 
\begin{equation}\label{eqtn_recollement}
\xymatrix{\dD_0 \ar[rr]|{i_*} && \dD \ar[rr]|{j^*} \ar@<2ex>[ll]|{i^!} \ar@<-2ex>[ll]|{i^*}&& \dD_1, \ar@<2ex>[ll]|{j_*} \ar@<-2ex>[ll]|{j_!}}
\end{equation}
such that
\begin{itemize}
	\item[(r1)]\label{con_r1} functors $i_*$, $j_*$, $j_!$ are fully faithful,
	\item[(r2)]\label{con_r2} $(i^*\dashv i_* \dashv i^!)$, $(j_! \dashv j^* \dashv j_*)$ are triples of adjoint functors,
	\item[(r3)]\label{con_r3} the kernel of $j^*$ is the essential image of $i_*$.
\end{itemize}

To simplify the exposition, we omit sometimes functors $i^*$, $i^!$, $j_!$ and $j_*$ in the notation of recollement (\ref{eqtn_recollement}). 

For recollement (\ref{eqtn_recollement}), category $\iota_*\dD_0$ is clearly admissible. Two corresponding SOD's are
\begin{equation}\label{eqtn_SOD_from_recoll}
\dD = \langle i_* \dD_0, j_! \dD_1\rangle = \langle j_* \dD_1, i_* \dD_0 \rangle.
\end{equation}
Conversely, an admissible subcategory $\dD_0 \subset \dD$ gives a recollement. Indeed, let $\dD_1$ be the quotient $\dD/\dD_0$ and $j^* \colon \dD \to \dD_1$ be the quotient functor. Distinguished triangles
\begin{align*}
&\iota_{0*}\iota_0^! \to \Id \to \iota_{\dD_0^\perp *}\iota_{\dD_0^\perp}^* \to \iota_{0*}\iota_0^![1],& &\iota_{{}^\perp\dD_0*}\iota_{{}^\perp \dD_0}^! \to \Id \to \iota_{0*}\iota_{0}^* \to \iota_{{}^\perp\dD_0*}\iota_{{}^\perp \dD_0}^![1]&
\end{align*} 
imply that $\dD_1\simeq \dD_0^\perp \simeq {}^\perp \dD_0$ and under these equivalences functor $j^*$ is isomorphic to $i_{\dD_0^\perp}^*$, respectively $i_{{}^\perp \dD_0}^!$. In particular, $j^*$ has two fully faithful adjoint functors: $j_*$ is the embedding of $\dD_0^\perp$ while $j_!$ is the embedding of ${}^\perp \dD_0$. Given admissible $\dD_0 \subset \dD$, we say that (\ref{eqtn_recollement}) is \emph{the recollement w.r.t. subcategory} $\dD_0$.

For a triangulated category $\dD$, define the \emph{right admissible poset} $\textrm{rAdm}(\dD)$ to be the poset of all right admissible subcategories with the order by inclusion. Similarly, we have the \emph{left admissible poset}  $\textrm{lAdm}(\dD)$ and the \emph{admissible poset} $\textrm{Adm}(\dD)$.
In general, $\textrm{rAdm}(\dD)$, $\textrm{lAdm}(\dD)$ and $\textrm{Adm}(\dD)$ have neither unions nor intersections.

Let $\lL$ be a finite lattice with the minimal element $0$ and the maximal element $1$ and $\dD$ a triangulated category. A \emph{right admissible} $\lL$\emph{-filtration} on $\dD$ is a map of posets $\lL \to \textrm{rAdm}(\dD)$, $I \mapsto \dD_I$, $I \in \lL$, such that 
\begin{itemize}
	\item[(Ri)] $\dD_{0}  = 0$, $\dD_{1}  = \dD$,
	\item[(Rii)] for any $I,J \in \lL$, $\dD_{I\cap J} = \dD_I \cap \dD_J$, and $\dD_{I \cup J}^\perp = \dD_I^\perp \cap \dD_J^\perp$.
\end{itemize}
A \emph{left admissible} $\lL$\emph{-filtration} on $\dD$ is a map of posets $\lL \to \textrm{lAdm}(\dD)$, $I \mapsto \dD_I$, s. t. 
\begin{itemize}
	\item[(Li)] $\dD_{0} = 0$, $\dD_1 = \dD$,
	\item[(Lii)] for any $I, J \in \lL$, $\dD_{I \cap J} = \dD_I \cap \dD_J$, and ${}^\perp\dD_{I \cup J} = {}^\perp \dD_I \cap {}^\perp \dD_J$.
\end{itemize}
An \emph{admissible} $\lL$\emph{-filtration} on $\dD$ is a map of posets $\lL \to \textrm{Adm}(\dD)$ which defines both right and left admissible $\lL$-filtrations.

Furthermore, we say that a left (or right) admissible $\lL$-filtration on $\dD$ is \emph{strict} if, besides the above conditions, we have:
\begin{itemize}
	\item[(iii)] for any $I, J \in \lL$, $\Hom_{\dD_{I\cup J}/\dD_{I\cap J}}(\dD_I/\dD_{I\cap J}, \dD_J/\dD_{I\cap J})=0$ in the Verdier quotient $\dD_{I\cup J}/\dD_{I\cap J}$.
\end{itemize}

For elements $I\preceq J$ in $\lL$, we consider the Verdier quotient:
$$
\dD_{[I,J]} :=\dD_{J}/\dD_{I}.
$$

\begin{REM}\label{rem_adm_on_quotients}
	Let $\dD$ be a triangulated category with a left (or right) admissible $\lL$-filtration and $I\preceq J \preceq K$ a triple of elements in $\lL$. The fully faithful functor $i\colon \dD_J \to \dD_K$ gives fully faithful $\ol{i} \colon \dD_J/\dD_I \to \dD_K/\dD_I$. Moreover, the composite $\dD_K \xrightarrow{i^{\sharp}} \dD_J \xrightarrow{Q} \dD_J/\dD_I$ of the functor left or right adjoint to $i$ with the quotient functor $Q$ takes $\dD_I \subset \dD_K$ to zero, hence it gives a functor $\ol{i}^{\sharp} \colon \dD_K/\dD_I \to \dD_J/\dD_I$. Then $\ol{i}^* \dashv \ol{i}$, $\ol{i} \dashv \ol{i}^!$ are adjoint pairs of functors, hence $\dD_J/\dD_I \subset \dD_K/\dD_I$ is left (respectively, right) admissible.
\end{REM}

For any $I \preceq K$ in $\lL$ the set
\[
[I,K] := \{ J \in \lL\,|\, I \preceq J \preceq K\}
\]
is a lattice with the minimal element $I$ and the maximal element $K$. Remark \ref{rem_adm_on_quotients} implies that a left (respectively right) admissible $\lL$-filtration on $\dD$ induces a left (respectively right) admissible $[I,K]$-filtration on $D_{[I,K]}$.

\begin{LEM}\label{lem_functor_D_T_1_D_T_left}
	Let $\dD$ be a category with a left (resp., right) admissible $\lL$-filtration. Then, for any triple of elements $I \preceq J \preceq K$ in $\lL$, we have functors $\dD_{[I,J]} \to \dD_{[I,K]}$ and $\dD_{[J,K]} \to \dD_{[I,K]}$ which give SOD $\dD_{[I,K]} = \langle \dD_{[I,J]}, \dD_{[J,K]} \rangle$ (resp., $\dD_{[I,K]} = \langle \dD_{[J,K]}, \dD_{[I,J]}\rangle$).
\end{LEM}
\begin{proof}
	This is clear.
\end{proof}

\begin{PROP}\label{prop_strict_right_admissible_filtr}
	Let $\dD$ be a triangulated category with a strict right (or left) admissible $\lL$-filtration. Then, for any $I,J \in \lL$, the embedding of subcategories gives an exact equivalence $\dD_I/\dD_{I \cap J} \oplus \dD_J/\dD_{I \cap J}\simeq \dD_{I\cup J}/\dD_{I \cap J}$.
\end{PROP}
\begin{proof}
	Category $\dD' := \dD_{I\cup J}/\dD_{I\cap J}$ has SOD's $\dD' = \langle \mathcal{Q}_I,\dD'_I \rangle = \langle   \mathcal{Q}_J, \dD'_J \rangle$ with $\dD'_I := \dD_I/\dD_{I \cap J}$ and $\dD'_J := \dD_J/\dD_{I \cap J}$. As subcategories $\dD'_I, \dD'_J \subset \dD'$ are orthogonal, we have $\dD'_I \subset \mathcal{Q}_J$ and $\dD'_J \subset \mathcal{Q}_I$. For any $D\in \dD'$, these two SOD's yield a diagram:
	\[
	\xymatrix{D_I \ar[r] & Q_J \ar[r] & D_{IJ}\\
		D_I \ar[r] \ar[u] & D \ar[r] \ar[u] & Q_I \ar[u] \\ 
		0 \ar[r] \ar[u] & D_J \ar[r] \ar[u] & D_J \ar[u]}
	\]	
	with $D_I \in \dD'_I$, $D_J \in \dD'_J$, $Q_I \in \mathcal{Q}_I$, $Q_J \in \mathcal{Q}_J$ and with rows and columns exact triangles. Since $D_I$ and $Q_J$ are objects of $\mathcal{Q}_J$, so is $D_{IJ}$. Moreover, as $D_J, Q_I \in \mathcal{Q}_I$, we have $D_{IJ} \in \mathcal{Q}_I$. Condition (Rii) implies that $\mathcal{Q}_I \cap \mathcal{Q}_J = 0$.  It follows that $D_{IJ} = 0$, i.e. $Q_I \simeq D_J$ and $Q_J \simeq D_I$. Thus $D = D_I \oplus D_J$. 
\end{proof}

\begin{COR}\label{cor_strict_left_is_admissible}
	Let $\dD$ be a triangulated category with a strict left (or right) admissible $\lL$-filtration. If, for any $I \in \lL$, subcategory $\dD_I \subset \dD$ is admissible, then $I \mapsto \dD_I$ defines a strict admissible $\lL$-filtration.
\end{COR}
\begin{proof}
	In view of Proposition \ref{prop_strict_right_admissible_filtr}, for a strict left admissible $\lL$-filtration $I \mapsto \dD_I$ , $\dD_{I\cup J}\subset \dD$ is the smallest triangulated subcategory containing $\dD_I$ and $\dD_J$. For $D\in \dD_I^\perp \cap \dD_J^\perp$, we have $\Hom(\dD_{I \cup J}, D) =0$. Hence the inclusion $\dD_{I\cup J}^\perp \subset \dD_I^\perp \cap \dD_J^\perp$ is an equivalence, i.e. condition (Rii) is satisfied.
\end{proof}

\begin{PROP}\label{prop_right_dual_filtr}
	Let $\dD$ be a triangulated category with a right (respectively left) admissible $\lL$-filtration. Then map $\lL^\textrm{op} \to \textrm{lAdm}(\dD)$, $I \mapsto \dD_{I}^\perp$, (respectively $\lL^\textrm{op} \to \textrm{rAdm}(\dD)$, $I\mapsto {}^\perp \dD_I$) defines a left (respectively right) admissible $\lL^\textrm{op}$-filtration on $\dD$, which is strict if the original filtration is.
\end{PROP}

\begin{proof}
	Since $0$ in $\lL$ becomes $1$ in $\lL^\textrm{op}$ and vice versa, condition (i) is clearly satisfied. Let now $I_1, I_2$ be elements of $\lL$. Since for a right admissible $\dD_0 \subset \dD$, ${}^\perp(\dD_0^\perp) \simeq \dD_0$, we have: 
	\begin{align*}
	&\dD_{I_1 \cap_{\lL^\textrm{op}} I_2} = \dD_{I_1 \cup I_2}^\perp = \dD_{I_1}^\perp \cap \dD_{I_2}^\perp,&\\
	&{}^\perp \dD_{I_1 \cup_{\lL^\textrm{op}} I_2} =\dD_{I_1 \cap I_2} = \dD_{I_1} \cap \dD_{I_2} = {}^\perp(\dD_{I_1}^\perp) \cap {}^\perp(\dD_{I_2}^\perp),&
	\end{align*}
	i.e. condition (Lii) is satisfied.
	
	Assume now that the $\lL$-filtration is strict. Then, by Proposition \ref{prop_strict_right_admissible_filtr}, category $\dD_{I_1\cup I_2}$ admits an SOD $\dD_{I_1\cup I_2} = \langle \mathcal{Q}_{I_1}, \mathcal{Q}_{I_2}, \dD_{I_1 \cap I_2} \rangle$, with 
	\begin{equation}\label{eqtn_completely_orth}
	\Hom(\mathcal{Q}_{I_1}, \mathcal{Q}_{I_2}) = 0 = \Hom(\mathcal{Q}_{I_2},\mathcal{Q}_{I_1})
	\end{equation}
	and $\dD_{I_1} = \langle  \mathcal{Q}_{I_1}, \dD_{I_1\cap I_2} \rangle$, $\dD_{I_2} = \langle \mathcal{Q}_{I_2},\dD_{I_1 \cap I_2} \rangle$. We also have $\dD = \langle \dD_{I_1\cup I_2}^\perp, \mathcal{Q}_{I_1}, \mathcal{Q}_{I_2}, \dD_{I_1\cap I_2} \rangle$. It follows that $\dD_{I_1 \cup_{\lL^\textrm{op}} I_2} = \langle \dD_{I_1\cup I_2}^\perp, \mathcal{Q}_{I_1}, \mathcal{Q}_{I_2} \rangle$. As $\dD_{I_1\cup_{\lL^\textrm{op}} I_2}/\dD_{I_1\cap_{\lL^{\textrm{op}}} I_2} \simeq \langle \mathcal{Q}_{I_1}, \mathcal{Q}_{I_2} \rangle$, formula (\ref{eqtn_completely_orth}) implies that the $\lL^\textrm{op}$-filtration is strict.
\end{proof}
Assume that $\dD$ admits a right admissible $\lL$-filtration. 
\begin{DEF}
	We say that the  $\lL^\textrm{op}$-filtration on $\dD$ given by Proposition \ref{prop_right_dual_filtr} is \emph{left dual} to the original $\lL$-filtration. Similarly, for a left admissible $\lL$-filtration on $\dD$, the filtration given by Proposition \ref{prop_right_dual_filtr} is its \emph{right dual}.
\end{DEF}

If the order on $\lL$ is full, then an admissible $\lL$-filtration on $\dD$ is just an ordinary filtration 
\begin{equation}\label{eqtn_filtrat}
0 = \dD_0 \subset \dD_1 \subset \ldots \subset \dD_{n-1} \subset \dD_n = \dD
\end{equation}
with $\dD_i \subset \dD$ admissible. As for any $I, J \in \lL$, the intersection $I\cap J$ is equal to either $I$ or $J$, conditions (Rii), (Lii) and (iii) are vacuous. In particular, any admissible $\lL$-filtration is strict. For this case, the right and left dual filtrations are those defined in \cite{BK1}.

Putting $\bB_k:=\dD_k/\dD_{k-1}$ yields, for any $k =\{2, \ldots, n\} $, a recollement:
\begin{equation}\label{eqtn_recol_i}
\xymatrix{\dD_{k-1} \ar[rr] && \dD_k \ar@<-2ex>[ll] \ar@<2ex>[ll] \ar[rr]|{\wt{j}_k^*} && \bB_k \ar@<2ex>[ll]|{\wt{j}_{k*}} \ar@<-2ex>[ll]|{\wt{j}_{k!}} }
\end{equation}
Decomposition (\ref{eqtn_SOD_from_recoll}) implies by iteration two SOD's
\begin{equation}\label{eqtn_SOD_j*_j!}
\dD = \langle j_{1!}\bB_1,\ldots, j_{n!}\bB_n \rangle,\quad \dD =  \langle j_{n*}\bB_n,\ldots, j_{1*}\bB_1\rangle
\end{equation}
and 
\begin{equation*} 
\dD_k =\langle j_{1!}\bB_1, \ldots, j_{k!}\bB_k\rangle,\quad  \dD_k =\langle j_{k*}\bB_k,\ldots, j_{1*}\bB_1\rangle.
\end{equation*} 
Above $j_{k!}\colon \bB_k \to \dD$ is the composite of $\wt{j}_{k!}\colon \bB_k \to \dD_k$ with the inclusion $\dD_k \to \dD$ and similarly for $j_{k*}$. In particular, $j_{1*}=j_{1!}$. The SOD $\langle j_{1!}\bB_1, \ldots, j_{n!} \bB_n \rangle$ is the right dual to $\langle j_{n*}\bB_n, \ldots, j_{1*} \bB_1 \rangle$.

\subsection{Gluing of  \tr es}\label{ssec_gluin_of_tr}

For what follows we assume that lattice $\lL$ is distributive. An element $s\in \lL$ is \emph{join-prime} if, for any $J_1, J_2 \in \lL$, the fact that $s \preceq J_1 \cup J_2$ implies that either $s\preceq J_1$ or $s \preceq J_2$. We denote by $JP(\lL)$ the poset of join-prime elements in $\lL$. 

By Birkhoff's theorem, \cite{Birk}, a finite distributive lattice $\lL$ is isomorphic to the lattice of lower ideals in the poset $S = JP(\lL)$. 
Conversely, any finite poset $S$ defines the finite distributive lattice $I(S)$ of lower ideals in $S$. 

Recall that a subset $I$ of a poset $ S$ is a \emph{lower ideal} if, together with $s\in I$, it contains all $s'\preceq s$. We say that $T\subset S$ is \emph{interval closed} if with any two elements $t_1, t_2$ it contains all $s$ such that $t_1 \preceq s \preceq t_2$. 

We consider elements of a finite distributive lattice $\lL$ as lower ideals in $JP(\lL)$. To a pair $I\preceq J$ in $\lL$ we assign the interval closed subset $T= J \setminus I \subset S$. The set $T$ itself has a poset structure induced from $S$. Conversely, any interval closed $T \subset S$ defines two lower ideals in $S$, i.e. elements of $\lL$:
$$
I_T:=\{s\in S\,|\, s\preceq t,\textrm{ for some }t \in T\},\quad I_{<T} := I_T \setminus T.
$$

Let $\dD$ be a category with a strict admissible $\lL$-filtration. For a join-prime $s \in \lL$, i.e. an element of $S$, define category $\dD^o_{s}$ as the Verdier quotient
$$
\dD^o_{s} := \dD_s/\dD_{I_{<\{s\}}}.
$$

\begin{THM}\label{thm_gluing_via_poset}
	Let $\lL$ be a finite distributive lattice and $\dD$ a triangulated category with a strict admissible $\lL$-filtration. Given \tr es on $D_s^o$, for every $s \in JP(\lL)$, there exists a unique system of \tr es on $\dD_{[J,L]}$, for all pairs $J\preceq L \in \lL$, such that, for any $K\in [J,L]$, the \tr e on $\dD_{[J,L]}$ is glued via recollement with respect to subcategory $\dD_{[J,K]}$ from the \tr es on $\dD_{[J,K]}$ and $\dD_{[K,L]}$.
\end{THM} 

We say that  the \tr e on $\dD$ as in the above theorem is \emph{glued via the $\lL$-filtration}.

First we prove the theorem for the case when $S$ has a full order. 

\begin{proof}[Proof of Theorem \ref{thm_gluing_via_poset} in the full order case]
Given a recollement (\ref{eqtn_recollement}) of triangulated categories and \tr es on $\dD_0$ and $\dD_1$ there exists a unique \emph{glued \tr e} on $\dD$ for which functors $i_*$ and $j^*$ are $t$-exact \cite{BBD}. Recollements (\ref{eqtn_recol_i}) allow us to glue \tr es on $\bB_k$ to obtain by iteration a \tr e $(\dD^{\lle 0}, \dD^{\gge 1})$ on $\dD$.

Denote by $j_k^*\colon \dD \to \bB_k$ the functor left adjoint to $j_{k*}$ and by $j_k^!\colon \dD \to \bB_k$ the functor right adjoint to $j_{k!}$. Then, for the glued \tr e $(\dD^{\lle 0}, \dD^{\gge 1})$, we have
\begin{equation}\label{eqtn_aisles_of_glued_tr}
\begin{aligned}
\dD^{\lle 0} &= \{D \in \dD\,|\, j_k^*(D)\in \bB_k^{\lle 0},\, \textrm{ for any }\, k\},\\
\dD^{\gge 1} &= \{D \in \dD\,|\, j_k^!(D) \in \bB_k^{\gge 1},\, \textrm{ for any }\, k\}.
\end{aligned}
\end{equation}

Filtration (\ref{eqtn_filtrat}) induces a filtration
\begin{equation}\label{eqtn_quotient_filtr}
0 = \dD_k/\dD_k \subset \dD_{k+1}/\dD_k \ldots \subset \dD_{l-1}/\dD_k \subset \dD_l/\dD_k
\end{equation}
on the category $\dD_l/\dD_k$ which allows SOD's $\langle j_{k+1!}\bB_{k+1}, \ldots, j_{l!}\bB_l \rangle$ and  $\langle j_{l*} \bB_l, \ldots, j_{k+1*} \bB_{k+1} \rangle$, for any $k<l$. We define \tr e on $\dD_l/\dD_k$ as the \tr e glued from \tr es on  $\bB_{k+1}, \ldots, \bB_l$ along filtration (\ref{eqtn_quotient_filtr}).
	
It follows from (\ref{eqtn_aisles_of_glued_tr}) that, for any $j<k<l$, the \tr e on $\dD_l/\dD_j$ is glued via recollement with respect to subcategory $\dD_k/\dD_j$ from the corresponding \tr es on $\dD_k/\dD_j$ and $\dD_l/\dD_k$.
\end{proof}

\begin{proof}[Proof of Theorem \ref{thm_gluing_via_poset}]
	By Birkhoff's theorem, we can identify lattice $\lL$ with the lattice of lower ideals in the poset $S:=JP(\lL)$. To simplify the notation we write $\dD_T$ for the subcategory $\dD_{[I,J]}$ with $I\preceq J \in \lL$ and $T:=J\setminus I \subset S$. By assumption, category $\dD_{s}^o=\dD_{I_s \setminus I_{<\{s\}}}$ is endowed with a \tr e, for any $s \in S$.
	
	An ordered triple $I \preceq J \preceq K$ of elements in $\lL$ corresponds to the lower ideal $T_1 := J \setminus I$ in the interval closed $T := K \setminus I$. In order to prove the theorem we show the existence of a unique system of \tr es on $\dD_T$, for any interval closed $T \subset S$, such that, for any lower ideal $T_1 \subset T$, the \tr e on $\dD_T$ is glued via the recollement w.r.t. subcategory $\dD_{T_1}$. 
	
	We proceed by induction on $|S|$, the case $|S|=1$ being obvious.
	
	If $|S|>1$, let $s_0 \in S$ be a maximal element and put $S':=S\setminus \{s_0\}$. Let $T\subset S$ be an interval closed subset. If $s_0\notin T$, then 
	$T\subset S'$, hence the \tr e on $\dD_T$ is defined by the induction hypothesis applied to $S'$ and satisfies the required gluing property for any lower ideal $T_1 \subset T$.
	
	If $s_0 \in T$, put $T':=T \setminus \{s_0\}$. Then the \tr e on $\dD_{T'}$ is already defined by induction hypothesis on $S'$, and we define the \tr e on $\dD_T$, for any $T\subset S$ containing $s_0$, by gluing via the recollement w.r.t. subcategory $\dD_{T'}$. 
	
	Let now $T_1 \subset T$ be a lower ideal. First, assume that $s_0 \notin T_1$. Then $T_1\subset T'$ and category $\dD_T$ admits a three-step filtration $\dD_{T_1} \subset \dD_{T'} \subset \dD_T$. Since the \tr e on $\dD_T$ is glued via the recollement w.r.t. $\dD_{T'}$ and, by induction hypothesis, the \tr e on $\dD_{T'}$ is glued via the recollement w.r.t. $\dD_{T_1}$, the \tr e on $\dD_T$ is glued via the recollement w.r.t. $\dD_{T_1}$, see the proof of Theorem \ref{thm_gluing_via_poset} for the full order case.
	
	It remains to consider the case of $s_0 \in T_1$. Put $T'_1 := T'\cap T_1 = T_1\setminus s_0$. It is a lower ideal both in $T_1$ and $T'$. By construction and inductive hypothesis, the \tr e on $\dD_T$ is glued for the filtration $\dD_{T'_1} \subset \dD_{T'} \subset \dD_T$. Since the $S$-filtration on $\dD$ is strict and $T = T'\cup T_1$, the quotient $\dD_{T\setminus T'_1} = \dD_T/\dD_{T'_1}$ is isomorphic to $\dD_{T'\setminus T'_1} \oplus \dD_{T_1\setminus T'_1}$, see Proposition \ref{prop_strict_right_admissible_filtr}. It follows that the \tr e on $\dD_{T\setminus T'_1}$ glued via the recollement w.r.t. subcategory $\dD_{T'\setminus T'_1}$ is also glued via the recollement w.r.t. subcategory $\dD_{T_1\setminus T'_1}\subset \dD_{T\setminus T'_1}$. Then the proof in the full order case implies that the \tr e on $\dD_T$ is glued via the filtration $\dD_{T'_1} \subset \dD_{T_1} \subset \dD_T$.
	
	The uniqueness of the system of \tr es follows from the uniqueness of the glued \tr e. 
\end{proof}

\begin{REM}\label{rem_full_order}
	In the course of the proof we have chosen, in fact, a full order on $JP(\lL)$ compatible with the poset structure, i.e. a filtration on $\dD$ as in (\ref{eqtn_filtrat}). However, the theorem ensures that the glued \tr e on $\dD$ is independent of this choice.  
\end{REM}

\section{\textbf{The distributive lattice of decompositions for $f$}}

We consider algebraic spaces defined over an algebraically closed field.
For a birational morphism of smooth algebraic spaces $f\colon X \to Y$, we denote by $\Ex(f)\subset X$ the (set-theoretic) \emph{exceptional divisor} of $f$. We say that $f$ is \emph{relatively projective} if there exists an $f$-ample line bundle. Note that this notion is not local over $Y$ at least in the complex-analytic topology.

Let $f\colon X\to Y$ be a relatively projective birational morphism of smooth algebraic spaces with dimension of fibers bounded by 1.
\begin{DEF}
	The \emph{partially ordered set of smooth decompositions} for $f$ is
	\begin{align*}
		\textrm{Dec}(f) = \{ (g \colon X \to Z, h\colon Z \to Y) \, |\,&g\textrm{ and }h \textrm{ are birational, } Z\, \textrm{is smooth},\, f=h\circ g\}/\sim,
	\end{align*}
	where $(g\colon X \to Z, h\colon Z \to Y)\sim (g'\colon X \to Z', h' \colon Z'\to Y)$ if there exists an isomorphism $\alpha \colon Z \to Z'$ such that $\alpha g = g'$ and $h' \alpha = h$. For $(g \colon X \to Z, h\colon Z \to Y)$ and $(\wt{g} \colon X \to \wt{Z}, \wt{h} \colon \wt{Z} \to Y)$, we put $(g,h) \preceq (\wt{g}, 
	\wt{h})$ if $\wt{g}$ factors via $g$, i.e. if $ \wt{g} \colon X \xrightarrow{g} Z \xrightarrow{\f} \wt{Z},$ for some morphism $\f$. We write $(g',h')\prec (g,h)$ if $(g',h')\preceq (g,h)$ and $g'\neq g$.
\end{DEF}

By abuse of notation we shall sometimes say that $g\colon X \to Z$ is an element of $\Dec(f)$ meaning the existence of $h\colon Z \to Y$ such that $(g,h) \in \Dec(f)$.

We will prove that $\Dec(f)$ is a distributive lattice. Assuming this is proven, let us give a description of the poset $\Conn(f)$ of join-prime elements in $\Dec(f)$:
$$
\Conn(f) = JP(\Dec(f)).
$$
By Birkhoff's theorem, we should have:
$$
\Dec(f) = I(\Conn(f)).
$$
First, let us consider the abstract set-up of a poset $S$ and give another characterisation for elements of $S$ in $I(S)$. Element $s\in S$ is identified with a principal lower ideal generated by $s$. This ideal is exactly the corresponding join-prime element in $I(S)$.
\begin{LEM}
	A lower ideal $I\in I(S)$ is principal if and only if there exists a lower ideal $I_{<}$ which is maximal among all lower ideals strictly smaller than $I$.
\end{LEM} 
\begin{proof}
	If $I$ is generated by $s\in S$, then $I_{<} :=I\setminus \{s\}$.  On the other hand, if the set of maximal elements in $I$ has at least two distinct elements $s_1$, $s_2$, then $I\setminus \{s_1\}$ and $I\setminus \{s_2\}$ are non-comparable lower ideals strictly smaller than $I$. 
\end{proof}

This leads us to the definition of \emph{the poset of connected contractions for }$f$:

\begin{equation}\label{eqtn_def_Conn_f}
	\begin{aligned}
		\Conn(f) = \{(g,h) \in \Dec(f)\,|&\,g\neq \Id_X,\, \exists\, (s_g, t_g) \in \Dec(g),\, s_g \neq g,\, \textrm{ such that }\\&\forall \, (g', h') \in \Dec(g) \textrm{ with } g'\neq g,\, (g',h') \preceq (s_g,t_g)\} 
	\end{aligned}
\end{equation}
with the partial order induced from $\Dec(f)$. We shall give a more geometric description of $\Conn(f)$ in (\ref{eqtn_Conn_f_via_B_g}).

\subsection{Decomposition of birational morphisms of smooth algebraic spaces}

We recall V. Danilov's decomposition:
\begin{THM}\cite[Theorem 1]{Dan}\label{thm_Danilov}
	Let $f\colon X \to Y$ be a relatively projective birational morphism of smooth algebraic spaces with dimension of fibers less than or equal to one. Then $f$ is decomposed as a sequence of blow-ups with smooth centers of codimension two.
\end{THM}
The existence of a decomposition as in Theorem  \ref{thm_Danilov} is equivalent to the existence of smooth $\bB_f \subset Y$ such that $f$ factors through the blow-up of $\bB_f$. In order to find such a $\bB_f$, we consider the open embedding $j\colon U \to X$ of $U:=X\setminus \Ex(f)$. Then $f\circ j \colon U \to Y$ is an open embedding with complement of codimension two. Let $\lL$ be an $f$-very ample line bundle on $X$. Then $\wt{\lL}:= (f\circ j)_* j^* \lL$ is an invertible sheaf on $Y$ and functor $f_*$ applied to the canonical embedding $\lL \to j_*j^* \lL$ yields a map $f_* \lL \to \wt{\lL}$. Since twist with the pull-back of a line bundle on $Y$ preserves $f$-very ampleness of $\lL$, we can assume without loss of generality that $\wt{\lL}\simeq \oO_Y$, i.e. $f_* \lL$ is a sheaf of ideals on $Y$. We denote by $B_{\lL}\subset Y$ the closed subspace defined by this sheaf of ideals. The support of $B_{\lL}$ is the image $f(\Ex(f))$ of the exceptional divisor of $f$. 

For any point $\xi \in Y$ of codimension two, consider the multiplicity of $B_{\lL}$ at $\xi$:
$$
\nu_\xi(B_{\lL}) = \max \{\nu\,|\, f_* \lL \subset \mathfrak{m}_\xi^\nu\}.
$$
\begin{LEM}\label{lem_max_mult_for_the_first_cent}
	Let $f \colon X \to Y$ be a composition of blow-ups with smooth codimension two centers and let $B \subset Y$ be the center of the first blow-up. Then there exists an $f$-very ample line bundle $\lL$ such that $\nu_B(B_{\lL})$ is maximal among all irreducible components of $f(\Ex(f))$.
\end{LEM}
\begin{proof}
	Let $E_1, \ldots, E_n$ be the irreducible components of $\Ex(f)$. For $\xi\in Y$ of codimension two and an $f$-ample line bundle $\lL = \oO_X(- \sum a_i E_i)$ with $a_i\geq 0$ and such that $R^1f_* \lL =0$, we have:
	\begin{equation}\label{eqtn_multiplicity}
		\nu_{\xi}(B_{\lL}) = \sum_{\{E_j\,|\, f(E_j)=\xi \}} a_j.
	\end{equation}
	This is clear from the short exact sequence obtained by applying $f_*$ to
	$$
	0 \to \oO_X(-\sum a_i E_i) \to \oO_X \to \oO_{\sum a_i E_i} \to 0.
	$$
	Let $h\colon Z \to Y$ be the blow-up of $Y$ along $B$ and $f= h\circ g$ the resulting decomposition. Denote by $E\subset Z$ the exceptional divisor of $h$ and by $E_1\subset \Ex(f) \subset X$ its strict transform. Since $g$ is decomposed into a sequence of blow-ups, one can easily find a $g$-ample line bundle $\lL' =\oO_X(-\sum_{j\geq 2} b_j E_j)$ with $b_j>0$. For a $g$-ample line bundle $\lL'$, an $h$-ample line bundle $\lL_h$ and $N$  big enough, line bundle $\lL' \otimes g^* \lL_h^{\otimes N}$ is $f$-ample, cf. \cite[Theorem 1.7.8]{Laz}. As $\oO_Z(-E)$ is $h$-ample, it follows that there exists $N> \sum_{j \geq 2}b_j$ such that the line bundle $\lL:= \lL' \otimes g^* \oO_X(-N E) = \oO_X(-N E_1 - \sum_{j\geq 2} b'_j E_j)$ is $f$-ample. As $g^*\oO_X(E) = \oO_X(E_1 + \sum c_j E_j)$ with $c_j> 0$ if $f(E_j) = B$ and $c_j =0$ otherwise, we have $b'_j > b_j$ if $f(E_j)= B$ and $b'_j=b_j$ otherwise. In particular, $b'_j>0$, for all $j$. The $f$-ampleness of $\lL$ implies that there exists $m$ such that $\lL^{\otimes m}$ is $f$-very ample and $R^1f_* \lL^{\otimes m} = 0$ (cf. \cite[Theorem 1.7.6]{Laz}). As $\lL^{\otimes m} = \oO_X(-Nm E_1 -\sum m b'_j E_j)$, it follows from formula (\ref{eqtn_multiplicity}) that $\nu_B(B_{\lL}) = Nm + \sum_{\{E_j\,|\, f(E_j) = B\}} mb'_j$. As $N > \sum_{j\geq 2} b_j$, the multiplicity $\nu_B(B_{\lL})$ is maximal among all components of $f(\Ex(f))$.
\end{proof}

If $f \colon X \to Y$ is a proper birational morphism of smooth algebraic spaces such that  $\Ex(f) \neq \emptyset$ and $Y$ a spectrum of a henselian local ring with closed point $y$, then the composite of the blow-up $\tau \colon Y' \to Y$ of $Y$ along $y$ with $f^{-1}$ is well defined on the generic point $\mu$ of the exceptional divisor of $\tau$. The image $f^{-1} \tau(\mu)$ is the generic point of an irreducible component $E$ of $f^{-1}(y)$.
\begin{PROP}\cite[Theorem 3.1, Lemma 4.4]{Dan}\label{prop_Danilov_local}
	There exists a unique irreducible component $D$ of $\Ex(f)$ which contains $E$. The image $f(D)\subset Y$ is a smooth subspace and $\nu_{f(D)}(B_{\lL}) >\nu_{\xi'}(B_{\lL})$, for any $\xi'\in Y$ of codimension two. Morphism $f$ factors through the blow-up of $Y$ along $f(D)$.
\end{PROP}
Consider the set $S$ of points of codimension 2 in $Y$, such that multiplicity $\nu_\xi(B_{\lL})$ is greater than $\nu_{\xi'}(B_{\lL})$, for all codimension 2 points $\xi'\in Y$ with $\ol{\xi} \cap \ol{\xi'} \neq \emptyset$. Define algebraic subspace $\bB_f$ in $Y$ as the union of irreducible subspaces corresponding to points in $S$. It follows from Proposition \ref{prop_Danilov_local} that $\bB_f\subset Y$ is a disjoint union of smooth codimension 2 algebraic subspaces and that $\bB_f$ is independent of the choice of $\lL$. 

A direct consequence of Lemma \ref{lem_max_mult_for_the_first_cent} and Proposition \ref{prop_Danilov_local} is the following statement, which in turn implies Theorem \ref{thm_Danilov}:
\begin{PROP}\label{prop_Danilov}
	Morphism $f$ allows a decomposition $f=h \circ g$, for the blow-up $h\colon Z \to Y$ of $Y$ along $\bB_{f}$ and $\bB_f$ is maximal with the above property, i.e. if $W\subset Y$ is a smooth closed subspace of codimension two such that $f$ admits a factorisation via the blow-up of $Y$ along $W$, then $W$ is a union of some components of $\bB_f$.	
\end{PROP}
\begin{proof}
	Let $y \in B \cap W$ be a closed point where $W$ meets another component $B$ of $f(\Ex(f))$. Let $f_y$ be the base change of $f$ to the henselization of $Y$ in $y$. By Lemma \ref{lem_max_mult_for_the_first_cent}, there exists an $f_y$-ample line bundle $\lL$ such that $\nu_W(B_{\lL})> \nu_{B}(B_{\lL})$. Then by Proposition \ref{prop_Danilov_local}, $\nu_W(B_{\lL'}) > \nu_{B}(B_{\lL'})$, for any $f_y$-ample $\lL'$. By definition of $\bB_f$, $W$ is a component of $\bB_f$.	
\end{proof}

It follows from Proposition \ref{prop_Danilov} that 
\begin{align}\label{eqtn_Conn_f_via_B_g}
	\textrm{Conn}(f) = \{ (g,h) \in \textrm{Dec}(f) \, |\, \bB_g \, \textrm{has one component} \},
\end{align}
i.e. a decomposition $f\colon X \xrightarrow{g} Z \xrightarrow{h} Y$ lies in $\Conn(f)$ if there exists a unique irreducible smooth subspace $W\subset Z$ such that $g$ factors via the blow-up of $Z$ along $W$.

\begin{LEM}\label{lem_h_projective}
	Let $(g\colon X \to Z, h\colon Z \to Y)$ be an element of $\Dec(f)$. Then $g$ and $h$ are relatively projective.
\end{LEM}
\begin{proof}
	An $f$-ample line bundle is also $g$-ample hence $g\colon X \to Z$ is a relatively projective birational morphism with fibers of dimension bounded by one. By Theorem \ref{thm_Danilov}, $g$ admits a decomposition into blow-ups along closed smooth codimension two centers. This  allows us to reduce by induction the proof that $h$ is relatively projective to the case when $g \in \Dec(f)$ is the blow-up of $Z$ along an irreducible $B \subset Z$. Let $E$ denote the exceptional divisor of $g$. Note that all fibers of $g$ over closed points of $B$ are numerically equivalent relatively over $Z$. We denote by $[f_g]$ their class in the group of numerical equivalence classes of relative 1-cycles on $X$. Let $L \in \textrm{Pic}(X/Y)$ be an $f$-ample divisor and let $k := L.f_g>0$. Consider $L' := L + kE$. Since $E.f_g =-1$, $L'$ intersects all fibers of $g$ trivially, i.e. $\oO_X(L') = g^* \oO_Z(\wt{L})$, for some $\wt{L} \in \textrm{Pic}(Z/Y)$. 
	
	Choose a closed point $y \in Y$ and an irreducible component $C$ of the fiber $h^{-1}(y)$. Since the dimension of fibers of $f$ is bounded by one, the preimage of $C$ in $X$ has dimension one. Let $C'$ be the component of $f^{-1}(y)$ such that $g|_{C'} \colon C' \to C$ is an isomorphism. $C'$ is not contained in $E$, hence $E.C' \geq 0$. As $L$ is $f$-ample, we have $L.C'>0$. It shows that $\wt{L}.C = L'.C'= L.C' + kE.C' >0$. By the numerical criterion of $h$-ampleness \cite[Theorem 1.7.8]{Laz}, $\wt{L}$ is $h$-ample.
\end{proof}

\subsection{$\Dec(f)$ as a distributive lattice}

For a relatively projective birational morphism $f\colon X\to Y$ of smooth algebraic spaces with dimension of fibers bounded by 1, we denote by $\textrm{Irr}(f)$ the set of irreducible components of the exceptional divisor for $f$. 
For $(g,h)\in \Dec(f)$, we consider $\Irr(g)$ as a subset of $\Irr(f)$. Strict transform along $g$ allows us also to view $\Irr(h)$ as a subset of $\Irr(f)$ as well. Then clearly $\Irr(f)$ is a disjoint union of $\Irr(g)$ and $\Irr(h)$.

Let $\ol{h} \colon \ol{Z} \to Y$ be the blow-up of $Y$ along $\bB_f$. The strict transform of the exceptional divisor of $\ol{h}$ to $X$ allows us to consider the set $\Irr(\bB_f)\simeq \Irr(\ol{h})$ of irreducible components of $\bB_f$ as a subset of $\Irr(f)$. 

\begin{LEM}\label{lem_inclusion_of_B_h}
	Let $(g, h)$ be an element of $\Dec(f)$. Then $\bB_h \subset \bB_f$. 
\end{LEM}
\begin{proof}
	It follows from Proposition \ref{prop_Danilov} because $f$ factors via $h\colon Z \to Y$ and $h$ factors via the blow-up of $Y$ along $\bB_h$.
\end{proof}

\begin{LEM}\label{lem_disjoint_or_comm_comp}
	If $f=h_1\circ g_1=h_2 \circ g_2$, then either $\bB_{h_1}$ and $\bB_{h_2}$ have a common component or they are disjoint.
\end{LEM}
\begin{proof}
	It follows from Lemma \ref{lem_inclusion_of_B_h} as both $\bB_{h_1}$ and $\bB_{h_2}$ are unions of components of smooth $\bB_f$.
\end{proof}

\begin{PROP}\label{prop_g_defined_by_Irr_g}
	An element $g\in \Dec(f)$ is uniquely determined by $\Irr(g)\subset \Irr(f)$.	
\end{PROP}
\begin{proof}
	Let $g_1\colon X\to Z_1$, $g_2\colon X \to Z_2$ be elements of $\Dec(f)$ such that $\Irr(g_1) = \Irr(g_2)$. We shall show that $g_2 \circ g_1^{-1} \colon Z_1 \dashrightarrow  Z_2$ is a regular morphism, hence an isomorphism with inverse $g_1 \circ g_2^{-1}$. 
		
	Let $z_1 \in Z_1$ be a closed point. If $z_1\notin g_1(\Ex(g_1))$, then $g_2 \circ g_1^{-1}$ is regular in a neighbourhood of $z_1$. If $z_1 \in g_1(\Ex(g_1))$, then set-theoretically $g_1^{-1}(z_1)$ is a connected component of $\Ex(g_1) \cap f^{-1}(h_1(z_1))$. Since set-theoretically $\Ex(g_1) = \Ex(g_2)$, $g_1^{-1}(z_1)$ is a connected component of $\Ex(g_2) \cap f^{-1}(h_1(z_1))$, i.e. it is a fiber of $g_2$ over some $z_2 \in Z_2$. Thus $g_2 g_1^{-1}(z_1) = z_2$, i.e. the map is defined on closed-points.
	
	Let $U \subset Z_2$ be an open neighbourhood of $z_2$ and let $\f$ be a regular function on $U$. Then $g_2^{-1}(U) \subset X$ is an open subset which contains the fiber $g_1^{-1}(z_1)$. Since $g_1$ is proper and surjective, there exists an open neighbourhood $V$ of $z_1$ contained in $g_1(g_2^{-1}(U))$. This proves that the map $g_2 \circ g_1^{-1}$ is continuous. An appropriate restriction of $\f|_{U \setminus g_2(\Ex(g_2))}$ can be viewed as a function on $V\setminus g_1(\Ex(g_1))$. It has a unique extension to a function $\ol{\f}$ on $V$ (because $Z_1$ is normal and $g_2(\Ex(g_2))$ is of codimension two). Then, $(\f, U) \mapsto (\ol{\f},V)$ defines a homomorphism of local rings $\oO_{Z_2,z_2} \to \oO_{Z_1, z_1}$.	
\end{proof}

\begin{PROP}\label{prop_h_1_and_h_2_don_int}
	Let $f=h_1\circ g_1 = h_2\circ g_2$ be decompositions such that $\bB_{h_1} \cap \bB_{h_2} = \emptyset$. Then $h_1(\Ex(h_1))\cap h_2(\Ex(h_2))= \emptyset$. 
\end{PROP}
\begin{proof}
	We proceed by induction on $|\Irr(f)|$, the case $|\Irr(f)|=1$ being clear.
	
	Let $\tau_1$ be the blow-up of $Y$ along an irreducible component of $\bB_{h_1}$ and $\tau_2$ the blow-up of $Y$ along an irreducible component of $\bB_{h_2}$. Since $\bB_{h_1}\cap \bB_{h_2}=\emptyset$, the fiber product $\wt{Y} = Y_1 \times_Y Y_2$ is smooth. Morphisms $h_1$ and $h_2$ admit decompositions $h_1 = \tau_1 \circ h'_1$, $h_2 = \tau_2 \circ h'_2$. By the universal property of the fiber product, morphisms $h'_1 \circ g_1 \colon X \to Z_1$ and $h'_2 \circ g_2 \colon X \to Z_2$ give $\wt{f} \colon X \to \wt{Y}$:
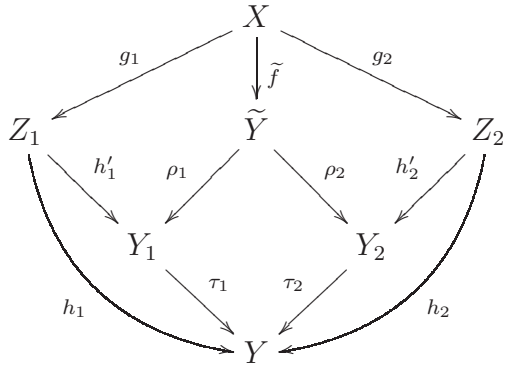
	\[
	\xymatrix{&& X \ar[dll]_{g_1} \ar[drr]^{g_2} \ar[d]^{\wt{f}} && \\
		Z_1 \ar[dr]^{h'_1} \ar@/_2pc/[ddrr]_{h_1} & & \wt{Y} \ar[dl]_{\rho_1} \ar[dr]^{\rho_2} & & Z_2 \ar[dl]_{h'_2} \ar@/^2pc/[ddll]^{h_2} \\
		& Y_1 \ar[dr]^{\tau_1} && Y_2 \ar[dl]_{\tau_2} & \\
		&& Y&&}
	\]
	Assume that $\bB_{h'_1} \cap \bB_{\rho_1} \neq  \emptyset$. As $\bB_{\rho_1} \simeq \tau_1^{-1} \bB_{\tau_2}$ is irreducible, Lemma \ref{lem_disjoint_or_comm_comp} implies that $\bB_{\rho_1} \subset \bB_{h'_1}$, hence there exists $\f\colon Z_1 \to \wt{Y}$ such that $h'_1 = \rho_1 \circ \f$. Then, $h_1 = \tau_1 \circ h'_1 = \tau_1 \circ \rho_1 \circ \f  = \tau_2 \circ \rho_2 \circ \f$, i.e. $\bB_{\tau_2} \subset \bB_{h_1}$ which contradicts the assumption. 
	Thus, $h'_1 \circ g_1 = \rho_1 \circ \wt{f}$ are decompositions such that $\bB_{h'_1} \cap \bB_{\rho_1} = \emptyset$. As $|\Irr(\rho_1 \circ \wt{f})|= |\Irr(f)|-1$, the inductive hypothesis imply that $h'_1(\Ex(h'_1)) \cap \rho_1(\Ex(\rho_1)) = \emptyset$. Since $\tau_1$ is an isomorphism in the neighbourhood of $\bB_{\rho_1}$, it follows that $\tau_1 h'_1(\Ex(h'_1))$ is disjoint from $\tau_1 \rho_1\Ex(\rho_1) =\tau_2(\Ex(\tau_2))$. Since $h_1 (\Ex(h_1)) = \tau_1(\Ex(\tau_1)) \cup h_1(\Ex(h'_1))$, we conclude that $h_1(\Ex(h_1)) \cap \tau_2(\Ex(\tau_2)) = \emptyset$. 
	
	It implies that the fiber product $Z_1 \times_Y Y_2$ is smooth and $\tau_2$ is an isomorphism in the neighbourhood of $\ol{h}_1(\Ex(\ol{h}_1))$, for the projection $\ol{h}_1 \colon Z_1 \times_Y Y_2 \to Y_2$. 
	
	Let us assume that the intersection $\bB_{\ol{h}_1} \cap \bB_{h'_2} $ is not empty. In view of Lemma \ref{lem_disjoint_or_comm_comp}, $\bB_{\ol{h}_1}$ and $\bB_{h'_2}$ have a common component $B$, i.e. $\ol{h}_1$ and $h'_2$ factor via the blow up $\psi \colon W\to Y_2$ of $Y_2$ along $B$:
	\[
	\xymatrix{& W \ar[d]_\psi & \\
		Z_1 \times_Y Y_2 \ar[r]^(0.6){\ol{h}_1} \ar[ur] \ar[d] & Y_2 \ar[d]_{\tau_2} & Z_2 \ar[l]_{h'_2} \ar[dl]^{h_2} \ar[ul]_{\wt{h}_2} \\
		Z_1 \ar[r]^(0.6){h_1} & Y & }
	\]
	Note that $\psi(\Ex(\psi)) \subset \ol{h}_1(\Ex(\ol{h_1}))$ is disjoint form $\Ex(\tau_2)$. Hence, by Lemma \ref{lem_change_if_disjoint} below, there exists a decomposition $\tau_2 \circ \psi \colon W \xrightarrow{\ol{\tau}_2} \ol{Y}_2 \xrightarrow{\ol{\psi}} Y$ with $\Irr(\psi) = \Irr(\ol{\psi})$ and $\Irr(\tau_2)  = \Irr(\ol{\tau}_2)$. In particular, we have $\bB_{\ol{\psi}}= \tau_2 \bB_{\psi}$. It follows that $h_2 = \tau_2 \circ \psi \circ \wt{h}_2 = \ol{\psi} \circ \ol{\tau}_2 \circ \wt{h}_2$ factors via the blow up of $Y$ along $\bB_{\ol{\psi}}$. On the other hand, since $\tau_2$ is an isomorphism in the neighbourhood of $\ol{h}_1(\Ex(\ol{h}_1))$, the fact that $\ol{h}_1$ factors via the blow of $Y_2$ along $\bB_{\psi}$ implies that $h_1$ factors via the blow-up of $Y$ along $\tau_2\bB_{\psi}$. The uniqueness of $\bB_f$ (see Proposition \ref{prop_Danilov}) implies that $\tau_2\bB_{\psi} (= \bB_{\ol{\psi}})$ is a common component of $\bB_{h_1}$ and $\bB_{h_2}$ which contradicts our assumption.
	
	Thus $\bB_{\ol{h}_1} \cap \bB_{h'_2} = \emptyset$ and inductive hypothesis for morphism $\rho_2 \circ \wt{f}$ imply that $\ol{h}_1(\Ex(\ol{h}_1)) \cap h'_2(\Ex(h'_2))= \emptyset$. Then $h_1(\Ex(h_1))  = \tau_2 \ol{h}_1(\Ex(\ol{h}_1))$ is disjoint from $\tau_2 h'_2(\Ex(h'_2))$. It finishes the proof because $h_2(\Ex(h_2)) =\tau_2(\Ex(\tau_2)) \cup \tau_2 h'_2(\Ex(h'_2))$ and we have already proven that $h_1\Ex(h_1) \cap \tau_2(\Ex(\tau_2)) = \emptyset$. 
\end{proof}

\begin{LEM}\label{lem_change_if_disjoint}
	Let $(g, h)$ be an element of $\Dec(f)$ such that $\Ex(g)\cap g^{-1}(\Ex(h)) = \emptyset$. Then $\Ex(g)$ and $g^{-1}(\Ex(h))$ can be contracted in the opposite order, i.e. there exists an element $(\ol{h}, \ol{g}) \in \Dec(f)$ such that $\Irr(g)= \Irr(\ol{g})$ and $\Irr(h) = \Irr(\ol{h})$.
\end{LEM}
\begin{proof}
	Since $g^{-1}(Ex(h))\cap Ex(g)=\emptyset $, we have an open covering $X=U_1\cup U_2$, where $U_1 = X\setminus {\rm Ex} (g)$ and
	$U_2 = X\setminus g^{-1}({\rm Ex}(h))$. Define $\ol{Z}$ by an open covering $\ol{Z} = \ol{U}_1\cup U_2$, where $\ol{U}_1 = Y\setminus f({\rm Ex} (g))$.
	Here $\ol{U}_1$ and $U_2$ are patched together via open subset $U_{12}=U_1\cap U_2$, which is embedded into $\ol{U}_1$ by means of the identification of $U_{12}$ with $f(U_{12})\subset \ol{U}_1$. The morphism $\ol{g} \colon X \to \ol{Z}$ is defined by $h$ on $U_1\simeq g(U_1 )$ and by identity on $U_2$,
	while $\ol{h} \colon \ol{Z} \to Y$ is the embedding on $\ol{U}_1$ and $f$ on $U_2$.
\end{proof}

Now we aim at constructing unions and intersections in the poset $\Dec(f)$. First, note that for any pair of elements $(g_1,h_1)$, $(g_2,h_2)$ of $\Dec(f)$ there exists a smooth algebraic space $Z_\cup$ and a morphism $\tau \colon Z_\cup \to Y$ such that both $h_1$ and $h_2$ factor via $\tau$, $h_1 = \tau \circ \f_1$, $h_2 = \tau \circ \f_2$ and $\bB_{\f_1} \cap \bB_{\f_2} =\emptyset$. To construct $Z_\cup$ and $\tau$ we iteratively blow-up $Y$ along common components of $\bB_{h_1}$ and $\bB_{h_2}$ until they do not meet. 

Then $\f_1(\Ex(\f_1))$ and $\f_2(\Ex(\f_2))$ do not meet by Proposition \ref{prop_h_1_and_h_2_don_int}. Hence, the fiber product $Z_\cap :=Z_1 \times_{Z_\cup} Z_2$ is smooth. We have a diagram:
\begin{equation}\label{eqtn_diag_cup_and_cap} 
\xymatrix{& X \ar@/_1pc/[ddl]_{g_1} \ar@/^1pc/[ddr]^{g_2} \ar[d]|{\xi } & \\
	& Z_{\cap} \ar[dl]^{\psi_1} \ar[dr]_{\psi_2}\ar[dd]|{\kappa}& \\
	Z_1 \ar[dr]_{\f_1} \ar@/_1pc/[ddr]_{h_1}&& Z_2 \ar[dl]^{\f_2} \ar@/^1pc/[ddl]^{h_2} \\
	& Z_{\cup} \ar[d]|{\tau} &\\
	& Y &}
\end{equation}
\begin{LEM}\label{lem_union_and_int}
	In the poset $\Dec(f)$ element $(\xi, \tau\circ\kappa)$ is the intersection of $(g_1,h_1)$ with $(g_2,h_2)$ while $(\kappa \circ \xi, \tau)$ is their union.
\end{LEM}
\begin{proof}
	Let $(g,h)$ be an element in $\Dec(f)$ which is smaller than $(g_1, h_1)$ and $(g_2,h_2)$. If $g\colon X\to Z$ then $g_1 = \alpha_1\circ g$ and $g_2 = \alpha_2 \circ g$, for some $\alpha_1 \colon Z \to Z_1$, $\alpha_2\colon Z \to Z_2$. As $\f_1 \circ g_1 = \f_2\circ g_2$, it implies that $\Irr(\f_1 \circ \alpha_1) =\Irr(\kappa \circ \xi)\setminus \Irr(g) = \Irr(\f_2 \circ \alpha_2)$, i.e. $\f_1 \circ \alpha_1 = \f_2 \circ \alpha_2$ by Proposition \ref{prop_g_defined_by_Irr_g}. By the universal property of the fiber product, there exists $\alpha \colon Z \to Z_\cap$ such that $\alpha_1 = \psi_1 \circ \alpha$ and $\alpha_2 = \psi_2 \circ \alpha$. Then $\xi = \alpha \circ g$, i.e. $(g,h)\preceq(\xi, \tau \circ \kappa)$, which proves that $(\xi, \tau\circ \kappa) = (g_1,h_1)\cap (g_2,h_2)$.
	
	Let now $(\wt{g}, \wt{h})$ be an element which is bigger than $(g_1,h_1)$ and $(g_2, h_2)$. If $\wt{g} \colon X \to \wt{Z}$, then there exists maps $\rho_1\colon Z_1 \to \wt{Z}$ and $\rho_2 \colon Z_2 \to \wt{Z}$ such that $\wt{g} =\rho_1 \circ g_1 = \rho_2 \circ g_2$. Then $\Irr(\rho_1\circ \psi_1) =\Irr(\wt{g}) \setminus \Irr(\xi) =  \Irr(\rho_2\circ \psi_2)$, hence $\rho_1 \circ \psi_1 = \rho_2 \circ \psi_2$ by Proposition \ref{prop_g_defined_by_Irr_g}. 
	
	We have $\f_1(\Ex(\f_1)) \cap \f_2(\Ex(\f_2)) = \emptyset$, i.e. $Z_\cup$ is a union of two open sets $U_1 = Z_\cup\setminus \f_1(\Ex(\f_1))$ and $U_2 = Z_\cup\setminus \f_2(\Ex(\f_2))$. Morphisms $\f_1$, $\f_2$ and $\kappa$ induce isomorphisms $U_1 \simeq Z_1 \setminus \Ex(\f_1)$, $U_2 \simeq Z_2 \setminus \Ex(\f_2)$ and $U_1 \cap U_2 \simeq Z_\cap \setminus \{\Ex(\psi_1)\cup \Ex(\psi_2)\}$. The equality $\rho_1\circ \psi_1= \rho_2 \circ\psi_2$ implies that, analogously as in the proof of Lemma \ref{lem_change_if_disjoint}, one can glue $\rho_1|_{U_1}$ and $\rho_2|_{U_2}$ to get a morphism $\rho \colon Z_\cup \to \wt{Z}$ such that $\rho_1 = \rho \circ \f_1$ and $\rho_2 = \rho \circ \f_2$, which implies that $(\kappa \circ \xi, \tau) = (g_1,h_1)\cup (g_2,h_2)$.
\end{proof}

We put
\begin{align*}
&g_1\cap g_2 :=\xi,& &g_1\cup g_2 := \kappa \circ \xi.&
\end{align*}

\begin{THM}\label{thm_dec_f_poset_of_ideals}
	The sets of irreducible components of the union and intersection of elements $g_1$, $g_2$ of $\Dec(f)$
	satisfy
	\begin{align*}
	&\Irr(g_1\cap g_2) = \Irr(g_1)\cap\Irr(g_2),& &\Irr(g_1\cup g_2) = \Irr(g_1) \cup \Irr(g_2).&
	\end{align*}
	In particular, $\Dec(f)$ is a distributive lattice.
\end{THM}
\begin{proof}
	In the notation of diagram (\ref{eqtn_diag_cup_and_cap}),  $\f_1(\Ex(\f_1)) \cap \f_2(\Ex(\f_2)) = \emptyset$. It follows that $\Irr(\f_1)\cap \Irr(\f_2) = \emptyset$. Then
	\begin{align*}
	&\Irr(g_1) \cup \Irr(g_2) = (\Irr(g_1\cup g_2)\setminus \Irr(\f_1)) \cup (\Irr(g_1\cup g_2)\setminus \Irr(\f_2)) = \Irr(g_1 \cup g_2),&\\
	&\Irr(g_1)\cap \Irr(g_2) = (\Irr(g_1\cup g_2)\setminus \Irr(\f_1))\cap (\Irr(g_1\cup g_2)\setminus \Irr(\f_2))& 
	\\&= \Irr(g_1 \cup g_2) \setminus \Irr(\kappa) = \Irr(g_1 \cap g_2),&
	\end{align*}
	since $\Ex(\kappa)$ for the fiber product $\kappa$ of maps $\f_1$ and $\f_2$ is the union $\Ex(\psi_1)\cup \Ex(\psi_2) = \Ex(\f_1)\cup \Ex(\f_2)$. 
	
	As an element $g$ of $\Dec(f)$ is uniquely determined by components of its exceptional divisor, see Proposition \ref{prop_g_defined_by_Irr_g}, the above proves that unions and intersections are defined via unions and intersections of subsets, hence the distributivity law holds.
\end{proof}

By Birkhoff's theorem, we get
\begin{COR}\label{cor_Dec_id_in_Con}
	Lattice $\Dec(f)$ is isomorphic to the lattice of lower ideals in $\Conn(f)$.
\end{COR} 

\begin{REM}\label{rem_disjoint_div}
	Let $g_1\colon X \to Z_1$, $g_2\colon X\to Z_2$ be elements of $\Dec(f)$ and let $\psi_1$, $\psi_2$ be as in (\ref{eqtn_diag_cup_and_cap}), such that $g_1 = \psi_1 \circ (g_1 \cap g_2)$ and $g_2= \psi_2 \circ (g_1 \cap g_2)$. As $Z_\cap= Z_1 \times_{Z_\cup} Z_2$ and $\f_1(\Ex(\f_1))\cap \f_2(\Ex(\f_2)) = \emptyset$, the exceptional divisors of $\psi_1$ and $\psi_2$ are disjoint.
\end{REM}

\subsection{$\Conn(f)$ and irreducible components of the exceptional divisor}

Define a map 
\begin{align*}
& \alpha \colon \Conn(f) \to \Irr(f),& &\alpha((g,h)) = \Irr(\bB_g).& 
\end{align*}
In view of description (\ref{eqtn_Conn_f_via_B_g}) for $\Conn(f)$, map $\alpha$ is well-defined.

In the opposite direction, define $\beta \colon \Irr(f) \to \Conn(f)$ by:
$$
\beta(E) = \bigcap \{(g,h)\in \Dec(f)\,|\, E\in \Irr(g)\}.
$$ 
\begin{THM}\label{thm_Conn_f_is_comp_of_Ex_f}
	Map $\beta$ is well-defined and inverse to $\alpha$, hence bijection $\Conn(f) \simeq \Irr(f)$.
\end{THM}
\begin{proof}
	Let $\beta(E)=(g_E,h_E)$. Since $g_E\colon X\to Z_E$ factors via the blow-up $\tau$ of $Z_E$ along $\bB_{g_E}$, $g_E = \tau \circ g'_E$ and $g'_E \precneqq g_E$, the definition of $\beta(E)$ implies that $\{E\} = \Irr(\tau) \subset \Irr(\bB_{g_E})$. Moreover, if $\bB_{g_E}$ had more than one component then $g_E$ would factor via the blow-up $\tau'$ of $Z_E$ along a component of $\bB_{g_E}$ different from $g_E(E)$, $g_E = \tau' \circ \wt{g}_E$. As $E\in \Irr(\wt{g}_E)$ and $\wt{g}_E \precneqq \beta(E)$ this contradicts the definition of $\beta(E)$. Thus, $g_E$ is an element of $\Conn(f)$.
	
	The above argument shows also that $E = \Irr(\bB_{g_E})$, hence $\alpha \circ \beta = \Id_{\Irr(f)}$. For the composition $\beta \circ \alpha$, let $g\in \Conn(f)$, $E =\alpha(g)$ and let $g' \in \Dec(f)$ be an element which contracts $E$. If $g\cap g' \neq g$, then there exists a non-trivial $\f$ such that $g = \f \circ (g\cap g')$. By Lemma  \ref{lem_inclusion_of_B_h}, we have $\Irr(\bB_\f)\subset \Irr(\bB_g) = \{E\}$, hence $E \notin \Irr(g\cap g')$. This contradicts equality $\Irr(g\cap g') = \Irr(g) \cap \Irr(g')$, hence $g\cap g' = g$, i.e. $g \preceq g'$. Since $g'$ was arbitrary, it follows that $g = \beta(E) = \beta\alpha(g)$.    
\end{proof}

In view of the theorem, the partial order on $\Conn(f)$ is transported by $\alpha$ to a partial order on $\Irr(f)$. Theorem \ref{thm_dec_f_poset_of_ideals} implies
\begin{COR}\label{cor_union_and_int_in_Dec}
	Map $\Dec(f) \to I(\Irr(f))$, $(g,h) \mapsto \Irr(g)$ induces a bijection of $\Dec(f)$ with the lattice of lower ideals in $\Irr(f)$.
\end{COR}

Note that for $(g,h) \in \Dec(f)$, both embeddings $\Irr(g) \to \Irr(f)$ and $\Irr(h) \to \Irr(f)$ (the latter given by the strict transform along $g$) are morphisms of posets. In particular, for any $(g,h) \in \Dec(f)$, we have an isomorphism of partially ordered sets:
$$
\Irr(f)\setminus\Irr(g) \xrightarrow{\simeq} \Irr(h).
$$

Thus, for any $g \in \Dec(f)$, we have an isomorphism $\gamma_g \colon \Conn(f) \setminus \Conn(g) \to \Conn(h)$ satisfying the equality:
\begin{equation}\label{eqtn_iso_gamma}
g\cup g' = \gamma_g(g') \circ g.
\end{equation}

\section{\textbf{Filtrations and the standard \tr e on the null category}}\label{sec_the_st_tr}

In this section $X$, $Y$ are normal algebraic spaces over an algebraically closed field $k$ of characteristic 0 and $f\colon X\to Y$ is a proper morphism. We denote by $\cC_f$ the \emph{triangulated null category}:
\begin{equation}\label{eqtn_def_C_f}
\cC_f := \{ E \in \dD^b(X)\,|\, Rf_* E = 0\}.
\end{equation}
If $Y$ is smooth, $X$ Gorenstein and $Rf_* \oO_X \simeq \oO_Y$, 
$\cC_f$ is an \emph{admissible} subcategory of $\dD^b(X)$, i.e. the embedding $\iota_{f*} \colon \cC_f \to \dD^b(X)$ has both left and right adjoint functors. Indeed, since $\dD^b(Y) \simeq \Perf(Y)$, functor $Lf^*$ takes $\dD^b(Y)$ to $\dD^b(X)$ and we have a  formula for the functor $f^!$ right adjoint to $Rf_*$:
$$
f^!(-)\simeq Lf^*(-) \otimes^L f^!(\oO_Y),
$$
\emph{cf.} \cite[Lemma A.1]{BodBon}. $X$ is Gorenstein, hence $f^!(\oO_Y)$, which differs from the dualising complex on $X$ by a twist with a line bundle, is a perfect complex in $\dD^b(X)$. It follows that $f^!$ also takes $\dD^b(Y)$ to $\dD^b(X)$. As $Rf_*(\oO_X)\simeq \oO_Y$, functors $Lf^*$, $f^!$ are fully faithful and
\begin{equation}\label{eqtn_ff!=id} 
Rf_* \circ Lf^* \simeq \Id_{\dD^b(Y)} \simeq Rf_* \circ f^!.
\end{equation}
Hence, the cones of the adjunction morphisms lie in $\cC_f$ and adjoint functors to $\iota_{f*}$ can be defined by canonical triangles \cite{B}:
\begin{align}
&Lf^*Rf_* \to \textrm{Id} \to \iota_{f*}\iota_f^* \to Lf^*Rf_*[1],\label{eqtn_def_of_iota*}&\\
&\iota_{f*} \iota_f^! \to \textrm{Id} \to f^!Rf_* \to \iota_{f*}\iota_f^![1]. \label{eqtn_def_of_iota!}&
\end{align}

For a proper morphism $f\colon X \to Y$, we denote by $\omega_f^\bcdot$ the relative dualising complex:
$$
\omega_f^\bcdot = f^!\oO_Y.
$$
\begin{LEM}\label{lem_dual_preserv_C_f}
	Let $f\colon X \to Y$ be a proper morphism such that $Rf_* \oO_X = \oO_Y$. Then the duality functor $\mathscr{D}_f = R\Hom(-, \omega_f^\bcdot)$ preserves the category $\cC_f$.
\end{LEM}
\begin{proof}
	The fact follows immediately from Grothendieck's duality applied to $E \in \cC_f$:
	\begin{equation*} 
	Rf_*R\hH om_X(E, \omega_f^\bcdot) \simeq Rf_* R\hH om_X(E, f^!\oO_Y) \simeq R\hHom_Y(Rf_*E, \oO_Y) = 0.\qedhere 
	\end{equation*}
\end{proof}

We consider the full subcategory of $\cC_f$: 
\begin{equation}\label{eqtn_def_A_f}
\mathscr{A}_f := \{ E\in \Coh(X)\,|\, Rf_*(E) = 0\}.
\end{equation}

\begin{LEM}\cite[Lemma 3.1]{Br1}\label{lem_A_f_is_a_heart}
	Let $f\colon X\to Y$ be a proper morphism with dimension of fibers  bounded by one. An object $E$ in $\dD^b(X)$ lies in $\cC_f$ if and only if $\hH^i(E) \in \mathscr{A}_f$, for all $i \in \mathbb{Z}$. Thus, the restriction of the standard \tr e on $\dD^b(X)$ defines a bounded \tr e on category $\cC_f$ with heart $\mathscr{A}_f$.
\end{LEM}
It follows that under the assumptions of Lemma \ref{lem_A_f_is_a_heart}, category $\mathscr{A}_f$ is abelian. We call it the \emph{(abelian) null category} of $f$. We refer to the \tr e on $\cC_f$ given by Lemma \ref{lem_A_f_is_a_heart} as the \emph{standard} \tr e.
\begin{PROP}\cite[Proposition 2.9]{BodBon} \label{prop_g_is_exact}
	Let $f\colon X\to Y$ be a proper morphism with dimension of fibers bounded by one and let
	\[
	\xymatrix{X \ar[dr]^g \ar[dd]_f & \\ & Z \ar[dl]^h\\ Y&}
	\]
	be a decomposition of $f$. Then, for any $E \in \Coh(X)$ with $R^1f_* E = 0$, we have $R^1g_*E =0$. It follows that functor $g_*$ takes $\mathscr{A}_f$ to $\mathscr{A}_h$ and it is exact.
\end{PROP}

\begin{LEM}\label{lem_g(O)_is_O}
	Let $f\colon X \to Y$ be a proper morphism with dimension of fibers bounded by one such that $Rf_* \oO_X = \oO_Y$. Let $f\colon X\xrightarrow{g} Z \xrightarrow{h} Y$ be a decomposition of $f$ with surjective $g$. If $Z$ is normal, then $Rg_* \oO_X = \oO_Z$ and $Rh_* \oO_Z = \oO_Y$.
\end{LEM}
\begin{proof}
	Since $f_*\oO_X = \oO_Y$, the fibers of $f$ are connected. Hence, so are the fibers of $h$ and $g$. Thus, as $Z$ is normal, $g_*(\oO_X) \simeq \oO_Z$ (cf. \cite[Lemma 4.1]{BodBon}). By Proposition \ref{prop_g_is_exact}, $R^1g_* \oO_X = 0$, i.e. $Rg_* \oO_X = \oO_Z$.
	Then: $Rh_*\oO_Z = Rh_* Rg_* \oO_X = Rf_* \oO_X = \oO_Y$.
\end{proof}

\subsection{$\Dec(f)$ and $\Dec(f)^\textrm{op}$-filtrations on the null-category}\label{ssec_exc_coll}

\begin{PROP}\label{prop_recoll_for_st_t-st_on_C_f}
	Let $f\colon X\to Y$ be a proper morphism with dimension of fibers bounded by one such that $Rf_*(\oO_X) = \oO_Y$. Let $f\colon X\xrightarrow{g} Z \xrightarrow{h} Y$ be a decomposition of $f$ with surjective $g$. Assume that algebraic space $Z$ is smooth and $X$ is Gorenstein. Then we have a recollement
	\[
	\xymatrix{\cC_g \ar[rr]|{\iota_*} && \cC_f \ar@<2ex>[ll]|{\iota^!} \ar@<-2ex>[ll]|{\iota^*}  \ar[rr]|{Rg_*} && \cC_h \ar@<-2ex>[ll]|{Lg^*} \ar@<2ex>[ll]|{g^!}}
	\]
\end{PROP}

\begin{proof}
	Since $Z$ is smooth, $Lg^* \dD^b(Z) \subset \dD^b(X)$ and $g^!(-)\simeq Lg^*(-)\otimes g^!(\oO_Z)$. Since $X$ is Gorenstein, $g^!$ is a functor $\dD^b(Z) \to \dD^b(X)$. By Lemma \ref{lem_g(O)_is_O}, $Rg_*\oO_X \simeq \oO_Z$. It follows by (\ref{eqtn_ff!=id}) that functors $Lg^*$, $g^!$ are fully faithful, and so are their restrictions to $\cC_h$. Furthermore, $Rf_*Lg^*\cC_h \simeq Rh_*\cC_h \simeq 0$ and $Rf_*g^!\cC_h \simeq Rh_* \cC_h \simeq 0$, hence $Lg^*|_{\cC_h}$ and $g^!|_{\cC_h}$ take values in $\cC_f$. Triangles (\ref{eqtn_def_of_iota*}) and (\ref{eqtn_def_of_iota!}) prove existence of left and right adjoint functors to $\iota_*$. Finally, $\cC_g$ is the kernel of $Rg_*|_{\cC_f}$, hence the recollement.
\end{proof}

Let now $f\colon X \to Y$ be a relatively projective birational morphism of smooth algebraic spaces with dimension of fibers bounded by one. By Theorem \ref{thm_dec_f_poset_of_ideals}, we have a distributive lattice $\Dec(f)$ of smooth decompositions for $f$.

\begin{PROP}\label{prop_Conn(f)-filtration}
	Map $\textrm{Dec}(f) \to \textrm{Adm}(\cC_f)$, $(g,h) \mapsto \cC_g$, endows $\cC_f$ with a strict admissible  $\Dec(f)$-filtration.
\end{PROP}
\begin{proof}
	For $g \in \textrm{Dec}(f)$, category $\cC_g$ is an admissible subcategory of $\cC_f$. By Corollary \ref{cor_strict_left_is_admissible}, in order to check that $g\mapsto \cC_g$ is a strict admissible filtration it suffices to show that it is a strict left admissible $\Dec(f)$-filtration.
	
	We shall constantly use the notation of diagram (\ref{eqtn_diag_cup_and_cap}) with $g_1 \cap g_2 = \xi$.
	
	The minimal element in $\Dec(f)$ is $(\Id_X, f)$, the maximal one is $(f, \Id_Y)$. Condition (Li) of Section \ref{ssec_strict_poset_filt} is obviously satisfied. For a pair $g_1, g_2 \in \textrm{Dec}(f)$, clearly $\cC_{g_1\cap g_2} \subset \cC_{g_1} \cap \cC_{g_2}$. Now let $F\in \cC_{g_1} \cap \cC_{g_2}$. Then $F':=R(g_1\cap g_2)_* F$ lies in $\cC_{\psi_1}\cap \cC_{\psi_2}$. In particular, $F'$ is supported on $\Ex(\psi_1)\cap \Ex(\psi_2)$. It follows from Remark \ref{rem_disjoint_div} that $F'=0$, i.e. $F\in \cC_{g_1\cap g_2} $.
	
	For a decomposition $f=h\circ g$, the left orthogonal to $\cC_g$ in $\cC_f$ is $Lg^* \cC_h$. Thus, in order to check condition (Lii), it suffices to show that $L(g_1\cup g_2)^* \cC_{\tau} = Lg_1^* \cC_{h_1} \cap Lg_2^* \cC_{h_2}$, where $f= \tau \circ (g_1 \cup g_2)$. Since $L\f_1^* \cC_{\tau} \subset \cC_{h_1}$ and $L\f_2^* \cC_{\tau} \subset \cC_{h_2}$ , category $L(g_1\cup g_2)^* \cC_{\tau}$ is contained both in $Lg_1^* \cC_{h_1}$ and in $Lg_2^* \cC_{h_2}$. Let now $F$ be an object in $Lg_1^* \cC_{h_1} \cap Lg_2^* \cC_{h_2}$. Then $F\simeq L(g_1\cap g_2)^*R(g_1\cap g_2)_* F$, hence it is enough to show that $F':= R(g_1 \cap g_2)_* F$ is an object of $L\kappa^* \cC_{\tau}$. Since $F'$ lies in $L\psi_1^* \cC_{h_1} \cap L\psi_2^* \cC_{h_2}$, $F'\simeq L\psi_1^*R\psi_{1*}F' \simeq L\psi_2^*R\psi_{2*} L\psi_1^* R\psi_{1*} F'$. By Lemma \ref{lem_base_change_iso} below, the latter is isomorphic to $L\psi_2^*L\f_2^* R\f_{1*}R\psi_{1*}F' \simeq L\kappa^* R\kappa_* F'$, i.e. $F\in L\kappa^* \cC_{\tau}$.
		
	It remains to check that the filtration is strict. For $g_1, g_2$ in $\Dec(f)$, the quotient $\cC_{g_1 \cup g_2}/\cC_{g_1 \cap g_2}$ is equivalent to $\cC_{\kappa}$. Thus, the filtration is strict, if $\cC_{\psi_1}$ and $\cC_{\psi_2}$ are orthogonal subcategories of $\cC_{\kappa}$. By Remark \ref{rem_disjoint_div}, the support of every object in $\cC_{\psi_1}$ is disjoint from the support of every object in $\cC_{\psi_2}$. Hence, $\Hom(\cC_{\psi_1}, \cC_{\psi_2}) = 0 = \Hom(\cC_{\psi_2}, \cC_{\psi_1})$. 
\end{proof}

\begin{LEM}\label{lem_base_change_iso}
	In the notation of (\ref{eqtn_diag_cup_and_cap}) the base change $L\f_2^* R\f_{1*} \to R\psi_{2*} L\psi_1^*$ is an isomorphism of functors $\dD^b(Z_1) \to \dD^b(Z_2)$.
\end{LEM}
\begin{proof}
	First, note that both functors are zero on $\cC_{\f_1}$. For $L\f_2^* R\f_{1*}$ it is clear. Let now $E$ be an object of $\cC_{\f_1}$. Then $E$ is supported on $\Ex(\f_1)$. Since $\f_1(\Ex(\f_1)) \cap \f_2(\Ex(\f_2)) = \emptyset$, by Proposition \ref{prop_h_1_and_h_2_don_int}, and $Z_\cap$ is the fiber product of $\f_1$ and $\f_2$, we have 
	$\Ex(\psi_2) =\psi_1^{-1}(\Ex(\f_1))$. It follows that object $L\psi_1^*E$ is supported on $\Ex(\psi_2)$. Hence, the support of $R\psi_{2*}L\psi_1^*E$ is contained in $\psi_2(\Ex(\psi_2))$. As $\f_2$ is an isomorphism in a neighbourhood of $\psi_2(\Ex(\psi_2))$, vanishing $R\f_{2*}R\psi_{2*}L\psi_1^*E \simeq R\f_{1*}R\psi_{1*}L\psi_1^*E \simeq 0$ implies $R\psi_{2*}L\psi_1^*E \simeq 0$.
	
	It remains to check that the base change is an isomorphism on the orthogonal complement to $\cC_{\f_1} \subset \dD^b(Z_1)$, i.e. on $L\f_1^* \dD^b(Z_\cup)$, which is clear.
\end{proof}

\begin{PROP}\label{prop_right_dual_filt_on_C_f}
	Maps $\Dec(f)^{\textrm{op}} \to \textrm{Adm}(\cC_f)$, $(g,h)\mapsto Lg^* \cC_h$, $(g,h)\mapsto g^!\cC_h$ endow $\cC_f$ with strict admissible  $\Dec(f)^{\textrm{op}}$-filtrations which are right, respectively left, dual to the filtration of Proposition \ref{prop_Conn(f)-filtration}.
\end{PROP}
\begin{proof}
	As $\cC_g \otimes \omega_g^{-1}$ is the kernel of the functor $Rg_*(-\otimes \omega_g)$ left adjoint to $Lg^*$, category $\cC_f$ admits an SOD $\cC_f = \langle Lg^*\cC_h, \cC_g\otimes \omega_g^{-1}\rangle$, i.e. subcategory $Lg^*\cC_h \subset \cC_f$ is admissible.
	
	Since the $g \mapsto \cC_g$ filtration is strict admissible by Proposition \ref{prop_Conn(f)-filtration}, the filtration $g \mapsto Lg^*\cC_h$ is strict and left admissible (see Proposition \ref{prop_right_dual_filtr}). By Corollary \ref{cor_strict_left_is_admissible}, the map $g \mapsto Lg^* \cC_h$ yields a strict admissible $\Dec(f)^{\textrm{op}}$-filtration of $\cC_f$.
	
	The statement for the $g\mapsto g^!\cC_h$ filtration follows from the SOD $\cC_f = \langle\cC_g \otimes \omega_g, g^!\cC_h \rangle$.
\end{proof}

\begin{REM}\label{rem_filtr_on_D_X}
	One can extend lattice $\Dec(f)$ to a lattice $\Dec^+(f)$ by adding a new maximal element $1^+$. The $\Dec(f)$-filtration on $\cC_f$ described in Proposition \ref{prop_Conn(f)-filtration} can be extended to a strict admissible $\Dec^+(f)$-filtration on $\dD^b(X)$ by assigning the category $\dD^b(X)$ to $1^+$. Its right and left dual $\Dec^+(f)^\textrm{op}$-filtrations map $1^+$ to 0 and $g\colon X\to Z$ in $\Dec(f)$ to $Lg^*\dD^b(Z) \subset \dD^b(X)$, respectively $g^! \dD^b(Z)\subset \dD^b(X)$. 
\end{REM}
\subsection{The standard \tr e on the null-category is glued}

By Lemma \ref{lem_A_f_is_a_heart}, the standard \tr e on $\dD^b(X)$ restricts to the standard \tr e on $\cC_f$ with heart $\mathscr{A}_f$. We shall show that the \tr e on $\cC_f$ is glued via the $\Dec(f)$-filtration of Proposition \ref{prop_Conn(f)-filtration}. First, we show that there is a recollement of the hearts of standard \tr es under weaker assumptions on morphism $f$ and its decomposition.

\begin{LEM}\label{lem_g!_in_A_f}
	Let $f\colon X \to Y$ be a proper morphism with dimension of fibers bounded by 1 and $Rf_* \oO_X = \oO_Y$. Let $f\colon X\xrightarrow{g} Z \xrightarrow{h} Y$ be a decomposition of $f$ with surjective $g$ and normal $Z$. Then, for any $E\in \cC_h$ and any $k \in \mathbb{Z}$, sheaves $L^k g^*E$, $\hH^k g^! E$ lie in $\mathscr{A}_f$.
\end{LEM}
\begin{proof} 
	The decomposition $Rf_* = Rh_*\circ Rg_*$ yields
	$Rf_* Lg^*(E) = Rh_* Rg_* Lg^*(E) = Rh_* (E) = 0$ and $Rf_* g^!(E) = Rh_* Rg_* g^!(E) = Rh_* (E) = 0$, by Lemma \ref{lem_g(O)_is_O} and (\ref{eqtn_ff!=id}). Hence, both $L^kg^* E$ and $\hH^k g^!\, E$ are objects of $\aA_f$, see Lemma \ref{lem_A_f_is_a_heart}.
\end{proof}

We say that a diagram of functors between abelian categories
\[
\xymatrix{\aA_0 \ar[r]|{i_*} & \aA \ar[r]|{j^*} \ar@<2ex>[l]|{i^!} \ar@<-2ex>[l]|{i^*}& \aA_1 \ar@<2ex>[l]|{j_*} \ar@<-2ex>[l]|{j_!}}
\]
is an \emph{abelian recollement} (cf. \cite{PsaVit}) if conditions (r1), (r2) and (r3) as in Subsection \ref{ssec_strict_poset_filt} are satisfied. Note that hearts of glued \tr es for a triangulated recollement (\ref{eqtn_recollement}) comprise an abelian recollement. However, not all abelian recollements come from triangulated ones.

\begin{PROP}\label{prop_abelian_recollement}
	Let $f\colon X \xrightarrow{g} Z \xrightarrow{h} Y$ be as in Lemma \ref{lem_g!_in_A_f}. Categories $\mathscr{A}_g$, $\mathscr{A}_f$ and $\mathscr{A}_h$ fit into an abelian recollement
	\begin{equation}\label{eqtn_abel_recoll}
	\xymatrix{\mathscr{A}_g \ar[rr]|{i_*}& & \mathscr{A}_f \ar[rr]|{g_*} \ar@<2ex>[ll]|{i^!} \ar@<-2ex>[ll]|{i^*} && \mathscr{A}_h \ar@<2ex>[ll]|{\hH^0g^!} \ar@<-2ex>[ll]|{g^*}}
	\end{equation}
	where $i^*:= \hH^0 \iota_{g}^*|_{\mathscr{A}_f}$ and $i^!:= \hH^0 \iota_g^!|_{\mathscr{A}_f}$, for functors $\iota_g^* \colon \dD^-(X) \to \cC_g^-$, $\iota_g^! \colon \dD^+(X) \to \cC_g^+$ defined as in (\ref{eqtn_def_of_iota*}), (\ref{eqtn_def_of_iota!}).
\end{PROP}

\begin{proof}
	By construction, functors $(i^* \dashv i_* \dashv i^!)$ and $(g^*\dashv g_* \dashv \hH^0 g^!)$ are adjoint. $\mathscr{A}_g\subset \mathscr{A}_f$ and, by Proposition \ref{prop_g_is_exact}, $g_*|_{\mathscr{A}_f} = Rg_*|_{\mathscr{A}_f}$, hence  $\mathscr{A}_g$ is the kernel of $g_*$. By Lemma \ref{lem_g(O)_is_O}, $Rg_* \oO_X \simeq \oO_Z$, i.e. formula (\ref{eqtn_ff!=id}) yields $Rg_*Lg^* \simeq \Id_{\dD^b(Z)} \simeq Rg_*g^!$. By Lemma \ref{lem_g!_in_A_f}, both $g^*$ and $\hH^0 g^!$ map $\mathscr{A}_h$ to $\mathscr{A}_f$. 
	Since $Rg_*$ is $t$-exact on $\cC_f$ by Proposition \ref{prop_g_is_exact}, we have $g_* g^* \simeq \textrm{Id}_{\aA_h} \simeq g_* \hH^0g^!$, i.e. functors $g^*$ and $\hH^0 g^!$ are fully faithful.
\end{proof}

If we further require that $Z$ is smooth and $X$ is Gorenstein then the standard \tr e on $\cC_f$ is glued:

\begin{LEM}\label{lem_st_is_glued}
	Let $f\colon X \xrightarrow{g} Z \xrightarrow{h}Y$ be as in Lemma \ref{lem_g!_in_A_f}. Assume that algebraic space $Z$ is smooth and $X$ is Gorenstein. Then the standard \tr e on $\cC_f$ is glued from the standard \tr es on $\cC_g$ and $\cC_h$ via recollement of Proposition \ref{prop_recoll_for_st_t-st_on_C_f}.
\end{LEM}
\begin{proof}
	Since the \tr es on $\cC_g$ and $\cC_f$ are obtained by restriction from $\dD^b(X)$, the embedding $\iota_*$ is $t$-exact. The $t$-exactness of $Rg_*$ follows from Proposition \ref{prop_g_is_exact}. 
	Hence, the standard \tr e is glued.
\end{proof}

Let $f\colon X\to Y$ be a relatively projective birational morphism of smooth algebraic spaces with dimension of fibers bounded by one. Morphism $g\in \Conn(f)$ admits a decomposition $g\colon X\xrightarrow{g'}Z' \xrightarrow{s} Z$, for the blow-up $s$ of $Z$ along $\bB_g$. Since $\Irr(g')=\Irr(g)\setminus \Irr(\bB_g)$, element $g'$ is maximal among all elements in $\Dec(f)$ strictly smaller than $g$. It follows that $\cC_{g'}$ is the category $\cC_{<\{g\}}$ as in Section \ref{ssec_gluin_of_tr} and 
$$
\cC_g^o:=\cC_g/\cC_{<\{g\}}
$$
is equivalent to $\cC_{s}$, see Proposition \ref{prop_recoll_for_st_t-st_on_C_f}. In particular, $\cC_g^o$ is endowed with the \tr e transported from the standard \tr e on $\cC_{s}$. 

Morphism $s$ fits into a diagram:
\begin{equation}\label{eqtn_diagram_single_blow-up}
\xymatrix{E_g \ar[r]^{\zeta} \ar[d]_{p} & Z' \ar[d]^{s} \\ \bB_g \ar[r]^{j} & Z.}
\end{equation}
Then $E_g \simeq \mathbb{P}(N_{Z/\bB_g})$ is the projectivisation of the normal bundle and $p$ is the projection. Define
\begin{equation}\label{eqtn_theta} 
\theta_g(-):=R\zeta_{*}Lp^*(-) \otimes \oO_{Z'}(E_g) \colon \dD^b(\bB_g) \to \dD^b(Z').
\end{equation} 
By \cite[Theorem 4.3]{Orl}, $\theta_g$ induces an equivalence $\dD^b(\bB_g) \xrightarrow{\simeq} \cC_{s}$.
\begin{LEM}\label{lem_stn_t-str_on_W_and_C_f}
	Functor $\theta_g$, (\ref{eqtn_theta}), is $t$-exact when both $\dD^b(\bB_g)$ and $\cC_{s}$ are endowed with the standard \tr es.
\end{LEM}
\begin{proof}
	We use the notation of (\ref{eqtn_diagram_single_blow-up}). Functor $\theta_g$ is the composite of $Lp^*$, for $p$ flat, $R\zeta_*$, for a closed embedding $\zeta$, and tensor product with a line bundle, hence it is $t$-exact.
\end{proof}

It follows that category $\cC_g^o$, associated to a join-prime element of $\Dec(f)$, i.e. $g\in \Conn(f)$, is equivalent to $\dD^b(\bB_g)$.

\begin{PROP}
	The standard \tr e on $\cC_f$ is glued via the strict $\Dec(f)$-filtration of Proposition \ref{prop_Conn(f)-filtration} from the standard \tr es on $\dD^b(\bB_g)$, for all $g\in \Conn(f)$.
\end{PROP}
\begin{proof}
	By Lemma \ref{lem_stn_t-str_on_W_and_C_f}, the standard \tr e on $\dD^b(\bB_g)$ corresponds under $\theta_g$ to the standard \tr e on $\cC_{s}$. 
	By Lemma \ref{lem_st_is_glued}, for any $g_1\prec g_2\in \Dec(f)$ with $g_2=\f \circ g_1$ and a decomposition $\f = \f_1\circ \f_2$, the \tr e on $\cC_\f$ is glued via the recollement with respect to subcategory $\cC_{\f_2}$ from the standard \tr es on $\cC_{\f_2}$ and $\cC_{\f_1}$. Hence, conditions on the glued \tr e of Theorem \ref{thm_gluing_via_poset} are satisfied.
\end{proof}

\section{\textbf{Tilting relative generators}}\label{sec_local_gen}

We shall describe a \tr e on $\cC_f$ glued from the standard \tr es on $\dD^b(\bB_g)$, for $g\in \Conn(f)$, via the left dual $\Dec(f)^{\textrm{op}}$-filtration of Proposition \ref{prop_right_dual_filt_on_C_f}. We shall show that this \tr e is induced by a tilting relative  generator over $Y$. First we discuss this notion and the resulted \tr e.

\subsection{Generators in stacks of subcategories in $\dD^b(X)$}

For a quasi-compact, quasi-separated algebraic space, we denote by $\dD(X)$ the unbounded derived category of $\oO_X$-modules with quasi-coherent cohomology. We consider algebraic spaces with their etale topology. By abuse of terminology we will speak about open subsets meaning open sets in the etale topology.

Let $f\colon X \to Y$ be a proper morphism of quasi-compact, quasi-separated algebraic spaces. A \emph{stack $\mathfrak{C}$ of subcategories of} $\dD(X)$ over $Y$ is the data of a full triangulated subcategory $\cC_U \subset \dD(f^{-1}(U))$, for any open $U\subset Y$, satisfying conditions:
\begin{itemize}
	\item[(S1)] for any open embedding $U_1\subset U$ and $E \in \cC_U$, the restriction $E|_{f^{-1}(U_1)}$ lies in $\cC_{U_1}$,
	\item[(S2)] for any open $U\subset Y$, any $\eE\in \Perf(U)$ and $E \in \cC_U$, object $Lf^*\eE\otimes E$ is in $\cC_U$,
	\item[(S3)] for any open $U\subset Y$, an object $E\in \dD(f^{-1}(U))$ is in $\cC_U$ if there exists an open covering $U = \bigcup_{i\in I} U_i$ such that $E|_{f^{-1}(U_i)} \in \cC_{U_i}$, for all $i\in I$. 
\end{itemize} 

We say that $T \in \cC_Y$ is \emph{tilting over $Y$} if $\mathscr{A}_T:= Rf_* R\hHom_X(T,T)$ is isomorphic to its zero'th cohomology sheaf $A_T :=\hH^0(X, \mathscr{A}_T)$. We say that $T$ is a \emph{relative generator for $\mathfrak{C}$ over $Y$} if, for any affine open $U \subset Y$, category $\cC_U$ is \emph{split generated} by $T|_{f^{-1}(U)}$, i.e. equivalent to the smallest triangulated subcategory of $\dD(f^{-1}(U))$ containing $T|_{f^{-1}(U)}$ and closed under direct summands.

We shall replace $\mathscr{A}_T$ with a suitable DG algebra and prove
\begin{THM}\label{thm_relative_tilting}
	Let $\mathfrak{C}$ be a stack of subcategories of $\dD(X)$ over $Y$ and $T \in \cC_Y$ a relative generator for $\mathfrak{C}$ over $Y$. Then $\mathfrak{C}$ is equivalent to the stack of perfect $\mathscr{A}_T$-modules, i.e. for any open $U\subset Y$, category  $\cC_U$ is equivalent to the full subcategory $\Perf(\mathscr{A}_T|_{U})$ of perfect objects in the derived category of sheaves of right $\mathscr{A}_T|_U$-modules.
\end{THM}
To simplify the exposition, we shall assume that the open $U\subset Y$ is equal to $Y$.

We shall use an appropriate DG enhancement \cite{BK} for $\cC_Y$ to define $\mathscr{A}_T$, functor $\Phi_T \colon \cC_Y \to \Perf(\mathscr{A}_T)$ and its left adjoint. We then check that $\Phi_T$ is fully faithful and essentially surjective. When $T$ is tilting over $Y$, we give sufficient conditions for category $\Perf(\mathscr{A}_T)$ to be equivalent to the bounded derived category $\dD^b(A_T)$ of coherent $A_T$-modules.

Let $\wt{\cC}_Y$ be the DG category of complexes of $h$-injective $\oO_X$-modules quasi-isomorphic to objects in $\cC_Y$. Let $\wt{T}$ be an $h$-injective resolution for $T$. We define a sheaf
$$
\mathscr{A}_T:=f_*\hHom_X(\wt{T}, \wt{T})
$$
of DG algebras on $Y$. Denote by $\textrm{Mod}(\mathscr{A}_T)$ the DG category of sheaves of right $\mathscr{A}_T$--DG modules and by $\pPerf(\mathscr{A}_T)$ its full DG subcategory of perfect complexes, i.e. DG modules which locally are quasi-isomorphic to objects in the full subcategory of $\textrm{Mod}(\mathscr{A}_T)$ generated by $\mathscr{A}_T$.
Denote by $\dD(\mathscr{A}_T)$ the derived category of $\textrm{Mod}(\mathscr{A}_T)$ and by $\Perf(\mathscr{A}_T)\subset \dD(\mathscr{A}_T)$ the full subcategory of perfect complexes.

Further, denote by 
\begin{equation}\label{eqtn_functor_Phi_T}
	\Phi_T \colon \cC_Y \to \Perf(\mathscr{A}_T)
\end{equation}
the functor induced by
$$
\wt{\Phi}_T(-):=f_*\hHom_X(\wt{T},-) \colon \wt{\cC}_Y \to \pPerf(\mathscr{A}_T).
$$
Since, for any affine $U \subset Y$, category $\cC_U$ is generated by $T|_{f^{-1}(U)}$ and $\Phi_T(T) \simeq \mathscr{A}_T$, $f_*\hHom(\wt{T},\wt{E})$ is indeed an object of $\pPerf(\mathscr{A}_T)$, for any $\wt{E} \in \wt{\cC}_Y$.

To define the functor left adjoint to $\Phi_T$, we recall that any sheaf $M$ of $\mathscr{A}_T$--DG modules admits an $h$-flat resolution $\ol{M}$, \cite[Theorem 1.3.3]{Riche}. We define $\Psi_T\colon \Perf(\mathscr{A}_T) \to \cC_Y$ as the functor induced by 
$$
\wt{\Psi}_T\colon \pPerf_h(\mathscr{A}_T) \to \wt{\cC}_Y,\quad \Psi_T(\ol{M}) = f^{-1}(\ol{M}) \otimes_{f^{-1}(\mathscr{A}_T)} \wt{T},
$$
where $\pPerf_h(\mathscr{A}_T)$ is the DG category of $h$-flat $\mathscr{A}_T$--DG modules quasi-isomorphic to objects of $\pPerf(\mathscr{A}_T)$. Above, the $f^{-1}(\mathscr{A}_T)$-module structure on $\wt{T}$ is given by the composite
$$
f^{-1}f_*\hHom_X(\wt{T},\wt{T}) \otimes \wt{T} \to \hHom_X(\wt{T},\wt{T}) \otimes \wt{T} \to \wt{T}
$$
of the adjunction counit and the evaluation map.

On the category generated by $\mathscr{A}_T$, functor 
$\wt{\Psi}_T$ is uniquely defined by equality $\wt{\Psi}_T(\mathscr{A}_T) = \wt{T}$.
Since any $M \in \Perf(\mathscr{A}_T)$ is locally an object of the category generated by $\mathscr{A}_T$, $\Psi_T(M)$ is locally an object of $\mathfrak{C}$. Then (S3) implies that $\Psi_T(M) \in \cC_Y$, i.e. $\Psi_T$ is well-defined.

For $h$-flat $\ol{M}$ and $\wt{E} \in \wt{\cC}_Y$, canonical morphism
$\hHom_X(f^{-1}(\ol{M}) \otimes_{f^{-1}(\mathscr{A}_T)} \wt{T}, \wt{E}) \to \hHom_{f^{-1}(\mathscr{A}_T)}(f^{-1}(\ol{M}), \hHom_X(\wt{T}, \wt{E}))$ gives $f_*\hHom_X(f^{-1}(\ol{M}) \otimes_{f^{-1}(\mathscr{A}_T)} \wt{T}, \wt{E}) \to f_*\hHom_{f^{-1}(\mathscr{A}_T)}(f^{-1}(\ol{M}), \hHom_X(\wt{T}, \wt{E}))$. By adjunction, we have 
$$
f_*\hHom_{f^{-1}(\mathscr{A}_T)}(f^{-1}(\ol{M}), \hHom_X(\wt{T}, \wt{E})) \simeq \hHom_{\mathscr{A}_T}(\ol{M}, f_* \hHom_X(\wt{T}, \wt{E})).
$$
Thus, for any $M \in \Perf(\mathscr{A}_T)$, $E\in \cC_Y$, we get 
$$
\kappa_{M,E} \colon Rf_* R\hHom_X(\Psi_T(M), E) \to R\hHom_{\mathscr{A}_T}(M, \Phi_T(E)).
$$ 

Locally $\Psi_T$  is inverse to $\Phi_T$, in particular left adjoint. It follows that $\kappa_{M,E}$ is a quasi-isomorphism, for any $M \in \Perf(\mathscr{A}_T)$, $E \in \cC_Y$. As global $\Hom$ is the hypercohomology group of $Rf_*R\hHom$, this implies that functor $\Psi_T$ is left adjoint to $\Phi_T$.

\begin{LEM}\label{lem_Phi_T_ff}
	Functor $\Phi_T$ is fully faithful.
\end{LEM}
\begin{proof}
	For $\wt{E}_1, \wt{E}_2\in \wt{\cC}_Y$, we aim to construct a morphism $\alpha_{\wt{E}_1, \wt{E}_2} \colon f_* \hHom_X(\wt{E}_1, \wt{E}_2) \to \hHom_{\mathscr{A}_T}(\wt{\Phi}_T(\wt{E}_1), \wt{\Phi}_T(\wt{E}_2))$ and show that it is a quasi-isomorphism.
	
	Composition of morphisms gives, for any $\wt{E}_1, \wt{E}_2\in \wt{\cC}_Y$, a map $\alpha'_{\wt{E}_1, \wt{E}_2} \colon \hHom_X(\wt{E}_1, \wt{E}_2) \to \hHom_X(\hHom_X(\wt{T}, \wt{E}_1), \hHom_X(\wt{T}, \wt{E}_2))$. To get $\alpha_{\wt{E}_1, \wt{E}_2}$, we note that, for any open subset $U\subset Y$, a section of $f_* \hHom_X(\hHom_X(\wt{T}, \wt{E}_1), \hHom_X(\wt{T}, \wt{E}_2))(U)$ is a morphism $\f \colon \hHom_X(\wt{T},\wt{E}_1)|_{f^{-1}(U)} \to \hHom_X(\wt{T}, \wt{E}_2)|_{f^{-1}(U)}$. Then $f_*\f$ gives $\wt{\Phi}_T(\wt{E}_1)|_{U} \to \wt{\Phi}_T(\wt{E}_2)|_U$. As $f_* \f$ respects the right $\mathscr{A}_T|_U$-module structure, we get a map $f_* \hHom_X(\hHom_X(\wt{T}, \wt{E}_1), \hHom_X(\wt{T}, \wt{E}_2)) \to \hHom_{\mathscr{A}_T}(\wt{\Phi}_T(\wt{E}_1), \wt{\Phi}_T(\wt{E}_2))$. Composite with $f_* \alpha'$ yields the desired $\alpha_{\wt{E}_1, \wt{E}_2}$.
	
	By considering $h$-injective resolutions, we get $\alpha_{E_1,E_2} \colon Rf_*R\hHom_X(E_1,E_2) \to R\hHom_{\mathscr{A}_T}(\Phi_T(E_1), \Phi_T(E_2))$, for any pair $E_1, E_2 \in \cC_Y$. Clearly, $\alpha_{T,T}$ is a quasi-isomorphism. Since, for any open affine $U\subset Y$, objects $E_1|_{f^{-1}(U)}$, $E_2|_{f^{-1}(U)}$ lie in the category generated by $T|_{f^{-1}(U)}$, we get by unwinding that $\alpha_{E_1,E_2}$ is locally a quasi-isomorphism. By taking global sections, we conclude that  $\Hom_X(E_1,E_2) = \Hom_{\mathscr{A}_T}(\Phi_T(E_1), \Phi_T(E_2))$.
\end{proof}

\begin{proof}[Proof of Theorem \ref{thm_relative_tilting}]
	In view of Lemma \ref{lem_Phi_T_ff}, it suffices to show that $\Phi_T$ is essentially surjective. We check that the adjunction unit $\Id \to \Phi_T \Psi_T$ is an isomorphism when applied to any object of $\Perf(\mathscr{A}_T)$. For any $M \in \Perf(\mathscr{A}_T)$, the cone $M'$ of the adjunction unit $M \to \Phi_T\Psi_TM$ satisfies $\Hom_{\aA_T}(M', \Phi_T(E)) =0$, for any $E\in \cC_Y$. We shall show that this implies $M' \simeq 0$.

	Let $\eE$ be an object in $\Perf(Y)$. Then $T\otimes^L_X Lf^* \eE \in \cC_Y$ and the projection formula 
	\begin{align*}
		&Rf_*R\hHom_X(T, T \otimes^L_X Lf^* \eE) \simeq Rf_*(R\hHom_X(T,T)\otimes^L_X Lf^*\eE) \simeq Rf_*R\hHom_X(T, T) \otimes^L_Y \eE
	\end{align*} 
	implies $ \Phi_T(T\otimes^L_X Lf^* \eE)\simeq \mathscr{A}_T\otimes^L_Y \eE$. Hence $\Hom_{\mathscr{A}_T}(M', \mathscr{A}_T\otimes^L_Y \eE)\simeq 0$, for any $\eE \in \Perf(Y)$. 
	
	We have $R\hHom_{\mathscr{A}_T}(M', \mathscr{A}_T \otimes^L_Y \eE) \simeq R\hHom_{\mathscr{A}_T}(M', \mathscr{A}_T) \otimes^L_Y \eE$. It follows that  $R\hHom_{\mathscr{A}_T}(M', \mathscr{A}_T \otimes^L_Y \eE) \simeq R\hHom_Y(\eE^\vee, M'^{\vee})$, where $M'^\vee :=R\hHom_{\mathscr{A}_T}(M', \mathscr{A}_T)$ denotes the dual $\mathscr{A}_T$--DG module. By taking zero'th hypercohomology, we get $\Hom_Y(\eE^\vee, M'^\vee)=0$, for any $\eE\in \Perf(Y)$. As $\dD_{\textrm{qc}}(Y)$ is compactly generated (see \cite{BvdB}, and \cite[Appendix B]{ToeVaq} for algebraic spaces), it follows that $M'^\vee = 0$. Vanishing of $M'$ follows from the fact that for any $M'\in \Perf(\mathscr{A}_T)$, we have $(M'^\vee)^\vee \simeq M'$. 
\end{proof}
 
\begin{PROP}\label{prop_Perf=Db}
	Let $f\colon X\to Y$ be a proper morphism of algebraic spaces over a field with smooth $X$ and let $\mathfrak{C}$ be a stack of subcategories of $\dD^b(X)$ over $Y$ such that, for any open $U\subset Y$, category $\cC_U \subset \dD^b(f^{-1}(U))$ is left (or right) admissible. Let $T\in \cC_Y$ be a tilting relative generator for $\mathfrak{C}$ over $Y$ and $A_T= \hH^0(Rf_* R\hHom_X(T,T))$. Then $\Perf(\mathscr{A}_T) \simeq \dD^b(A_T)$.
\end{PROP}
\begin{proof} 
	In view of \cite[Proposition 1.5.6]{Riche}, there exists an equivalence  $\chi \colon \dD(\mathscr{A}_T) \to \dD(A_T)$. The fact that $\chi$ induces isomorphism $\Perf(\mathscr{A}_T) \simeq \dD^b(A_T)$ is local over $Y$, hence we can assume that $Y = \Spec R$ is affine.

	Since $X$ is smooth, $\dD^b(X)$ admits a strong generator $E$, see \cite[Theorem 3.1.4]{BvdB} for smooth schemes over a field and \cite[Corollary 4.8]{Toen1} for algebraic spaces. If we denote by $\iota^*\colon \dD^b(X) \to \cC_Y$ the functor left adjoint to the embedding $\cC_Y \to \dD^b(X)$, then $\iota^*E$ is a strong generator for $\cC_Y$, and similarly for the right adjoint functor.
	
	Any object $M\in \dD^b(A_T)$ defines a locally finite cohomological functor $$
	\Hom_{A_T}(\chi\Phi_T(-),M)\colon \cC_Y \to \textrm{mod--}R,
	$$
	i.e. for any $E\in \cC_Y$, $\bigoplus_{i}\Hom_{A_T}(\chi\Phi_T(E),M[i])$ is a finitely generated $R$-module. Since $X$ is proper over $\Spec \,R$, $\bigoplus_i \Hom(E_1, E_2[i])$ is a finitely generated $R$-module, for any pair $E_1, E_2$ of objects in $\dD^b(X)$, in particular, in $\cC_Y$. Then \cite[Corollary 4.18]{Rou2} implies that $\Hom_{A_T}(\chi\Phi_T(-),M)\simeq \Hom_X(-,B)$ is representable by some $B\in \cC_Y$. The identity morphism of $B$ gives $\f \colon \chi \Phi_T(B) \to M$. As functor $\chi \Phi_T$ is fully faithful, the induced morphism $\Hom_{A_T}(\chi \Phi_T(E),\chi\Phi_T(B)) \to \Hom_{A_T}(\chi \Phi_T(E),M)$ is an isomorphism, for any $E\in \cC_Y$. It follows that the cone of $\f$ is orthogonal to $\Perf(A_T)$. Since $\dD(A_T)$ is perfectly generated \cite[Theorem 1.5.10]{Lur}, we conclude that the cone is zero and $\f$ is an isomorphism. The inverse inclusion $\chi\Perf(\mathscr{A}_T) = \Perf(A_T) \subset \dD^b( A_T)$ is clear.
\end{proof}

Proposition \ref{prop_Perf=Db} implies that, for a proper morphism $f\colon X \to Y$ of smooth algebraic spaces and a stack $\mathfrak{C}$ of subcategories of $\dD^b(X)$, a tilting relative generator $T\in \cC_Y$ induces a \tr e on $\cC_Y$. In fact, we can consider two \tr es, the $T$-\emph{projective} one where $T$ is a projective generator locally over $Y$ and the $T$-\emph{injective} one where $T$ is an injective generator locally over $Y$. They are defined as:
\begin{equation}\label{eqtn_t-str_on_stack}
\begin{aligned}
\cC^{\lle 0}_{Y\,\textrm{pro }T}&=\{E \in \cC_Y\,|\, Rf_*R\hHom_X(T,E) \in \dD^b(\aA_T)^{\lle 0}\},\\
\cC^{\gge 1}_{Y\,\textrm{pro }T}&=\{E \in \cC_f\,|\, Rf_*R\hHom_X(T, E) \in \dD^b(\aA_T)^{\gge 1}\},\\
\cC^{\lle 0}_{Y\,\textrm{inj }T}&=\{E \in \cC_Y\,|\, Rf_*R\hHom_X(E, T) \in \dD^b(\aA_T^\textrm{op})^{\gge 0}\},\\
\cC^{\gge 1}_{Y\,\textrm{inj }T}&=\{E \in \cC_f\,|\, Rf_*R\hHom_X(E, T) \in \dD^b(\aA_T^\textrm{op})^{\lle -1}\}.
\end{aligned}
\end{equation}

\subsection{A tilting relative  generator for the null category}

Let $X$, $Y$ be smooth algebraic spaces over an algebraically closed field of characteristic zero, $f\colon X \to Y$ a relatively projective birational morphism  with dimension of fibers bounded by one and let $(g,h)$ be an element of $\Dec(f)$. In this section we shall construct a tilting relative  generator for $\cC_g$ over $Y$.
\begin{LEM}
	Let $(g,h)$ be an element of $\Dec(f)$. The assignment to any open $U\subset Y$ the category $\cC_{g|_{f^{-1}(U)}} \subset \dD^b(f^{-1}(U))$ defines a stack $\mathfrak{C}_g$ of subcategories of $\dD^b(X)$ over $Y$. 
\end{LEM}
\begin{proof}
	We verify conditions (S1)-(S3).
	
	Let $U_1 \subset U$ be open subsets of $Y$ and $V_1:=h^{-1}(U_1)$, $V:=h^{-1}(U)$ their preimages in $W$. Since the embedding $j\colon V_1 \to V$ is flat, we have $Lj^*R(g|_{g^{-1}(V)})_* \simeq R(g|_{g^{-1}(V_1)})_*Li^*$, for the embedding $i\colon g^{-1}(V_1) \to g^{-1}(V)$. It follows that, for any $E \in \cC_{g|_{g^{-1}(V)}}$, we have $R(g|_{g^{-1}(V_1)})_* E|_{g^{-1}(V_1)} = 0$, i.e. condition (S1) is satisfied.
	
	Condition (S2) follows immediately from the projection formula. Finally, (S3) holds because vanishing of $Rg_*$ is local in $W$, hence also in $Y$.
\end{proof}

For a sheaf $\fF$ on $X$ and a divisor $D\subset X$, denote by $\textrm{res}^{\fF}_{D}$ the  restriction morphism 
\begin{equation}\label{eqtn_restr_morp}
\textrm{res}^{\fF}_{D} \colon \fF\to \fF\otimes \oO_{D}.
\end{equation}

For a birational morphism $g\colon X \to Z$ with smooth $Z$, denote by $\omega_g =g^!(\oO_Z) = \omega_{X/Z}$ the relative canonical sheaf. We define \emph{discrepancy divisor} $D_g \subset X$ as zeroes of the canonical section $s_g$ of $\omega_g$ (given by the adjunction). We have:
\begin{equation}\label{eqtn_dis_of_f}
0 \to \oO_X \xrightarrow{s_g} g^! \oO_Z \xrightarrow{\psi_g} \omega_g|_{D_g} \to 0
\end{equation}
Since $s_g$ is a section of a line bundle, we have $\psi_g = \textrm{res}^{\omega_g}_{D_g}$. For a decomposition $f
\colon X\xrightarrow{g} Z \xrightarrow{h} Y$, we have $\omega_f = g^* \omega_h \otimes \omega_g$.

We define the \emph{discrepancy sheaf} of $g$ as $\omega_X|_{D_g} \simeq \omega_g|_{D_g} \otimes g^*(\omega_Z)$.

Let $f\colon X \to Y$ be a projective birational morphism of smooth algebraic spaces with dimension of fibers bounded by one. For $g\in \Dec(f)$, consider a decomposition $f\colon X \xrightarrow{g} Z \xrightarrow{h} Y$. Define
\begin{equation}\label{eqtn_def_of_T_ff}
	T_g^f := \bigoplus_{g' \in \Conn(g)} \omega_X|_{D_{g'}}, \quad \ol{T}_g^f := \bigoplus_{g'\in \Conn(f)\setminus \Conn(g)} \omega_X|_{D_{g'}}.
\end{equation}
Then
\begin{equation*}
T_f:=T^f_f = T^f_g \oplus \ol{T}^f_g.
\end{equation*}

We shall show that $T_f\in \cC_f$ is a tilting relative  generator for $\mathfrak{C}_f$ over $Y$ and, more generally, $T_g \in \cC_g$ is a tilting relative  generator for $\mathfrak{C}_g$ over $Y$.

\begin{LEM}\label{lem_discr_in_C_g}
	Object $\omega_X|_{D_g}$ lies in $\cC_g$. 
\end{LEM}
\begin{proof}
	For $g\colon X \to Z$, we have $\omega_X|_{D_g} \simeq Lg^*(\omega_Z) \otimes \omega_g|_{D_g}$. The statement follows by the projection formula and vanishing of $Rg_* \omega_g|_{D_g}$, which follows by applying $Rg_*$ to (\ref{eqtn_dis_of_f}).
\end{proof}

It follows from Lemma \ref{lem_discr_in_C_g} that $T^f_g \in \cC_g$.

Recall that , for any $g\in \Dec(f)$, we have an isomorphism $\gamma_g \colon \Conn(f) \setminus \Conn(g) \to \Conn(h)$ as in (\ref{eqtn_iso_gamma}).
\begin{PROP}\label{prop_rez_of_omega_f_D_g}
	For $g\in \Dec(f)$ and $g_2 \in \Conn(f) \setminus \Conn(g)$, we have  $Rg_* (\omega_X|_{D_{g_2}}) = \omega_Z|_{D_{\gamma_g(g_2)}}$. In particular, $Rg_*(T_f) = T_h$.
\end{PROP}
\begin{proof}
	We use the notation of (\ref{eqtn_diag_cup_and_cap}) with $g=g_1$, $\psi= \psi_1$ and $\f = \f_1$. 
	 
	As $g_2\notin \Conn(g)$, morphism $\psi_2$ is not the identity. Applying $R\xi_*$ to the sequence
	$$
	0 \to g_2^* \omega_{Z_2} \to \omega_X \to \omega_X|_{D_{g_2}} \to 0
	$$
	yields $R\xi_* (\omega_X|_{D_{g_2}}) = \omega_{Z_{\cap}}|_{D_{\psi_2}}$. Thus, $Rg_* (\omega_X|_{D_{g_2}}) \simeq R\psi_* R\xi_*(\omega_X|_{D_{g_2}}) \simeq R\psi_* (\omega_{Z_{\cap}}|_{D_{\psi_2}})$.  Since, by Remark \ref{rem_disjoint_div}, $\Ex(\psi) \cap \Ex(\psi_2) = \emptyset$, $R \psi_* (\omega_{Z_\cap}|_{D_{\psi_2}}) \simeq \omega_{Z}|_{D_{\f}}$.
	 
	Since $\gamma_g$ as in (\ref{eqtn_iso_gamma}) is a bijection and, by Lemma \ref{lem_discr_in_C_g}, $Rg_* (\omega_X|_{D_{g'}}) =0$, for any $g' \in \Conn(g)$, it follows that
	$Rg_* T_f = T_h$.
\end{proof}

By Lemma \ref{lem_discr_in_C_g}, $Rg_* T^f_g =0$. Hence, Proposition  \ref{prop_rez_of_omega_f_D_g} implies that 
$$
Rg_* \ol{T}^f_g = T_h.
$$

\begin{PROP}\label{prop_T_ff_no_higher_ext}
	Object $Rf_* R\hHom_X(T_f, T_f)$ is a pure sheaf on $Y$.
\end{PROP}
\begin{proof}
	Let $(g_1,h_1)$, $(g_2,h_2)$ be elements of $\Conn(f)$. We show that 
	\begin{equation}\label{eqtn_formula_1} 
	Rf_* R\hHom_X(\omega_X|_{D_{g_1}}, \omega_{X}|_{D_{g_2}}) \simeq Rf_* R\hHom_X(\omega_f|_{D_{g_1}}, \omega_f|_{D_{g_2}})
	\end{equation} 
	is a pure sheaf on $Y$.
	
	First, assume that $g_1$, $g_2$ are not comparable. In the notation of (\ref{eqtn_diag_cup_and_cap}) with $g_1 \cap g_2 = \xi$, Lemma \ref{lem_epi_of_discr} below implies an exact sequence
	\begin{equation}\label{eqtn_ses_1}
	0 \to (g_1\cap g_2)^* (\omega_h|_{D_{\psi_1}}) \to \omega_f|_{D_{g_1}} \to \omega_{f}|_{D_{g_1\cap g_2}} \to 0,
	\end{equation}
	where $f= h \circ (g_1 \cap g_2)$. We show that morphism $Rf_* R\hHom_X(\omega_f|_{D_{(g_1\cap g_2)}}, \omega_f|_{D_{g_2}}) \to Rf_*R\hHom_X(\omega_f|_{D_{g_1}}, \omega_f|_{D_{g_2}})$ obtained by applying $Rf_*R\hHom_X(-, \omega_f|_{D_g})$ to (\ref{eqtn_ses_1}) is an isomorphism. To this end, we check that $Rf_* R\hHom((g_1\cap g_2)^* (\omega_h|_{D_{\psi_1}}), \omega_f|_{D_{g_2}})$ is isomorphic to zero. We consider a short exact sequence:
	$$
	0 \to g_2^*\omega_{Z_2/Y} \to \omega_f \to \omega_f|_{D_{g_2}} \to 0.
	$$
	Applying $R(g_1\cap g_2)_*(-)$ to it yields $R(g_1\cap g_2)_* \omega_f|_{D_{g_2}} \simeq \omega_h|_{D_{\psi_2}}$. It follows that $Rf_* R\hHom((g_1\cap g_2)^* (\omega_h|_{D_{\psi_1}}), \omega_f|_{D_{g_2}}) \simeq Rh_*(\omega_{h}|_{D_{\psi_1}}, \omega_h|_{D_{\psi_2}}) =0$, since the supports of the two sheaves are disjoint, see Remark \ref{rem_disjoint_div}. We have thus reduced the question calculating cohomology of (\ref{eqtn_formula_1}) to the case of comparable elements $g_1\cap g_2 \preceq g_2$ of $\Dec(f)$.
	
	Now, we show that (\ref{eqtn_formula_1}) is a pure sheaf, for $g_1, g_2\in \Dec(f)$ comparable.
	
	Assume $g_1 \succeq g_2$, i.e. $g_1 \colon X \xrightarrow{g_2} Z_2 \xrightarrow{\f} Z_1$. It follows from Lemma \ref{lem_twisted_discr_in_C_g} below that the cone $Rf_* R\hHom_X(\oO_X, \omega_{g_1}|_{D_{g_2}})$ of map $Rf_* R\hHom_X(\omega_{g_1}|_{D_{g_1}}, \omega_{g_1}|_{D_{g_2}})
	\to Rf_*R\hHom_X(\omega_{g_1}, \omega_{g_1}|_{D_{g_2}})$ is zero. We use the resulting isomorphism to calculate
	\begin{align*}
	&Rf_* R\hHom_X(\omega_f|_{D_{g_1}}, \omega_f|_{D_{g_2}}) \simeq Rf_* R\hHom_X(\omega_{g_1}|_{D_{g_1}}, \omega_{g_1}|_{D_{g_2}}) \\
	&\simeq Rf_*R\hHom_X(\omega_{g_1}, \omega_{g_1}|_{D_{g_2}}) \simeq Rf_*(\oO_{D_{g_2}}).
	\end{align*}
	Since $Rf_* \oO_X$ is a pure sheaf, $\oO_{D_{g_2}}$ is a quotient of $\oO_X$, and $f$ has dimension of fibers bounded by one, then $Rf_*(\oO_{D_{g_2}})$ is a pure sheaf.
	
	Next, we calculate $Rf_* R\hHom_X(\omega_f|_{D_{g_2}}, \omega_f|_{D_{g_1}})$. Lemma \ref{lem_twisted_discr_in_C_g} below implies vanishing of $Rf_*R\hHom_X(\omega_{g_1}|_{D_{g_2}}, g_2^!(\omega_{\f}|_{D_{\f}})) = Rh_{2*}R \hHom_X(Rg_{2*}\omega_{g_1}|_{D_{g_2}},\omega_{\f}|_{D_{\f}})$. Then, exact sequence given by Lemma \ref{lem_embed_of_discrep} proves that morphism $Rf_*R\hHom_X(\omega_{g_1}|_{D_{g_2}}, \omega_{g_2}|_{D_{g_2}}) \to Rf_* R\hHom_X(\omega_{g_1}|_{D_{g_2}}, \omega_{g_1}|_{D_{g_1}})$ is an isomorphism. Thus, we have
	\begin{align*}
	& Rf_* R\hHom_X(\omega_f|_{D_{g_2}}, \omega_f|_{D_{g_1}}) \simeq Rf_* R\hHom_X(\omega_{g_1}|_{D_{g_2}}, \omega_{g_1}|_{D_{g_1}})\\
	&\simeq Rf_* R\hHom_X(\omega_{g_1}|_{D_{g_2}}, \omega_{g_2}|_{D_{g_2}}) \simeq Rf_*R\hHom_X(\omega_{g_1}, \omega_{g_2}|_{D_{g_2}}).
	\end{align*}
	The last isomorphism follows from an exact sequence 
	$$
	0 \to Lg_2^*(\omega_{\f}) \to \omega_{g_1} \to \omega_{g_1}|_{D_{g_2}} \to 0
	$$
	and vanishing of $Rf_* R\hHom_X(Lg_2^*(\omega_{\f}), \omega_{g_2}|_{D_{g_2}})\simeq Rh_{2*} R\hHom_X(\omega_{\f}, Rg_{2*}\omega_{g_2}|_{D_{g_2}})$, see Lemma \ref{lem_discr_in_C_g}. 
	
	Thus, to finish the proof, we need to show that
	$$
	Rf_* R\hHom_X(\omega_{g_1}, \omega_{g_2}|_{D_{g_2}}) \simeq Rf_* R\hHom_X(Lg_2^*(\omega_{\f}), \oO_{D_{g_2}}) \simeq Rh_{2*}(\omega_{\f}^{-1} \otimes Rg_{2*}(\oO_{D_{g_2}}))
	$$
	is a pure sheaf, where $h_2$ is such that $f\colon X\xrightarrow{g_2}Z_2 \xrightarrow{h_2}Y$.
	
	Since $Rg_{2*}\oO_X \simeq \oO_{Z_2}$, object $Rg_{2*}(\oO_{D_{g_2}})$ is a pure sheaf on $Z_2$. Moreover, it is supported on a closed subscheme $W_2\subset Z_2$ such that the fiber of $g_2$ over every closed point of $W_2$ is one dimensional.
	
	Let $\ol{h}_2=h_2|_{W_2}\colon W_2 \to Y$ be the induced map. Since $\omega_{\f}^{-1} \otimes Rg_{2*}(\oO_{D_{g_2}})$ is supported on $W_2$, we have $Rh_{2*}(\omega_{\f}^{-1} \otimes Rg_{2*}(\oO_{D_{g_2}}))\simeq R\ol{h}_{2*}(\omega_{\f}^{-1} \otimes Rg_{2*}(\oO_{D_{g_2}}))$. As $\omega_{\f}^{-1} \otimes Rg_{2*}(\oO|_{D_{g_2}})$, the derived tensor product of locally free sheaf $\omega_{\f}^{-1}$ with sheaf $Rg_{2*}(\oO_{D_{g_2}})$, is a pure sheaf, so is its push-forward $R\ol{h}_{2*}(\omega_{\f}^{-1}\otimes Rg_{2*}(\oO_{D_{g_2}}))$ by a finite map $\ol{h}_2$, see Lemma \ref{lem_get_finite_map} below. 
\end{proof}

\begin{LEM}\label{lem_get_finite_map}
	Let $g\colon X\to Z$ be an element of $\Dec(f)$. Let $W \subset Z$ be the image of the exceptional locus of $g$. Then the map $\ol{h}:=h|_W\colon W \to Y$ is finite.
\end{LEM}
\begin{proof}
	Map $\ol{h}$ is a restriction of $h$ to a closed subset, hence it is proper. Thus, by \cite[Theorem 8.11.1]{EGAIV}, in order to show that it is finite it suffices to check that it has finite fibers.
	
	Note that, for any $w \in W$, the fiber $g^{-1}(w)$ is of dimension one. If the fiber of $\ol{h}$ over $y\in h(W)$ would have dimension one, then the fiber of $\ol{h} \circ g$ over $y$, i.e. the fiber of $f$, would have dimension two, which contradicts the assumptions.
\end{proof}

\begin{LEM}\label{lem_twisted_discr_in_C_g}
	For $g_1 \succeq g_2 \in \Dec(f)$, $\omega_{g_1}|_{D_{g_2}} \in \cC_{g_2}$.
\end{LEM}
\begin{proof}
	Let $g_1 = \f \circ g_2$. Then the statement follows by the projection formula from the decomposition $\omega_{g_1}|_{D_{g_2}}\simeq g_2^*(\omega_{\f}) \otimes \omega_{g_2}|_{D_{g_2}}$ and vanishing of $Rg_{2*} \omega_{g_2}|_{D_{g_2}}$.
\end{proof}

Note that the discrepancy sheaf $\omega_X|_{D_g}$ of $g\colon X \to Z$ has a locally free resolution consisting of two invertible sheaves:
$$
0 \to g^*\omega_Z \to \omega_X \to \omega_X|_{D_g} \to 0.
$$
Since a pull back of a non-zero morphism of invertible sheaves is an injective morphism, for any surjective $\f \colon \wt{X} \to X$, we have
$$
L\f^*(\omega_X|_{D_g}) \simeq \f^*(\omega_X|_{D_g}).
$$
If $\wt{X}$ is Gorenstein, i.e. if $\omega_{\wt{X}}$ is an invertible sheaf, then smoothness of $X$ implies that
$$
\f^! (\omega_X|_{D_g}) \simeq L\f^*(\omega_X|_{D_g}) \otimes^L \f^!(\oO_X) \simeq \f^*(\omega_X|_{D_g}) \otimes \f^!(\oO_X)
$$
is also a pure sheaf.

For $ g \in \Dec(f)$, $g' \in \Dec(g)$ and decomposition $g \colon X \xrightarrow{g'} Z' \xrightarrow{h'} Z$, we have a canonical morphism $\alpha_{gg'}\colon \omega_{g'}|_{D_{g'}}\to \omega_{g}|_{D_g}$ induced by the canonical morphism $g'^!(\oO_{Z'}) \to g^!(\oO_Z)$.

\begin{LEM}\label{lem_embed_of_discrep}
	For $ (g,h) \in \Dec(f)$ and $(g',h') \in \Dec(g)$, we have a short exact sequence:
	\begin{equation}\label{eqtn_embed_of_discrep}
	0 \to \omega_{g'}|_{D_{g'}} \xrightarrow{\alpha_{gg'}} \omega_g|_{D_{g}}\to g'^!\omega_{h'}|_{D_{h'}}\to 0.
	\end{equation}
\end{LEM}
\begin{proof}
Applying $g'^!$ to short exact sequence (\ref{eqtn_dis_of_f}) for $h' \colon Z' \to Z$ yields a diagram
\[
\xymatrix{& g'^! \di{h'} \ar[r]^{\simeq} & g'^! \di{h'} \\
	\oO_X \ar[r] & g'^!h'^! \oO_{Z} \ar[r] \ar[u] & \di{g} \ar[u]\\
	\oO_X \ar[r]  \ar[u]^{\simeq} & g'^! \oO_{Z'} \ar[r] \ar[u] & \di{g'} \ar[u]}
\]
with exact rows and columns. The right column gives the result. 
\end{proof}
\begin{LEM}\label{lem_epi_of_discr}
	For $(g,h)\in \Dec(f)$ and $(g',h')\in \Dec(g)$, we have a short exact sequence:
	\begin{equation}\label{eqtn_restr_of_discr}
	0 \to g'^*(\omega_{h\circ h'}|_{D_{h'}}) \to \omega_f|_{D_g} \xrightarrow{\textrm{res}_{D_{g'}}^{\omega_f|_{D_g}} } \omega_f|_{D_{g'}} \to 0.
	\end{equation}
\end{LEM} 
\begin{proof}
	Recall from Appendix \ref{sec_canonical_morp} that there exists a canonical morphism of functors $\f_{Rg_*} \colon Lg^* \to g^!$. In particular, $s_g = \f_{Rg_*,\oO_{Z}}\colon \oO_X \to g^!(\oO_Z)$.
	  
	Since $s_{h'} = \f_{Rh'_*, \oO_Z}$, Proposition \ref{prop_map_f_g}(ii) applied to decomposition $Rg_* = Rh'_*\circ Rg'_*$ implies that the bottom left square of diagram
	\begin{equation}\label{eqtn_diagram_for_beta}
	\xymatrix{& \omega_g|_{D_{g'}}\ar[r]^{\textrm{id}} & \omega_g|_{D_{g'}}\\
		\oO_X \ar[r]^{s_g} & \omega_g \ar[r]^{\psi_g} \ar[u]^{\textrm{res}_{D_{g'}}^{\omega_g} } & \di{g} \ar[u]_{\textrm{res}_{D_{g'}}^{\di{g}}} \\
		\oO_X \ar[r]^{g'^*(s_{h'})} \ar[u]^{\textrm{id}} & g'^* \omega_{h'} \ar[r]^{g'^*(\psi_{h'})} \ar[u]|{\f_{Rg'_*, \omega_{h'}}} & g'^* \di{h'} \ar[u] }
	\end{equation}
	commutes. Map $\f_{Rg'_*, \omega_{h'}} \colon g'^*\omega_{h'} \to \omega_g$ is a morphism of line bundles with zeroes along $D_{g'}$, hence its cokernel is $\textrm{res}_{D_{g'}}^{\omega_g}$. As $\psi_g = \textrm{res}_{D_g}^{\omega_g}$ and a composition of restriction morphisms is a restriction morphism, the upper right square of the above diagram commutes. It follows that (\ref{eqtn_diagram_for_beta}) is commutative with exact rows and columns. Sequence (\ref{eqtn_restr_of_discr}) is obtained as the tensor product of the right column of (\ref{eqtn_diagram_for_beta}) with $g^*(\omega_h) \simeq g'^*(h'^* \, \omega_h)$.
\end{proof}

\begin{LEM}\label{lem_pull_back_of_discr}	
	For $(g\colon X \to Z, h\colon Z \to Y)\in \Dec(f)$ and $g'\in \Conn(f)\setminus \Conn(g)$ let $\f \in \Conn(f)$ be such that $g\cup g' = \f \circ g$. Then sequence
	$$
	0 \to Lg^*(\omega_Z|_{D_{\f}}) \to \omega_X|_{D_{g'}} \xrightarrow{\textrm{res}_{D_{g\cap g'}}^{\omega_X|_{D_{g'}}}} \omega_X|_{D_{g\cap g'}} \to 0
	$$
	is exact, where $\omega_{X}|_{D_{\Id}}:=0$.
\end{LEM} 
\begin{proof}
	If $g\cap g'= \Id_X$, morphism $g$ is an isomorphism in the neighbourhood of $D_{g'}$, by Remark \ref{rem_disjoint_div}, hence $\omega_X|_{D_{g'}} \simeq Lg^*Rg_* \omega_{X}|_{D_{g'}} \simeq Lg^* \omega_{Z}|_{D_\f}$, see Proposition \ref{prop_rez_of_omega_f_D_g}.
	
	Assume that $g\cap g' \neq \Id_X$. Morphisms $g$, $g'$ fit into diagram (\ref{eqtn_diag_cup_and_cap}) with $g_1=g$ and $g_2 = g'$. We put $\psi:= \psi_1$ and $\psi' := \psi_2$. Tensor product of (\ref{eqtn_restr_of_discr}) for $g \cap g' = \xi \in \Dec(g')$ with $f^*\omega_Y$ yields a short exact sequence
	\begin{equation*}\label{eqtn_restr_of_twisted_discr}
	0 \to L(g\cap g')^*(\omega_{Z_{\cap}}|_{D_{\psi'}}) \to \omega_X|_{D_{g'}} \to \omega_X|_{D_{g\cap g'}} \to 0. 
	\end{equation*}
	
	We show that $Lg^*(\omega_Z|_{D_{\f}}) \simeq L(g\cap g')^*(\omega_{Z_{\cap}}|_{D_{\psi'}})$. 
	
	It follows from Remark \ref{rem_disjoint_div} that $\psi$ is an isomorphism in the neighbourhood of $\Ex(\psi')$. Hence, $\omega_{Z_{\cap}}|_{D_{\psi'}} \simeq L\psi^* R\psi_* (\omega_{Z_{\cap}}|_{D_{\psi'}})$. Since $Lg^* = L(g\cap g')^* \circ L\psi^*$, the statement follows from isomorphisms $R\psi_* \omega_{Z_{\cap}}|_{D_{\psi'}} \simeq Rg_* \omega_{X}|_{D_{g'}} \simeq \omega_Z|_{D_{\f}}$, both given by Proposition \ref{prop_rez_of_omega_f_D_g}; the first for the pair $g\cap g'$, $g'$ of elements of $\Dec(f)$.
\end{proof}
 
\begin{PROP}\label{prop_pull_back_T_hh}
	For $(g,h)\in \Dec(f)$, sequence
	\begin{equation}\label{eqtn_pull_back_T}
	0 \to Lg^*(T_h) \to \ol{T}^f_g \to \bigoplus_{g'\in \Conn(f)\setminus \Conn(g)} \omega_X|_{D_{g\cap g'}} \to 0
	\end{equation}
	is exact. 
\end{PROP}
\begin{proof}
	By Proposition \ref{prop_rez_of_omega_f_D_g} any direct summand of $T_h$ is of the form $Rg_* \omega_X|_{D_{g'}}$, for some $g'\in \Conn(f)\setminus \Conn(g)$. Short exact sequence (\ref{eqtn_pull_back_T}) is then the direct sum of the short exact sequences of Lemma \ref{lem_pull_back_of_discr}, for all $g'\in \Conn(f)\setminus \Conn(g)$.
\end{proof}

\begin{THM}\label{thm_tilting_in_C_f}
	For $g\in \Dec(f)$, object $T_g \in \cC_g$ is a tilting relative  generator over $Y$.
\end{THM}
First, we consider the case when $g=s\colon X \to Z$ is the blow-up of $Z$ along a smooth irreducible $\bB_g$ of codimension two, see diagram (\ref{eqtn_diagram_single_blow-up}). 
\begin{LEM}\label{lem_generating_base_case}
	For $g=s\in \Conn(f)$ as above, $\omega_X|_{D_g}$ is a tilting relative  generator over $Y$.
\end{LEM}
\begin{proof}
	We shall assume that $Y = \Spec(R)$ is affine and check that $\omega_X|{D_g}$ generates $\cC_g$.
	
	Since $g^*\omega_Z\otimes(-) \colon \cC_g \to \cC_g$ is an equivalence, $\omega_X|_{D_g}\simeq g^*\omega_Z \otimes \omega_g|_{D_g}$ generates $\cC_g$ if and only if $\omega_g|_{D_g} =\oO_E(E)= R\zeta_*Lp^*(\oO_{\bB_g})\otimes \oO_X(E)$ does. As the codimension of $\bB_g$ is 2, functor $R\zeta_*Lp^*(-)\otimes \oO_X(E)\colon \dD^b(\bB_g) \to \cC_g$ is an equivalence (see \cite[Theorem 4.3]{Orl}, \cite[Proposition 3.2]{BonOrl}). Thus, it suffices to check that $\oO_{\bB_g}$ generates $\dD^b(\bB_g)$ over $Y$. 
	
	Note that, by Lemma \ref{lem_get_finite_map} morphism $h|_{\bB_g} \colon \bB_g \to Y$ is affine. Since $Y$ is affine, so is $\bB_g$. As $\bB_g$ is smooth, we have $\dD^b(\bB_g)=\langle \oO_{\bB_g} \rangle$, which concludes the proof.
\end{proof}

\begin{proof}[Proof of Theorem \ref{thm_tilting_in_C_f}]
	Note that $T_g$ is a direct summand of $T_f$. Hence, it follows from 
	Proposition \ref{prop_T_ff_no_higher_ext} that $Rf_* R\hHom_X(T_g, T_g)$ is a pure sheaf on $Y$. Therefore, it suffices to check that, for an open affine $U\subset Y$ and $V:=f^{-1}(U)$, category $\cC_{g|_V}$ is generated by $T_g|_V$. Since $g$ factors through $g' \in \Conn(g)$, Lemma \ref{lem_discr_in_C_g} implies that $R(g|_{V})_* \, T_g|_{V} = 0$. Thus, the triangulated subcategory of $\dD^b(V)$ generated by $T_g|_{V}$ is contained in $\cC_{g|_{V}}$. 
	
	Let $Y$ be affine and $g\in \Dec(f)$. We prove by induction on $|\Conn(g)|$ that $T_g$ generates $\cC_g$. If $|\Conn(g)|=1$, the statement follows from Lemma \ref{lem_generating_base_case}.
	If $|\Conn(g)|>1$, then we consider a decomposition $g\colon X \xrightarrow{g'} Z' \xrightarrow{h'} Z$ such that $g'$ is the blow-up of a smooth irreducible $W'\subset Z'$. 
	According to Proposition \ref{prop_pull_back_T_hh} applied to $(g',h') \in \Dec(g)$, sequence
	$$
		0 \to Lg'^*(T_{h'}) \to \ol{T}_{g'}^g \to \bigoplus_{\wt{g}\in \Conn(g)\setminus \{g'\}} \omega_X|_{D_{\wt{g}\cap g'}} \to 0
	$$
	is exact. 
	Since intersection $\wt{g} \cap g'$ of $\wt{g}$ and $g'$ in $\Dec(g)$ equals $g'$ or $\Id_X$, for any $\wt{g} \in \Conn(g)$, the last object of this sequence is a direct sum of copies of $T_{g'}$. As this object together with $\ol{T}_{g'}^g$ are direct summands of $T_g$, object $Lg'^*(T_{h'})$ is split generated by $T_g$.

	Decomposition $g=h'\circ g'$ yields SOD $\cC_g = \langle \cC_{g'}, Lg'^*\cC_{h'} \rangle$. 
	By the inductive hypothesis $Lg'^*\cC_{h'}\subset \cC_g$ is split generated by $Lg'^{*} T_{h'}$, hence by $T_g$. The other component $\cC_{g'}$ of the SOD is split generated by $T_g$, because it is generated by $T_{g'}=\omega_X|_{D_{g'}}$ by Lemma \ref{lem_generating_base_case}. 
\end{proof}

\subsection{Tilting relative generators for $\dD^b(X)$}\label{ssec_tiliting_object}

Define

$$
T_X = T_{X,f} := \omega_X \oplus T_f =\omega_X \oplus \bigoplus_{g\in \Conn(f)}\omega_{X}|_{D_g}.
$$

\begin{PROP}\label{prop_T_Xf_no_higher_ext}
	$Rf_*R\hHom_X(T_{X,f},T_{X,f})$ is a pure sheaf on $Y$.
\end{PROP}
\begin{proof}
	In view of Proposition \ref{prop_T_ff_no_higher_ext}, it suffices to show that $Rf_* R\hHom_X(\omega_X, \omega_X|_{D_g})$ and $Rf_* R\hHom_X(\omega_X|_{D_g}, \omega_X)$ is a pure sheaf, for any $g \in \Conn(f)$.
	
	Object
	\begin{align*}
	Rf_*R \hHom_X(\omega_X, \omega_X|_{D_g}) \simeq Rf_* \oO_{D_g}
	\end{align*} 
	is a pure sheaf as $\oO_{D_g}$ is a quotient of $\oO_X$ and $R^1f_* \oO_X =0$.
	
	Moreover, 
	\begin{align*}
	Rf_*R \hHom_X(\omega_X|_{D_g}, \omega_X) \simeq R\hHom(Rf_*(\omega_X|_{D_g}), \omega_Y) \simeq 0,
	\end{align*}
	by Lemma \ref{lem_discr_in_C_g}.
\end{proof}

Clearly, $\dD^b(X)$ is a stack of subcategories of $\dD^b(X)$ over $Y$. We have:
\begin{THM}\label{thm_tilting_in_D(X)}
	Object $T_{X,f}\in \dD^b(X)$ is a tilting relative  generator for $\dD^b(X)$ over $Y$.
\end{THM}
\begin{proof}
	In view of Proposition \ref{prop_T_Xf_no_higher_ext}, it suffices to check that $T_{X,f}$ generates $\dD^b(X)$ if $Y$ is affine. Since $Rf_* \oO_X \simeq \oO_Y$, functor $f^!$ is fully faithful and we have a SOD 
	$$
	\dD^b(X) = \langle f^! \dD^b(Y), \cC_f \rangle.
	$$
	In view of Theorem \ref{thm_tilting_in_C_f}, $T_f$ generates $\cC_f$. As $\omega_X = f^!(\omega_Y)$, the first summand of $T_{X,f}$, generates $f^!\dD^b(Y)$, the statement follows.
\end{proof}

Finally, note that, for $f\colon X \to Y$ as above, the relative duality functor $\mathscr{D}_f(-)=R\hHom_X(-, \omega_f^\bcdot)$ is local over $Y$. By Lemma \ref{lem_dual_preserv_C_f}, $\mathscr{D}_f(\cC_f) \subset \cC_f$. Then $\mathscr{D}_f(T_f) \in \cC_f$, $\mathscr{D}_f(T_{X,f}) \in \dD^b(X)$ are tilting  relative generators over $Y$. It will be convenient for us to use the functor $\mathscr{D}_X(-)= R\hHom_X(-,\omega_X)$ instead. Functor $\mathscr{D}_X$ differs from $\mathscr{D}_f$ by a twist with $Lf^* \omega_Y$, hence it also satisfies $\mathscr{D}_X(\cC_f) \subset \cC_f$.
Define 
\begin{align*} 
&S_f:= \mathscr{D}_X(T_f),& &S_{X,f}=S_X:= \mathscr{D}_X(T_{X,f}).& 
\end{align*} 
A simple calculation for these objects shows
\begin{align*} 
&S_f=\bigoplus_{g\in \Conn(f)} \omega_g|_{D_g}[-1],& &S_{X,f}= \oO_X \oplus S_f.&
\end{align*} 
\begin{COR}\label{cor_dual_tilting}
	Objects $S_{X,f}$ and $S_f$ are tilting relative  generators over $Y$ for $\cC_f$ and for $\dD^b(X)$ respectively.
\end{COR}

\section{\textbf{A system of \tr e on $\dD^b(X)$ related to $f$}}

\subsection{T-structures on the null-category and on $\dD^b(X)$: one blow-up}

Since $X$ is smooth, it follows from Theorem \ref{thm_tilting_in_D(X)} that we have the $T_X$-projective and $T_X$-injective \tr es on $\dD^b(X)$ defined as in (\ref{eqtn_t-str_on_stack}). 
Similarly, in view of Theorem \ref{thm_tilting_in_C_f}, the null category $\cC_f$ admits the $T_f$-projective and the $T_f$-injective \tr es (\ref{eqtn_t-str_on_stack}). Also we have $S_X$-projective and $S_X$-injective \tr es on $\dD^b(X)$, and similarly for $\cC_f$.

\begin{REM}\label{rem_line_bundle_pro}
	Note that, for a smooth algebraic space $X$ and $f=\Id_X$, any line bundle $\lL\in \dD^b(X)$ is a tilting relative  object over $X$. Since tensor product with a line bundle is a $t$-exact functor in the standard \tr e and the $\oO_X$-projective \tr e is the standard one, then so is the $\lL$-projective \tr e, for any $\lL \in \textrm{Pic}(X)$. 	
	
	Define the \emph{dual \tr e} on $\dD^b(X)$ as the one induced by applying the duality $R\hHom(-, \oO_X) \colon \dD^b(X)^\textrm{op} \to \dD^b(X)$ to the standard \tr e on $\dD^b(X)$, \emph{cf.} \cite{Bon2}. Then  the $\lL$-injective \tr e is the dual one on $\dD^b(X)$, for any $\lL\in \textrm{Pic}(X)$.
\end{REM}

Consider the case of $f$ being the blow-up of $Y$ along a smooth $\bB_f$ of codimension two. Then, in the notation of (\ref{eqtn_diagram_single_blow-up}) for $s=f$, we have an equivalence
\begin{align*} 
&\theta \colon \dD^b(\bB_f) \xrightarrow{\simeq} \cC_f,& &\theta(-)= R\zeta_*Lp^*(-)\otimes \oO_X(E).
\end{align*} 
The following two Propositions describe the $T_f$-projective and $T_f$-injective \tr es on $\cC_f$ under the above equivalence.
\begin{PROP}\label{prop_dual_t-str_on_W_and_C_f}
	Equivalence $\theta$ is $t$-exact when $\dD^b(\bB_f)$ is endowed with the dual \tr e and $\cC_f$ with the 
	$T_f$-injective \tr e. 
\end{PROP}
\begin{proof}
	Consider $F\in \dD^b(\bB_f)$ such that $R\hHom_W(F, \oO_{\bB_f})$ is a pure sheaf. We need to check that $Rf_*R\hHom_X(R\zeta_*Lp^*(F)\otimes \oO_X(E), \omega_X|_{E})$ is a pure sheaf on $Y$. Since $\omega_X = \oO_X(E)\otimes f^*\omega_Y$, the projection formula yields
	\begin{align*}
	&Rf_*R\hHom_X(R\zeta_*Lp^*(F)\otimes^L_X \oO_X(E), \omega_X|_{E})\\
	&\simeq Rf_*R\hHom_X(R\zeta_*Lp^*(F)\otimes^L_X \oO_X(E), \oO_E(E)) \otimes^L_Y \omega_Y,
	\end{align*}
	hence, it suffices to check that $Rf_*R\hHom_X(R\zeta_*Lp^*(F), \oO_E)$ is a pure sheaf on $Y$. The statement is local on $Y$, therefore we can assume that $Y$, hence also $\bB_f$ are affine. In view of isomorphism $\oO_E \simeq R\zeta_*Lp^*(\oO_{\bB_f})$, it suffices to show that $\Ext^p_X(R\zeta_*Lp^*(F),R\zeta_*Lp^*(\oO_{\bB_f}))=0$, for $p\neq 0$. The statement follows from the fact that $R\zeta_*Lp^*$ is fully faithful \cite[Proposition 3.2]{BonOrl}.
\end{proof}

Similarly, one shows 
\begin{PROP}\label{prop_st_t-str_on_W_and_C_f}
	Equivalence $\theta$ is $t$-exact when $\dD^b(\bB_f)$ is endowed with the standard \tr e and $\cC_f$ with the $T_f$-projective \tr e. 
\end{PROP}

However, for an arbitrary $f$ neither of the four \tr es on $\cC_f$ induced by $T_f$ and $S_f$ is the standard or the dual one. 

\begin{EXM}
	Consider the contraction $f\colon X \to Y$ of a divisor in a smooth surface with two rational components $C_1$, $C_2$ with intersection numbers $C_1^{2}=-2$, $C_2^2=-1$, $C_1.C_2=1$. Then $T_f \simeq  \oO_{C_2}(C_2) \oplus \oO_{C_1+2C_2}(C_1 + 2C_2)$ and $\Ext^1_X(\oO_{C_2}(C_2), \oO_{C_1}(-1))$ has dimension 1, i.e. sheaf $\oO_{C_1}(-1)\in \cC_f$ does not lie in the heart of the $T_f$-projective \tr e on $\cC_f$. 
\end{EXM}

\subsection{Duality and gluing}

Let $\tT$ be a triangulated category with a dualising functor $\mathscr{D}\colon \tT^{\textrm{op}} \to \tT$. For $\aA\subset \tT$, we denote by $i_{\aA*}$ both embeddings $\aA \to \tT$ and $\aA^{\textrm{op}} \to \tT^{\textrm{op}}$. An SOD $\tT = \langle \aA, \bB\rangle$ gives an SOD $\tT = \langle \mathscr{D}(\bB^{\textrm{op}}), \mathscr{D}(\aA^{\textrm{op}}) \rangle$. If we denote by $i_{\aA}^*$ the functor left adjoint to the embedding $i_{\aA*}\colon \aA \to \tT$ and by $i_{\bB}^!$ the functor right adjoint to the embedding $i_{\bB*}\colon \bB\to \tT$ ,then $i_{\aA}^* \mathscr{D} \colon \tT \to \aA^{\textrm{op}}$ is right adjoint to the embedding $i_{\mathscr{D}(\aA^{\textrm{op}})*} = \mathscr{D} i_{\aA*}$ and $i_{\bB}^! \mathscr{D}\colon \tT \to \bB^{\textrm{op}}$ is left adjoint to $i_{\mathscr{D}(\bB^{\textrm{op}})*} = \mathscr{D} i_{\bB*}$.

If $\bB \subset \tT$ is admissible, i.e. we have SOD's $\tT = \langle \aA, \bB \rangle = \langle \bB, \cC \rangle$, then $\tT$ admits SOD's $\tT = \langle \mathscr{D}(\cC^{\textrm{op}}), \mathscr{D}(\bB^{\textrm{op}}) \rangle = \langle \mathscr{D}(\bB^{\textrm{op}}), \mathscr{D}(\aA^{\textrm{op}}) \rangle$, i.e. $\mathscr{D}(\bB^{\textrm{op}}) \subset \tT$ is admissible. Thus, the recollement
\begin{equation}\label{eqtn_recol_on_T}
\xymatrix{\bB \ar[rr]|{i_{\bB*}} && \tT \ar@<-2ex>[ll]|{i_{\bB}^*} \ar@<2ex>[ll]|{i_{\bB}^!} \ar[rr]|{j^*} && \tT/\bB \ar@<-2ex>[ll]|{i_{\cC *}}  \ar@<2ex>[ll]|{i_{\aA*}}  }
\end{equation}
yields the $\mathscr{D}$-dual recollement
\begin{equation}\label{eqtn_dual_recol_on_T}
\xymatrix{\bB^{\textrm{op}} \ar[rr]|{\mathscr{D}i_{\bB*}} && \tT \ar@<-2ex>[ll]|{i_{\bB}^! \mathscr{D}} \ar@<2ex>[ll]|{i_{\bB}^* \mathscr{D}} \ar[rr]|{j^*\mathscr{D}} && (\tT/\bB)^{\textrm{op}}. \ar@<-2ex>[ll]|{\mathscr{D}i_{\aA *}}  \ar@<2ex>[ll]|{\mathscr{D}i_{\cC*}}  }
\end{equation}
\begin{PROP}\label{prop_glued_and_dual}
Assume that $\bB$ and $\tT/\bB$ are endowed with \tr es $(\bB^{\lle 0}, \bB^{\gge 1})$, $((\tT/\bB )^{\lle 0}, (\tT/\bB )^{\gge 1})$. Let $(\tT^{\lle 0}, \tT^{\gge 1})$ be the \tr e on $\tT$ glued via recollement (\ref{eqtn_recol_on_T}). Then the $\mathscr{D}$-dual \tr e on $\tT$, $(\mathscr{D}(\tT^{\gge 0}), \mathscr{D}(\tT^{\lle -1}))$, is glued from corresponding \tr es on $\bB^{\textrm{op}}$ and $(\tT/\bB)^{\textrm{op}}$ via recollement (\ref{eqtn_dual_recol_on_T}).	
\end{PROP}
\begin{proof} Since $(\tT^{\lle 0}, \tT^{\gge 1})$ is glued, for any $B\in \bB^{\lle 0}\cap \bB^{\gge 0}$ and any $T \in \tT^{\lle 0} \cap \tT^{\gge 0}$, object $i_{\bB*}(B)$ lies in $\tT^{\lle 0} \cap \tT^{\gge 0}$ and $j^*(T)$ lies in $(\tT/\bB )^{\lle 0} \cap (\tT/\bB )^{\gge 0}$. It follows that $\mathscr{D}i_{\bB*}(B)$ lies in $\mathscr{D}(\tT^{\gge 0}) \cap \mathscr{D}(\tT^{\lle 0})$ and, for any $T' \in \mathscr{D}(\tT^{\gge 0}) \cap \mathscr{D}(\tT^{\lle 0}) = \mathscr{D}(\tT^{\lle 0} \cap \tT^{\gge 0})$, object $j^*\mathscr{D}(T')$ lies in $(\tT/\bB )^{\lle 0} \cap (\tT/\bB )^{\gge 0}$, i.e. functors $\mathscr{D}i_{\bB*}$ and $j^*\mathscr{D}$ are $t$-exact.
\end{proof} 
\begin{REM}\label{rem_duality_restr}
	If duality $\mathscr{D}$ restricts to a duality $\bB^{\textrm{op}} \to \bB$, then recollement (\ref{eqtn_recol_on_T}) is self dual. In this case, $\mathscr{D}(\aA^{\textrm{op}}) \simeq \cC$ and the duality $\mathscr{D}i_{\aA*}\colon \aA^{\textrm{op}}\to \cC$ composed with the mutation $i_{\aA}^* i_{\cC*}\colon \cC \xrightarrow{\simeq} \aA$ gives duality $\mathscr{D}_q$ on $\tT/\bB \simeq \aA$ such that $j^*\mathscr{D} = \mathscr{D}_q j^*$.
\end{REM}

Let $f\colon X \to Y$ be a birational morphism of smooth algebraic spaces. Since $Y$ is smooth, functors $Lf^*$ and $f^!$ are related via $f^!(-) = Lf^*(-) \otimes \omega_f$, \emph{cf.} \cite[Lemma A.1]{BodBon}. It follows that $f^!\dD^b(Y) \simeq Lf^*\dD^b(Y) \otimes \omega_f$. Twists by powers of $\omega_f$ of the SOD's $\dD^b(X) = \langle f^! \dD^b(Y), \cC_f \rangle = \langle \cC_f, Lf^* \dD^b(Y) \rangle$ give an infinite chain of SOD's
\begin{equation*}
\ldots = \langle \cC_f \otimes \omega_f, f^! \dD^b(Y)\rangle = \langle f^!\dD^b(Y), \cC_f \rangle = \langle \cC_f, Lf^* \dD^b(Y) \rangle = \langle Lf^*\dD^b(Y), \cC_f \otimes \omega_{f}^{-1} \rangle = \ldots
\end{equation*}
They yield in turn an infinite sequence of recollements for $\dD^b(X)$.

Since $S_X = \mathscr{D}_f(T_X)$ and $S_f = \mathscr{D}_f(T_f)$, the $T_X$-projective and $S_X$-injective \tr es are $\mathscr{D}_f$-dual, and so are $T_{f}$-injective and $S_f$-projective ones.
\begin{COR}\label{cor_dual_t-str_on_D_X}
	If the $T_X$-injective (resp. projective) \tr e on $\dD^b(X)$ is glued w.r.t. subcategory $\bB\subset \dD^b(X)$ then the $S_X$-projective (resp. injective) \tr e is glued w.r.t. subcategory $\mathscr{D}_f(\bB^{\textrm{op}}) \subset \dD^b(X)$. Similarly for the \tr es on $\cC_f$ induced by $T_f$ and $S_f$. 
\end{COR}

Duality $\mathscr{D}_f$ preserves $\cC_f \subset \dD^b(X)$, Lemma \ref{lem_dual_preserv_C_f}. Remark \ref{rem_duality_restr} implies that the quotient category $\dD^b(Y) \simeq \dD^b(X)/\cC_f$ admits a duality $\mathscr{D}_q$. It is actually the standard duality $\mathscr{D}_Y = R\Hom(-, \oO_Y)$, because the Grothendieck-Verdier duality implies that $Rf_* \mathscr{D}_f \simeq \mathscr{D}_Y Rf_*$. Hence, if a \tr e on $\dD^b(X)$ is glued with respect to subcategory $\cC_f$ from \tr es on $\cC_f$ and $\dD^b(Y)$ then its $\mathscr{D}_f$-dual is glued with respect to subcategory $\cC_f$ from the $\mathscr{D}_f$-dual \tr e on $\cC_f$ and the $\mathscr{D}_Y$-dual \tr e on $\dD^b(Y)$.

\subsection{The gluing properties for the \tr es}\label{ssec_gluing_prop_for_t_str}

Now we shall describe how our new \tr es are glued via various recollements.

\begin{LEM}\label{lem_t-exact_func_C_f_to_D(X)}
	Functor $\iota_{f*} \colon \cC_f\to \dD^b(X)$  is $t$-exact for the $T_f$-injective \tr e on $\cC_f$ and the $T_X$-injective \tr e on $\dD^b(X)$.
\end{LEM}
\begin{proof}
	Since $T_{X,f} = T_f \oplus \omega_X$,  in order to check that $A$ lies in the heart of the $T_f$-injective \tr e 
	it suffices, by (\ref{eqtn_t-str_on_stack}), to calculate $Rf_*R\hHom_X(\iota_{f*}A, \omega_X)$. As $A\in \cC_f$, we have $Rf_*R\hHom_X(\iota_{f*}A, \omega_X) \simeq R\hHom_Y(Rf_* \iota_{f*}A, \omega_Y) =0$, hence the result. 
\end{proof}
For $g\colon X \to Z$ in $\Dec(f)$, object $T_g\in \cC_f$ is relatively tilting both over $Z$ and over $Y$, c.f. Theorem \ref{thm_tilting_in_C_f}. Thus, the heart of the $T_g$-projective \tr e on $\cC_g$ reads 
\begin{equation}\label{eqtn_heart_on_C_g} 
\begin{aligned}
\cC_{g\, \textrm{pro }T_g}^{\lle 0} \cap \cC_{g\, \textrm{pro }T_g}^{\gge 0} &= \{E \in \cC_g\,|\, Rg_* R\hHom_X(\iota_{g*} T_g, \iota_{g*} E) \, \textrm{is a pure sheaf}\} \\
&=\{E \in \cC_g\,|\, Rf_* R\hHom_X(\iota_{g*} T_g, \iota_{g*} E) \, \textrm{is a pure sheaf}\}.  
\end{aligned}
\end{equation} 
Indeed, cohomology sheaves of the complex $Rg_*R\hHom_X(\iota_{g*} T_g, \iota_{g*}E)$ are supported on the image $W$ of the exceptional locus of $g$, for any $E\in \cC_g$. By Lemma \ref{lem_get_finite_map}, the map $h|_W$ is finite. It follows that $Rg_*R\hHom_X(\iota_{g*}T_g, \iota_{g*}E)$ is a pure sheaf if and only if so is $Rh_* Rg_* R\hHom_X(\iota_{g*}T_g, \iota_{g*}E)$.

\begin{LEM}\label{lem_t-exact_func_D(X)_to_C_f}
	For a decomposition $f\colon X \xrightarrow{g}Z \xrightarrow{h} Y$, 
	\begin{itemize}
		\item[(i)] functor $\iota_g^! \colon \dD^b(X) \to \cC_g$ and its restriction to $\cC_f$ are $t$-exact, for the $T_{X,f}$-projective \tr e on $\dD^b(X)$, $T_f$-projective \tr e on $\cC_f$ and the $T_g$-projective \tr e on $\cC_g$. 
		\item[(ii)] Functor $\iota_g^*\colon \dD^b(X) \to \cC_g$ and its restriction to $\cC_f$ are $t$-exact, for the $T_{X,f}$-injective \tr e on $\dD^b(X)$, $T_f$-injective \tr e on $\cC_f$ and the $T_g$-injective \tr e on $\cC_g$.
	\end{itemize} 
\end{LEM}
\begin{proof}
	We consider the case of functors defined on $\dD^b(X)$. The proofs of $t$-exactness of their restrictions to $\cC_f$ are analogous.
	 
	In view of description (\ref{eqtn_heart_on_C_g}), to check that $\iota_g^!$ is $t$-exact it suffices to calculate $Rf_*R\hHom_X(T_g, \iota_g^!(A))$, for $A$ in the heart of the $T_X$-projective \tr e on $\dD^b(X)$.
	
	Note that, for $E \in \cC_g$, applying $Rg_*R\hHom_X(\iota_{g*}E, -)$ to a triangle of functors (\ref{eqtn_def_of_iota!}) with $f$ replaced by $g$ and applied to object $A$, yields an isomorphism $Rg_*R\hHom_X(\iota_{g*}E, \iota_{g*}\iota_g^!A) \simeq Rg_*R\hHom_X(\iota_{g*}E, A)$. Since $T_g$ is a direct summand of $T_{X,f}$, for $A \in \dD^b(X)_{\textrm{pro }T_{X}}^{\lle 0} \cap \dD^b(X)_{\textrm{pro }T_{X}}^{\gge 0}$, object $Rf_*R\hHom_X(\iota_{g*}T_g, \iota_{g*}\iota_g^!A) \simeq Rf_* R\hHom_X(\iota_{g*}T_g, A)$ is a pure sheaf, i.e. $\iota_g^!$ is $t$-exact.
		
	Analogously, triangle (\ref{eqtn_def_of_iota*}) yields $Rg_*R\hHom_X(A, \iota_{g*}E) \simeq Rg_*R\hHom_X(\iota_{g*}\iota_g^*A, E)$. As $T^g_g$ is a direct summand of $T_{X,f}$, the statement follows.
\end{proof}

\begin{LEM}\label{lem_t-exact_func_D(X)_to_D(Y)}
	Functor $Rf_* \colon \dD^b(X) \to \dD^b(Y)$ is $t$-exact for the $T_X$-injective \tr e on $\dD^b(X)$ and the dual \tr e on $\dD^b(Y)$.
\end{LEM}
\begin{proof}
	This follows from description (\ref{eqtn_t-str_on_stack}) and  the local duality $R\hHom_Y(Rf_*(-), \omega_Y) \simeq Rf_*R\hHom_X(-, \omega_X)$.
\end{proof}

\begin{LEM}\label{lem_t-exact_func_D(Y)_to_D(X)}
	For a decomposition $f\colon X \xrightarrow{g}Z \xrightarrow{h} Y$, \begin{itemize} 
		\item[(i)] functor $Lg^*\colon \dD^b(Z) \to \dD^b(X)$ and its restriction $Lg^* \colon \cC_h \to \cC_f$ are $t$-exact, for the $T_{Z,h}$-injective \tr e on $\dD^b(Z)$, the $T_{X,f}$-injective \tr e on $\dD^b(X)$, the $T_h$-injective \tr e on $\cC_h$ and the $T_f$-injective \tr e on $\cC_f$. 
		\item[(ii)] Functor $g^!\colon \dD^b(Z) \to \dD^b(X)$ and its restriction $g^! \colon \cC_h \to \cC_f$ are $t$-exact, for the $T_{Z,h}$-projective \tr e on $\dD^b(Z)$, the $T_{X,f}$-projective \tr e on $\dD^b(X)$, the $T_h$-projective \tr e on $\cC_h$ and the $T_f$-projective \tr e on $\cC_f$.
	\end{itemize} 
\end{LEM}
\begin{proof}
By local duality we have $Rf_*R\hHom_X(Lg^*(-), \omega_X) \simeq Rh_*R\hHom_Z(-,\omega_Z)$ and $Rf_*R\hHom_X(Lg^*(-), T_f) \simeq Rh_*R\hHom_Z(-, T_h)$, c.f. Proposition \ref{prop_rez_of_omega_f_D_g}. It follows that functor $Lg^*$ and its restriction to $\cC_h$ are $t$-exact. Similarly, $Rf_*R\hHom_X(\omega_X,g^!(-)) \simeq Rh_*R\hHom_Z(\omega_Z, -)$ and $Rf_*R\hHom_X(T_f, g^!(-)) \simeq Rh_*R\hHom_Z(T_h,-)$, which proves that $g^!$ and its restriction to $\cC_h$ are $t$-exact.
\end{proof}

\begin{PROP}\label{prop_inj_t-str_D(X)_glued}
	The $T_{X,f}$-injective \tr e on $\dD^b(X)$ is glued: 
	\begin{itemize}
		\item[(i)] from the $T_{Z,h}$-injective \tr e on $\dD^b(Z)$ and 
		the $T_g$-injective \tr e on  $\cC_{g}$ via recollement $\dD^b(Z) \xrightarrow{Lg^*} \dD^b(X) \xrightarrow{\iota_g^*} \cC_g$, for any decomposition $f\colon X \xrightarrow{g}Z \xrightarrow{h} Y$. It is the \tr e glued from the dual \tr es on $\dD^b(Y)$ and $\dD^b(\bB_g)$ via the strict $\Dec^+(f)^\textrm{op}$-filtration $g\mapsto Lg^* \dD^b(Z)$ of Remark \ref{rem_filtr_on_D_X}.
		\item[(ii)]  
		from the $T_f$-injective \tr e on $\cC_{f}$ and the dual \tr e on $\dD^b(Y)$ via recollement $\cC_f \xrightarrow{\iota_{f*}} \dD^b(X) \xrightarrow{Rf_*} \dD^b(Y)$.
	\end{itemize}	
\end{PROP}
\begin{proof} Functors $Lg^*$, $\iota_g^*$, $\iota_{f*}$, $Rf_*$ are t-exact by Lemmas \ref{lem_t-exact_func_C_f_to_D(X)}, \ref{lem_t-exact_func_D(X)_to_C_f}, \ref{lem_t-exact_func_D(X)_to_D(Y)} and \ref{lem_t-exact_func_D(Y)_to_D(X)}. Uniqueness of the glued \tr e (see Theorem \ref{thm_gluing_via_poset}) implies that the $T_{X,f}$-injective \tr e is glued from the $\omega_Y$-injective \tr e on $\dD^b(Y)$ and $T_{f_i}$-injective \tr es on $\cC_{f_i}$. The statement follows from Remark \ref{rem_line_bundle_pro} and Proposition \ref{prop_dual_t-str_on_W_and_C_f}.
\end{proof}

\begin{PROP}
	For any decomposition $f\colon X \xrightarrow{g}Z \xrightarrow{h} Y$, the $T_{X,f}$-projective \tr e on $\dD^b(X)$ is glued from the $T_{Z,h}$-projective \tr e on $\dD^b(Z)$ and 
	the $T_g$-projective \tr e on $\cC_{g}$, via recollement $\dD^b(Z) \xrightarrow{g^!} \dD^b(X) \xrightarrow{\iota_g^!} \cC_g$.
\end{PROP}
\begin{proof}
It follows from Lemmas \ref{lem_t-exact_func_D(X)_to_C_f} and \ref{lem_t-exact_func_D(Y)_to_D(X)} that functors $g^!$ and $\iota_g^!$ are $t$-exact.
\end{proof}
\begin{PROP}\label{prop_final_tr_on_C_f}
	For any $(g,h)\in \Dec(f)$, the $T_f$-injective \tr e on $\cC_{f}$ is glued from the $T_h$-injective \tr e on $\cC_h$ and the $T_g$-injective \tr e on $\cC_g$, via recollement $\cC_h \xrightarrow{Lg^*} \cC_f \xrightarrow{\iota_g^*} \cC_g$.
	It coincides with the \tr e glued from dual \tr es on $\dD^b(\bB_g)$ via the strict $\Dec(f)^\textrm{op}$-filtration $g \mapsto Lg^* \cC_h$ of 
	Proposition \ref{prop_right_dual_filt_on_C_f}.
\end{PROP}
\begin{proof}
	The $t$-exactness of $Lg^*$ and $\iota_g^*$, assured by Lemmas \ref{lem_t-exact_func_D(Y)_to_D(X)} and \ref{lem_t-exact_func_D(X)_to_C_f}, implies that the \tr e on $\cC_f$ is glued via recollement w.r.t. subcategory $Lg^*\cC_h$. It follows that it is glued via the strict $\Dec(f)^\textrm{op}$-filtration $g \mapsto Lg^* \cC_h$ from the $T_{f_i}$-injective \tr es on $\cC_{f_i}$. The statement follows from Proposition \ref{prop_dual_t-str_on_W_and_C_f}. 
\end{proof}
It follows form Propositions \ref{prop_inj_t-str_D(X)_glued} and \ref{prop_final_tr_on_C_f} that the $T_{X,f}$-injective \tr e on $\dD^b(X)$ is glued via recollement with respect to subcategory $Lg^* \cC_h$, for any $(g,h) \in \Dec(f)$.

Recall from Corollary \ref{cor_dual_tilting} that $\dD^b(X)$ and $\cC_f$ have tilting relative  objects $S_{X,f}$, respectively $S_f$. Then, in view of Corollary \ref{cor_dual_t-str_on_D_X}, Proposition \ref{prop_inj_t-str_D(X)_glued} implies
\begin{PROP}\label{prop_pro_tr_on_X}
	The $S_{X,f}$-projective \tr e on $\dD^b(X)$ is glued 
	\begin{itemize}
		\item[(i)] from the $S_g$-projective \tr e on $\cC_{g}$ and the $S_{Z,h}$-projective \tr e on $\dD^b(Z)$ via recollement $\dD^b(Z)\xrightarrow{g^!} \dD^b(X) \xrightarrow{\iota_g^!} \cC_g$, for any decomposition $f\colon X \xrightarrow{g}Z \xrightarrow{h} Y$. It is the \tr e glued from the standard \tr e on $\dD^b(Y)$ and the shift by -1 of the standard \tr es on $\dD^b(\bB_g)$ via the strict $\Dec^+(f)^\textrm{op}$-filtration $g\mapsto g^!\dD^b(Z)$ of Remark \ref{rem_filtr_on_D_X}.
		\item[(ii)] from the $S_f$-projective \tr e on $\cC_{f}$ and the standard \tr e on $\dD^b(Y)$ via recollement $\cC_f \xrightarrow{\iota_{f*}} \dD^b(X) \xrightarrow{Rf_*} \dD^b(Y)$.
	\end{itemize}
\end{PROP}
\begin{EXM}\label{exm_1-perverse}
	For the blow-up $f$ of $Y$ along a smooth subscheme $W$, the $S_{X,f}$-projective \tr e on $\dD^b(X)$ is glued from the $S_f$-projective \tr e on $\cC_f$ and the standard \tr e on $\dD^b(Y)$ via recollement with respect to subcategory $\cC_f$. For a single blow-up $f$, we have $S_f =T_f[-1]\otimes f^*(\omega_Y)^{-1}$. Then Proposition \ref{prop_st_t-str_on_W_and_C_f} implies that the $S_f$-projective \tr e on $\cC_f$ is the standard \tr e shifted by -1. It follows that the $S_{X,f}$-projective \tr e on $\dD^b(X)$ is the \tr e of 1-perverse sheaves, with heart ${}^1\textrm{Per}(X/Y)$, as defined by Bridgeland in \cite{Br1}.
\end{EXM} 

Since, for a single blow-up $f_i$, the $S_{f_i}$-projective \tr e differs by the shift by -1 from the $T_{f_i}$-projective, by dualising Proposition \ref{prop_final_tr_on_C_f}, we get
\begin{PROP}\label{prop_proj_tr_on_C_f}
	For any $(g,h)\in \Dec(f)$, the $S_f$-projective \tr e on $\cC_{f}$ is glued from the $S_h$-projective \tr e on $\cC_h$ and the $S_g$-projective \tr e on $\cC_{g}$, via recollement $\cC_h \xrightarrow{g^!} \cC_f \xrightarrow{\iota_g^!} \cC_g$.
	It is the \tr e glued from the shift by -1 of the standard \tr es on $\dD^b(\bB_g)$ via the strict $\Dec(f)^\textrm{op}$-filtration $g \mapsto g^! \cC_h$ of Proposition \ref{prop_right_dual_filt_on_C_f}.
\end{PROP}

\subsection{Tilting of the \tr es in torsion pairs}\label{ssec_Def_f_syst_of_trs}

To a relatively projective morphism $f\colon X \to Y$ of smooth algebraic spaces with dimension of fibers bounded by 1 and an element $g\colon X \to Z$ in $\Dec(f)$, we assign the $S_{X, g}$-projective  \tr e on $\dD^b(X)$. 
We shall show that, for $g,g'\in \Dec(f)$, the corresponding \tr es on $\dD^b(X)$ are related by two tilts, provided $g$ and $g'$ are comparable in $\Dec(f)$. 

More precisely, we consider the partially ordered set $\Dec(f)$ as a category whose objects are elements of $\Dec(f)$ and with exactly one morphism $(g,h)\to (g',h')$ provided $(g,h)\preceq (g',h')$. We shall label objects and morphisms in $\Dec(f)$ with \tr es on $\dD^b(X)$ such that the \tr e $(\dD^b(X)_{\f}^{\lle 0}, \dD^b(X)_{\f}^{\gge 1})$ assigned to $\f\colon (g,h) \to (g',h')$ is smaller than the \tr es $(\dD^b(X)_g^{\lle 0}, \dD^b(X)_g^{\gge 1})$, $(\dD^b(X)_{g'}^{\lle 0}, \dD^b(X)_{g'}^{\gge 1})$  assigned to $(g,h)$ and to $(g',h')$, and differs from them by a single tilt in a torsion pair, i.e. we have 
\begin{align}\label{eqtn_cond_on_t-str}
&\dD^b(X)_\f^{\lle 0} \subset \dD^b(X)_g^{\lle 0} \subset \dD^b(X)_\f^{\lle 1},&  &\dD^b(X)_\f^{\lle 0} \subset \dD^b(X)_{g'}^{\lle 0} \subset \dD^b(X)_\f^{\lle 1}.&
\end{align}

We assign to $g\in {\rm Dec}(f)$ the $S_{X,g}$-projective \tr e.

Let $\f \colon g_0 \to g_1$ be a morphism in $\Dec(f)$ and let $g_1 \colon X \xrightarrow{g_0} Z_0 \xrightarrow{\f_1} Z_1$ be the corresponding decomposition. We define the \tr e $(\dD^b(X)_{\f}^{\lle 0}, \dD^b(X)_{\f}^{\gge 1})$ on $\dD^b(X)$ to be glued via the filtration
\begin{equation}\label{eqtn_filt_for_f} 
g_0^! \cC_{\f_1} \subset g_0^!\dD^b(Z_0) \subset \dD^b(X)
\end{equation} 
from the standard \tr es on $\cC_{\f_1}$ and $\dD^b(Z_1) \simeq \dD^b(Z_0)/\cC_{\f_1}$ and the $S_{g_0}$-projective \tr e on $\cC_{g_0} \simeq \dD^b(X)/\dD^b(Z_0)$. Note that the \tr e on $\dD^b(Z_0)$ is by definition the \emph{0-perverse} \tr e.
\begin{EXM}
	For a single blow-up $f\colon X \to Y$, the poset $\Dec(f)$ has two elements $(\Id_X,f)$ and $(f, \Id_Y)$. In view of Remark \ref{rem_line_bundle_pro} and Example \ref{exm_1-perverse}, the corresponding \tr es on $\dD^b(X)$ are the standard one and the 1-perverse \tr e for the map $f$. The \tr e assigned to $\f \colon (\Id_X,f) \to (f,\Id_Y)$ is the 0-perverse \tr e for $f$. 
\end{EXM}

We denote by $(\cC_f^{\lle 0}, \cC_f^{\gge 1})$ the standard \tr e on $\cC_f$.

\begin{PROP}\label{prop_t-str_on_C_f_differ_by_tilt}
	We have inclusions $\cC_f^{\lle 0} \subset \cC_{f\, \textrm{pro }S_f}^{\lle 0} \subset \cC_f^{\lle 1}$.
\end{PROP}
\begin{proof}
	We proceed by induction on the length of decomposition of $f$ into a sequence of blow-ups. If $f$ is the blow-up of $Y$ along a smooth $W$, then, in view of Proposition \ref{prop_st_t-str_on_W_and_C_f} and the fact that $S_f=T_{f}[-1]\otimes f^*(\omega_Y)^{-1}$, the $S_f$-projective \tr e is just the standard \tr e on $\cC_f$ shifted by -1, i.e. we have $\cC_{f\, \textrm{pro }S_f}^{\lle 0} =\cC_f^{\lle 1}$.
	
	Now let  $f$ be a morphism which can be decomposed into $n$ consecutive blow-ups and let $(g,h)\in \Dec(f)$ with $g\colon X \to Z$ the blow-up of $Z$ along a smooth $W$. Take $E\in \cC_f^{\lle 0}$. In view of Proposition \ref{prop_proj_tr_on_C_f} and formula (\ref{eqtn_aisles_of_glued_tr}) for the negative aisle of the glued \tr e, $E$ lies in $\cC_{f\, \textrm{pro }S_f}^{\lle 0}$ if and only if $Rg_* E\in \cC_{h\, \textrm{pro }S_h}^{\lle 0}$ and $\iota_g^!E \in \cC_{g\, \textrm{pro }S_g}^{\lle 0}$. By Proposition \ref{prop_g_is_exact}, $Rg_*E \in \cC_h^{\lle 0}$. The inductive hypothesis for $h$ implies that $Rg_*E \in \cC_{h\, \textrm{pro }S_h}^{\lle 0}$. Moreover, since $Lg^*$ is right $t$-exact so is $g^!(-) \simeq Lg^*(-)\otimes \omega_g$. Hence, we have $g^! Rg_* E \in \dD^b(X)^{\lle 0}$. It follows, by considering the long exact sequence of the cohomology sheaves for the triangle 
	\begin{align}\label{eqtn_iota_g}
	\iota_{g*}\iota_g^!E \to E \to g^!Rg_* E\to \iota_{g*}\iota_g^!E[1],
	\end{align}
	that $\iota_g^!E \in \cC_g^{\lle 1}= \cC_{g\, \textrm{pro }S_g}^{\lle 0}$. 
			
	Now let $E$ be an object in $\cC_{f\,\textrm{pro }S_f}^{\lle 0}$. Then $Rg_*E \in \cC_{h\,\textrm{pro }S_h}^{\lle 0} \subset \cC_h^{\lle 1}$, by the inductive hypothesis applied to $h$. Right $t$-exactness of $g^!$ implies that $g^!Rg_*E \in \cC_f^{\lle 1}$. The gluing property for the $S_f$-projective \tr e on $\cC_f$, Proposition \ref{prop_proj_tr_on_C_f},   implies that $\iota_g^!E \in \cC_{g\,\textrm{pro }S_g}^{\lle 0}$. The first paragraph of the proof assures that $ \cC_{g\,\textrm{pro }S_g}^{\lle 0} = \cC_g^{\lle 1}$. Since $\iota_{g*}$ is $t$-exact for the standard \tr e, we have $\iota_{g*}\iota_g^!E \in \cC_f^{\lle 1}$. 
	The long exact sequence of the cohomology sheaves for (\ref{eqtn_iota_g}), implies  that $E\in \cC_f^{\lle 1}$, i.e. $\cC_{f\, \textrm{pro }S_f}^{\lle 0} \subset \cC_f^{\lle 1}$.
\end{proof}
In order to show that (\ref{eqtn_cond_on_t-str}) is verified, we prove
\begin{LEM}\label{lem_gluing_and_tilting}
	Let $\dD$ be a triangulated category with a recollement (\ref{eqtn_recollement}) and $(\dD_0^{\lle_1 0}, \dD_0^{\gge_1 1})$, $(\dD_0^{\lle_2 0}, \dD_0^{\gge_2 1})$ \tr es on $\dD_0$ with $\dD_0^{\lle_1 0} \subset \dD_0^{\lle_2 0}\subset \dD_0^{\lle_1 1}$, $(\dD_1^{\lle_1 0}, \dD_1^{\gge_1 1})$, $(\dD_1^{\lle_2 0}, \dD_1^{\gge_2 1})$ \tr es on $\dD_1$ with $\dD_1^{\lle_1 0} \subset \dD_1^{\lle_2 0}\subset \dD_1^{\lle_1 1}$.
	Let $(\dD^{\lle_i 0}, \dD^{\gge_i 1})$ be the \tr e on $\dD$ glued from $(\dD_0^{\lle_i 0}, \dD_1^{\gge_i 1})$ and $(\dD_1^{\lle_i 0}, \dD_1^{\gge_i 1})$, for $i=1,2$. Then $\dD^{\lle_1 0} \subset \dD^{\lle_2 0} \subset \dD^{\lle_1 1}$.
\end{LEM}
\begin{proof}
	The inclusions immediately follow from the definition of the negative aisle of the glued \tr e.
\end{proof}

\begin{PROP}\label{prop_gluing_and_tilting_via_poset}
	Let $\lL$ be a finite distributive lattice and $\dD$ a triangulated category with a strict admissible  $\lL$-filtration. For every join-prime $s\in \lL$, assume the category $\dD_s^o$ has \tr es $({\dD_s^o}^{\lle_1 0}, {\dD_s^o}^{\gge_1 1})$ $({\dD_s^o}^{\lle_2 0}, {\dD_s^o}^{\gge_2 1})$ with ${\dD_s^o}^{\lle_1 0} \subset {\dD_s^o}^{\lle_2 0} \subset {\dD_s^o}^{\lle_1 1}$. Let $(\dD^{\lle_1 0}, \dD^{\gge_1 1})$, respectively $(\dD^{\lle_2 0}, \dD^{\gge_2 1})$, be the \tr e on $\dD$ glued via the $\lL$-filtration from the \tr es $({\dD_s^o}^{\lle_1 0}, {\dD_s^o}^{\gge_1 1})$, respectively $({\dD_s^o}^{\lle_2 0}, {\dD_s^o}^{\gge_2 1})$. Then $\dD^{\lle_1 0}\subset \dD^{\lle_2 0} \subset \dD^{\lle_1 1}$. 
\end{PROP}
\begin{proof}
	Since the \tr e on $\dD$ is glued via any admissible filtration on $\dD$ compatible with the order on $\lL$ (Remark \ref{rem_full_order}) the statement follows by induction from Lemma \ref{lem_gluing_and_tilting}.
\end{proof}

\begin{PROP}\label{prop_Per_0_on_the_edge}
	For $\f\colon g_0 \to g_1$ in $\Dec(f)$, we have embeddings (\ref{eqtn_cond_on_t-str}).
\end{PROP}
\begin{proof}
	The \tr e corresponding to $g_0$ is glued via the filtration $g_0^! \dD^b(Z_0)\subset \dD^b(X)$ from the standard \tr e on $\dD^b(Z_0)$ and the $S_{g_0}$-projective \tr e on $\cC_{g_0}$ by Proposition \ref{prop_pro_tr_on_X}(i). It is proved in \cite{VdB} that  $\dD^b(Z_0)_{{}^0\textrm{Per}(Z_0/Z_1)}^{\lle 0} \subset \dD^b(Z_0)^{\lle 0} \subset \dD^b(Z_0)_{{}^0\textrm{Per}(Z_0/Z_1)}^{\lle 1}$, hence, by Lemma \ref{lem_gluing_and_tilting}, $\dD^b(X)_{\f}^{\lle 0} \subset \dD^b(X)_{g_0}^{\lle 0} \subset \dD^b(X)_{\f}^{\lle 1}$.
	
	The \tr e corresponding to $g_1$ is glued via the filtration (\ref{eqtn_filt_for_f}) from the $S_{\f_1}$-projective \tr e on $\cC_{\f_1}$, the standard \tr e on $\dD^b(Z_1)$ and the $S_{g_0}$-projective \tr e on $\cC_{g_0}$, see Proposition \ref{prop_pro_tr_on_X}. By Proposition \ref{prop_t-str_on_C_f_differ_by_tilt},  $\cC_{\f_1}^{\lle 0} \subset \cC_{\f_1\,\textrm{pro }S_{\f_1}}^{\lle 0} \subset \cC_{\f_1}^{\lle 1}$. Then, Proposition \ref{prop_gluing_and_tilting_via_poset} implies that $\dD^b(X)_{\f}^{\lle 0} \subset \dD^b(X)_{g_1}^{\lle 0} \subset \dD^b(X)_{\f}^{\lle 1}$.
\end{proof}

\begin{PROP}
	For a morphism $\f \colon g_0 \to g_1$ in $\Dec(f)$, the \tr e $(\dD^b(X)_{\f}^{\lle 0}, \dD^b(X)_{\f}^{\gge 1})$ is the naive intersection of \tr es $(\dD^b(X)_{g_0}^{\lle 0}, \dD^b(X)_{g_0}^{\gge 1})$ and $(\dD^b(X)_{g_1}^{\lle 0}, \dD^b(X)_{g_1}^{\gge 1})$ (see \cite{Bon2}), i.e. we have $\dD^b(X)_{\f}^{\lle 0} = \dD^b(X)_{g_0}^{\lle 0} \cap \dD^b(X)_{g_1}^{\lle 0}$.
\end{PROP}
\begin{proof} 
	The inclusion of $\dD^b(X)_{\f}^{\lle 0}$ into $\dD^b(X)_{g_0}^{\lle 0} \cap \dD^b(X)_{g_1}^{\lle 0}$ follows from Proposition \ref{prop_Per_0_on_the_edge}. For $F \in\dD^b(X)_{g_0}^{\lle 0} \cap \dD^b(X)_{g_1}^{\lle 0}$, Proposition \ref{prop_pro_tr_on_X}(i) implies that $\iota_{g_0}^!F \in \cC_{g_0\,\textrm{pro }S_{g_0}}^{\lle 0}$. Hence $F \in \dD^b(X)^{\lle 0}_{\f}$ if $Rg_{0*}F \in \dD^b(Z_0)^{\lle 0}_{{}^0\textrm{Per}(Z_0/Z_1)}$.
	
	Proposition \ref{prop_pro_tr_on_X}(i) implies $Rg_{0*}F \in \dD^b(Z_0)^{\lle 0}$. Since $g_1 = \f_1 \circ g_0$, $Rg_{1*}F = R\f_{1*}Rg_{0*}F \in \dD^b(Z_1)^{\lle 0}$ and 
	$L\f_1^*R\f_{1*}Rg_{0*}F = L\f_1^* Rg_{1*}F \in \dD^b(Z_0)^{\lle 0}$. Then triangle of functors (\ref{eqtn_def_of_iota*}) for $f =\f_1$ applied to object $Rg_{0*}F$ implies that $\iota_{\f_1*}\iota_{\f_1}^*Rg_{0*}F \in \dD^b(Z_0)^{\lle 0}$. Since $\iota_{\f_1*}$ is $t$-exact and has no kernel, $\iota_{\f_1}^* Rg_{0*}F \in \cC_{\f_1}^{\lle 0}$. Together with $R\f_{1*}Rg_{0*}F \in \dD^b(Z_1)^{\lle 0}$ and the fact that the 0-perverse \tr e is glued from the standard \tr es via recollement w.r.t. subcategory $\cC_{\f_1}$, this implies that $Rg_{0*} F \in \dD^b(Z_0)_{{}^0\textrm{Per}(Z_0/Z_1)}^{\lle 0}$.
\end{proof}

\section{\textbf{Contractions as relative moduli of simple quotients of $\oO_X$}}\label{sec_Z_as_moduli}

Let $\dD$ be a triangulated category admitting a recollement (\ref{eqtn_recollement}) and let $(\dD_0^{\lle 0}, \dD_0^{\gge 1})$, $(\dD_1^{\lle 0}, \dD_1^{\gge 1})$ be \tr es on $\dD_0$, respectively $\dD_1$, with hearts $\aA_0$, respectively $\aA_1$. Let $(\dD^{\lle 0}, \dD^{\gge 1})$ be the \tr e on $\dD$ glued from these \tr es via recollement (\ref{eqtn_recollement}). We denote by $\aA$ its heart. We recall after \cite{BBD} the description of simple objects in $\aA$.

For an object $A_1 \in \aA_1$, there exists a canonical morphism $j_!A_1 \to j_* A_1$, which corresponds by adjunction to the isomorphism $j^*j_!A \simeq A$, see Appendix \ref{sec_canonical_morp}. Since $j_!A_1 \in \dD^{\lle 0}$ and $j_* A_1 \in \dD^{\gge 0}$, we get a morphism $\alpha_{A_1}\colon \hH^0(j_! A_1) \to \hH^0(j_* A_1)$. The image of $\alpha_{A_1}$ in $\hH^0(j_* A_1)$ defines the functor $j_{!*} \colon \aA_1 \to \aA$. 

An object of an abelian category is said to be \emph{simple} if it has no proper subobjects.
\begin{PROP}\cite[Proposition 1.4.26]{BBD}\label{prop_simple_object_in_glued} 
	A simple object in $\aA$ is isomorphic either to $i_* s_0$ or to $j_{!*}s_1$, for simple object $s_0 \in \aA_0$ and $s_1 \in \aA_1$.
\end{PROP}
\begin{PROP}\label{prop_simple_in_B_f}
	Assume that $\aA$ is a heart of a \tr e on $\dD$ glued from \tr es on $\dD_0$ and $\dD_1$ via recollements with respect to subcategories $i_*\dD_0$ and $j_* \dD_1$. Then a simple object in $\aA$ is isomorphic either to $i_*s_0$ or to $j_* s_1$, for some $s_0\in \aA_0$, $s_1\in \aA_1$ simple.
\end{PROP}
\begin{proof}
	Proposition \ref{prop_simple_object_in_glued} implies that $i_*s_0$ and the image of $\hH^0(j_!s_1) \to \hH^0(j_*s_1)$ are the only isomorphism classes of simple objects in $\aA$. Since the \tr e is glued with respect to subcategory $j_* \dD_1$, functor $j_*$ is $t$-exact and $\hH^0(j_* s_1) \simeq j_* s_1$. Moreover, Proposition \ref{prop_simple_object_in_glued} implies that $j_* s_1$ is simple in $\aA$, hence $j_{!*}s_1\simeq j_*s_1$.
\end{proof}
Consider a relatively projective birational morphism $f\colon X \to Y$  of smooth algebraic spaces with dimension of fibers bounded by one. We denote by $\mathscr{B}_f$ the heart of the $S_{X,f}$-projective \tr e. Note that the structure sheaf of $X$ is a direct summand of $S_{X,f}$, hence $\oO_X \in \mathscr{B}_f$.

\begin{COR}\label{cor_simpl_quot_of_O}
		A simple object in $\mathscr{B}_f$ is isomorphic either to $f^!(\oO_y)$, for a closed point $y \in Y$, or to $\iota_{f*}s$, for a simple object $s$ in the heart of the $S_f$-projective \tr e.
		Objects $f^!\oO_y$ are the only simple quotients of $\oO_X$ in $\mathscr{B}_f$.
\end{COR}
\begin{proof} The description of simple objects follows from Proposition \ref{prop_pro_tr_on_X} and Proposition \ref{prop_simple_in_B_f}.
		
	The morphism $ \oO_X \to f^! \oO_y$ which corresponds by adjunction to $\oO_Y \to \oO_y$  is an epimorphism, because it is non-zero and $f^!\oO_y$ is simple in $\mathscr{B}_f$. 
	The other simple objects are not quotients of $\oO_X$ in $\mathscr{B}_f$, because $\Hom_X(\oO_X,\iota_{f*}s_f) \simeq \Hom_Y(\oO_Y, Rf_* \iota_{f*}s_f) =0$. 
\end{proof}

Let $g\colon X \to Z$ be an element of $\Dec(f)$. 
For a $Y$-scheme $S\to Y$ and $s\in S$ a closed point, we consider the following diagram with fiber squares:
\begin{equation}\label{eqtn_diagram_for_Y-scheme}
\xymatrix{X_s \ar[r]^{i_s} \ar[d]_{g_s} & X \times S \ar[d]^{g_S} \ar[r]^{p_X} & X \ar[d]^{g} \\ Z_s \ar[r]^{j_s} \ar[d]_q & Z \times S \ar[d]^{\pi} \ar[r]^{p_Z} & Z \\
	s \ar[r]^{k_s} & S &}
\end{equation}
with  $Z_s\simeq Z$ 
and $X_s\simeq X$. For $E \in \dD^b(X\times S)$, we put
$E_s:=Li_s^*E$ and, for $F \in \dD^b(Z \times S)$, we put $F_s:=Lj_s^* F$.

We consider the following functor of points $Y\textrm{-Sch}^{\textrm{op}}\to {\rm Sets}$, which makes rigorous the idea of families of simple objects in $\mathscr{B}_g$ that are quotients of $\oO_X$ (compare it with the definition due to Bridgeland of point objects in his perverse \tr e for a flopping contraction \cite{Br1}):
\begin{align*}
\eE(S \to Y) =&\{(E,\psi)\,|\, E \in \dD^b(X \times S),\, \psi \colon \oO_{X\times S}\to E,\, \textrm{Supp}(E) \subset X \times_Y S,\,  \forall\, s\in S,\\& Li_s^*\psi\colon \oO_{X_s} \to E_s\, \textrm{is an epimorphism onto a simple object in}\, \mathscr{B}_g\}/\sim,
\end{align*}
where $(E,\psi) \sim (E', \psi')$ if there exists an isomorphism $\kappa \colon E \xrightarrow{\simeq} E'$ such that $\kappa \psi = \psi'$. For a morphism of $Y$-schemes $\sigma \colon S_1 \to S_2$, map $\eE(\sigma)\colon \eE(S_2 \to Y) \to \eE(S_1\to Y)$ is given by $\eE(\sigma)(E,\psi) = ((\Id_X \times \sigma)^*E, (\Id_X \times\sigma)^* \psi)$.

It will be accompanied with another functor $Y\textrm{-Sch}^{\textrm{op}}\to {\rm Sets}$:
\begin{align*}   
\fF(S \to Y) =&\{(F,\xi)\,|\, F \in \dD^b(Z \times S),\, \xi \colon \oO_{Z\times S} \to F,\, \textrm{Supp}(F) \subset Z \times_Y S,\,  \forall\, s \in S,\\& Lj_s^*\xi \colon \oO_{Z_s} \to F_s\,\textrm{is an epimorphism onto a simple object in}\, \Coh(Z_s)\}/\sim,
\end{align*}
with a similar equivalence relation $(F, \xi)\sim (F',\xi')$.

We shall show that functors $\fF$ and $\eE$ are isomorphic and conclude that, for any $Y$-scheme $S$, we have $\eE(S \to Y) = \Hom_{Y-\textrm{Sch}}(S,Z)$.

For a $Y$-scheme $S$, define $\mu_S\colon \fF(S \to Y) \to \eE(S \to Y)$ via
\begin{align*}
\mu_S(F,\xi) =(g_S^!(F),\, \oO_{X\times S} \xrightarrow{\delta^S_{\oO_{Z \times S}}} g_S^!(\oO_{Z\times S}) \xrightarrow{g_S^!(\xi)} g_S^!(F)),
\end{align*}
where $\delta^S = \f_{Rg_{S*}} \colon Lg_S^* \to g_S^!$ is the functorial morphism in formula (\ref{eqtn_def_of_phi_f}).

\begin{LEM}\label{lem_mu_is_morphism}
	System of maps $\mu_S$ defines a morphism of functors $\mu \colon \fF \to \eE$.
\end{LEM}
\begin{proof}
	We need to check that $\mu_S(F, \xi)$ is indeed an element of $\eE(S\to Y)$ and that $\mu$ is a natural transformation, i.e. for any $\sigma \colon S_1 \to S_2$, we have $\eE(\sigma)\circ \mu = \mu \circ \fF(\sigma)$. The latter follows from the fact that, for the commutative diagram defined by $\sigma$: 
	\[
	\xymatrix{X\times S_1 \ar[rr]^{\Id_X\times \sigma} \ar[d]_{g_{S_1}} && X \times S_2 \ar[d]^{g_{S_2}} \\ Z \times S_1 \ar[rr]^{\Id_Z\times \sigma} && Z \times S_2,}
	\]
	we have     
	\begin{equation}\label{eqtn_stars_commute}
	(\Id_X\times \sigma)^* \circ g_{S_2}^! \simeq g_{S_1}^! \circ (\Id_Z\times \sigma)^*.
	\end{equation}
	
	Now let $(F,\xi) \in \fF(S \to Y)$. Equality $g_S = g \times \Id_S$ implies that $g_S^!(\oO_{Z\times S}) \simeq  Lp_X^*(\omega_g)$ and $g_S^!(-)\simeq Lg_S^*(-)\otimes g_S^!(\oO_{Z\times S})$, hence $\mu_S(F) = g_S^!(F) \simeq Lg_S^*(F) \otimes Lp_X^*(\omega_g)$ is supported on the preimage of $Z\times_YS$, i.e. on $X \times_Y S$.
	
	Let $s\in S$ be a closed point. As $i_s = \Id_X \times \sigma_s$ and $j_s = \Id_Z \times \sigma_s$, for the embedding $\sigma_s \colon s \to S$, formula (\ref{eqtn_stars_commute}) implies that $Li_s^*g_S^!(E) \simeq g_s^!Lj_s^*(E) \simeq g_s^!(\oO_z)$, for some $z\in Z_s$, i.e. it is a simple object in $\bB_{g_s}$, see Corollary \ref{cor_simpl_quot_of_O}. 
	
	It remains to check that $Li_s^*\mu_S(\xi)\neq 0$. Since $(F,\xi) \in \fF(S\to Y)$, $Lj_s^*(\xi)\neq 0 $. As $Lg_s^*$ is fully-faithful, morphism $Lg_s^*Lj_s^*(\xi)=Li_s^*Lg_S^*(\xi) \colon \oO_{X_s} \to Li_s^*Lg_S^*(F)$ is non-zero. It follows that $Li_s^*g_S^!(\xi) \simeq Li_s^* g_S^*(\xi) \otimes \omega_{g_s} \colon \omega_{g_s} \to Lj_s^*g_S^!(F)$ is non-zero. 
	By adjunction, $\Hom_{X_s}(\omega_{g_s}|_{D_{g_s}}, g_s^!(\oO_z)) \simeq \Hom_{Z_s}(Rg_{s*} \omega_{g_s}|_{D_{g_s}}, \oO_z) =0$. Then applying $\Hom(-,Lj_s^*g_s^!(F))$ to triangle (\ref{eqtn_dis_of_f}) for $g=g_s$ implies the embedding $\Hom_{X_s}(\omega_{g_s}, Lj_s^*g_s^!(F)) \xrightarrow{(-)\circ \delta^s} \Hom_{X_s}(\oO_{X_s}, Lj_s^*g_s^!(F))$, where $\delta^s \colon \oO_{X_s} \to \omega_{g_s}$ is the canonical morphism. It follows that $Li_s^* \mu_S(\xi)=Li_s^*g_S^!(\xi)\circ \delta $ is non-zero.
\end{proof}

\begin{PROP}\label{prop_families_on_X_and_Z}
	Morphism $\mu\colon \fF \to \eE$ is an isomorphism.
\end{PROP}
\begin{proof}
	It suffices to check that $\mu_S \colon \fF(S\to Y) \to \eE(S \to Y)$ is a bijection, for any $Y$-scheme $S$.
	
	Define $\nu_S \colon \eE(S \to Y) \to \fF(S \to Y)$ via
	\begin{align*}
		\nu_S(E, \psi) = (Rg_{S*}(E),\,\oO_{Z\times S} \xrightarrow{\eta} Rg_{S*}\oO_{X\times S} \xrightarrow{Rg_{S*}\psi} Rg_{S*}E),
	\end{align*}
	where $\eta$ is the $Lg_S^* \dashv Rg_{S*}$ adjunction unit. We check that $\nu_S$ is well-defined and inverse to $\mu_S$. 

	Let $(E, \psi)$ be an element of $\eE(S \to Y)$. Then $Rg_{S*}(E)$ is supported on the image of $X \times_Y S$, i.e. on $Z \times_Y S$. To calculate $Lj_s^*Rg_{S*}(E)$, we note that 
	we have isomorphisms of functors
	\begin{equation}\label{eqtn_iso_of_func_2} 
	\begin{aligned}
		& Rj_{s*}Lj_s^*Rg_{S*}(-) \simeq Rg_{S*}(-) \otimes^L \oO_{Z_s} \simeq Rg_{S*}(- \otimes^L \oO_{X_s}) \simeq \\
		&Rg_{S*}Ri_{s*}Li_s^*(-) \simeq Rj_{s*}Rg_{s*}Li_s^*(-).
	\end{aligned}
	\end{equation} 
	Since Corollary \ref{cor_simpl_quot_of_O} implies that $Li_s^*E \simeq g_s^!(\oO_z)$, for some $z \in Z$, we have $Rj_{s*}Lj_s^* Rg_{S*}(E) \simeq Rj_{s*}Rg_{s*}Li_s^*(E) \simeq Rj_{s*} \oO_z$. Since $Rj_{s*}$ is $t$-exact and it is an equivalence with the image when restricted to the heart of the standard \tr e, we have $Lj_s^* Rg_{S*}(E) \simeq \oO_z$. 
	
	In order to check that $Lj_s^* Rg_{S*}(\psi) \neq 0$, it suffices to show that $Rj_{s*}Lj_s^*Rg_{S*}(\psi) \simeq Rj_{s*}Rg_{s*}Li_s^*(\psi) \neq 0$. By definition of $\eE$, $Li_s^*(\psi) \colon \oO_{X_s} \to g_s^!(\oO_z)$ is non-zero. By adjunction, $Li_s^*(\psi)$ corresponds to a non-zero morphism $Rg_{s*}\oO_{X_s} \xrightarrow{Rg_{s*}Li_s^*(\psi)} Rg_{s*}g_s^!(\oO_z) \xrightarrow{\varepsilon} \oO_z$, where  $\varepsilon$ is for the adjunction counit. It follows that $Rg_{s*}Li_s^*(\psi)\neq 0$.
	
	Since $Rp_{Z*}Rj_{s*}$ is the identity, functor $Rj_{s*}$ is faithful. Therefore, $Rj_{s*}Rg_{s*}Li_s^*(\psi) \neq 0$, i.e. $\nu_S$ is well-defined.

	The composite $\oO_{Z \times S} \xrightarrow{\eta} Rg_{S*}Lg_S^*(\oO_{Z \times S}) \xrightarrow{Rg_{S*}\delta^S} Rg_{S*}g_S^!(\oO_{Z\times S}) \xrightarrow{\varepsilon} \oO_{Z\times S}$ is the identity morphism, see (\ref{eqtn_property_of_f_F}). Thus, for $(F,\xi)\in \fF(S\to Y)$, the adjunction counit $\varepsilon\colon Rg_{S*}g_S^! \to \Id$ provides an isomorphism $\nu_S\mu_S(F,\xi) \xrightarrow{\simeq} (F,\xi)$ in view of the following commutative diagram:
	\[
	\xymatrix{& & &\oO_{Z\times S} \ar[rr]^\xi && F\\
		\oO_{Z\times S} \ar[r]|(0.4)\eta \ar[rrru]^{\Id}  &Rg_{S*} \oO_{X\times S} \ar[rr]^{Rg_{S*}\delta^S} && Rg_{S*}g_S^!(\oO_{Z\times S}) \ar[u]^{\varepsilon} \ar[rr]^{Rg_{S*}g_S^!(\xi)} && Rg_{S*}g_S^!(F) \ar[u]_{\varepsilon} }
	\]
	Lemma \ref{lem_another_unit} implies that, for any $(E,\psi) \in \eE(S\to Y)$, the adjunction unit $\eta'\colon \Id \to g_S^!Rg_{S*}$ provides a morphism $(E,\psi) \to\mu_S\nu_S(E,\psi)$ in view of the diagram:
	\[
	\xymatrix{\oO_{X\times S} \ar[r]^(0.4){\delta^S} & g_S^!(\oO_{Z\times S}) \ar[r]^(0.45){g_S^!(\eta)} & g_S^!Rg_{S*}\oO_{X\times S} \ar[rr]^(0.55){g_S^!Rg_{S*} \psi} && g_S^!Rg_{S*} E \\ && \oO_{X\times S} \ar[llu]^{\Id} \ar[rr]^{\psi} \ar[u]^{\eta'} && E \ar[u]_{\eta'}}
	\] 
	To prove that $\eta'\colon E \to g_S^!Rg_{S*}E$ is an isomorphism, we check that it is an isomorphism when restricted to the fiber $X_s$ of $\pi \circ g_S$ over any closed point $s\in S$. It is enough because of the support condition on $E$. Since both $\mu_S$ and $\nu_S$ are well-defined, $\mu_S\nu_S(E,\psi) \in \eE(S \to Y)$. In particular, $Li_s^* g_S^!Rg_{S*}E \simeq g_s^!(\oO_{z'})$, hence $Li_s^*\eta'$ is a morphism $g_s^!(\oO_z) \to g_s^!(\oO_{z'})$ and it is an isomorphism if and only if its is non-zero (because $g_s^!$ is fully faithful). Morphism $Rg_{S*}\eta'$ is the identity, hence so is $Lj_s^*Rg_{S*}\eta'$. As functor $Rj_{s*}$ is faithful, $Rj_{s*}Lj_s^*Rg_{S*}\eta' \neq 0$. In view of isomorphism (\ref{eqtn_iso_of_func_2}), it implies that $Rj_{s*}Rg_{s*}Li_s^*\eta'\neq 0$, i.e. $Li_s^*\eta' \neq 0$ is an isomorphism. 
\end{proof}

\begin{THM}\label{thm_Z_is_moduli_space}
	For an element $g\colon X \to Z$ of $\Dec(f)$, $Z$ is the fine moduli space of simple quotients of $\oO_X$ in $\mathscr{B}_g$, which have the form $g^!(\oO_z)$, where $z$ runs over the set of closed points in $Z$. More precisely, we have an isomorphism of functors $\eE(-) \simeq \Hom(-,Z) \colon Y\textrm{-Sch}^{\textrm{op}} \to \textrm{Sets}$.
\end{THM}
\begin{proof}
	In view of Proposition \ref{prop_families_on_X_and_Z} it suffices to check that functors $\fF(-)$ and $\Hom(-,Z)$ are isomorphic. By assigning to $\gamma \in \Hom_{Y\textrm{-Sch}}(S,Z)$ the structure sheaf of the graph of $\gamma$ we define a morphism $\Hom_{Y\textrm{-Sch}}(-,Z)\to \fF(-)$ of functors $Y\textrm{-Sch}^\textrm{op} \to \textrm{Sets}$. Let us check that it is an isomorphism.
	
	For an element $(F, \xi) \in \fF(S \to Y)$, $F$ is a sheaf on $Z \times S$, flat over $S$, because the restriction to every fiber is a pure sheaf, see \cite[Lemma 4.3]{Br6}. As $\xi$ is surjective on every fiber of $\pi$, it is surjective. It follows that $F \simeq \oO_{\Gamma}$, for some $\Gamma \subset Z \times S$, flat over $S$. Since the restrictions of $\oO_\Gamma$ to the fibers of $\pi$ are simple objects, i.e. skyscrapers, $\pi|_\Gamma \colon \Gamma \to S$ is an isomorphism. Hence, $\Gamma$ is a graph of a morphism $\gamma \colon S\to Z$.	
	
	The form $g^!(\oO_z)$ of the simple quotients follows from Corollary \ref{cor_simpl_quot_of_O}.
\end{proof}

\appendix

\section{\textbf{Mutations over admissible subcategories revisited}}\label{sec_canonical_morp}

The following is a write-up on the mutations over admissible subcategories \cite{B} and iteration over them in the non-triangulated set-up. We use the description of a relevant morphism between adjoint functors given below in Sections \ref{sec_local_gen} and \ref{sec_Z_as_moduli} of the main body of the paper.

Let $\cC_0$ $\cC_2$ be categories and $F\colon \cC_0\to \cC_2$ a functor that admits left and right adjoints, $F^*$,  $F^!$. Let $\eta \colon \textrm{Id} \to F\,F^*$ be the unit and $\varepsilon \colon F\, F^! \to \Id$ the counit of the adjunctions. 

If $F^*$ and $F^!$ are fully faithful, then both $\eta$ and $\varepsilon$ are isomorphisms. Let 
\begin{equation}\label{eqtn_def_of_phi_f}
\f_F \colon F^* \to F^!
\end{equation} 
be the morphism of functors that corresponds by the adjunction isomorphism $\varepsilon\circ F(-)\colon \Hom(F^*, F^!) \xrightarrow{\simeq} \Hom(FF^*, \Id)$ to $\eta^{-1} \colon FF^* \to \Id$, i.e. $\varepsilon\circ (F\f_F) = \eta^{-1}$. Hence
\begin{equation}\label{eqtn_property_of_f_F}
\varepsilon \circ (F\f_F) \circ \eta =\textrm{id}_\textrm{Id}.
\end{equation}
Since $\varepsilon\circ F(-)$ is an isomorphism, formula
(\ref{eqtn_property_of_f_F}) is a defining property for $\f_F$.

Note that $(F\f_F) \circ \eta\colon \Id \to FF^!$ is the morphism that corresponds to $\f_F$ by $F^*\dashv F$ adjunction. Since (\ref{eqtn_property_of_f_F}) is a defining property for $\f_F$, by exchanging roles of $\eta$ and $\varepsilon$, we see that $\f_F$ can equivalently be defined as the morphism $F^* \to F^!$ that corresponds to $\varepsilon^{-1}$ under the $F^*\dashv F$ adjunction.  

Denote by $\eta'\colon \Id \to F^!F$ the $F^! \dashv F$ adjunction unit.
\begin{LEM}\label{lem_another_unit}
	We have $F^!\eta \circ \f_F = \eta' {F^*} \colon  F^*\to F^!FF^*$.
\end{LEM}
\begin{proof}
	We need show that $F^!\eta \circ\f_F$ corresponds under the isomorphism $\Hom(F^*, F^!FF^*) \simeq \Hom(FF^*, FF^*)$ to the identity, i.e. that $\varepsilon {FF^*}\circ FF^!\eta \circ  F\f_F = \textrm{id}_{FF^*}$. 
	
	Consider the diagram 
	\[
	\xymatrix{FF^! \ar[rr]^{FF^!\eta} && FF^!FF^* \ar[rr]^{\varepsilon FF^*} && FF^* \\
		FF^* \ar[rr]^{FF^*\eta} \ar[u]^{F\f_F} && FF^*FF^* \ar[u]|{F\f_F {FF^*}} && FF^* \ar[u]_{\Id} \ar[ll]_{\eta FF^*} }
	\]
	Since $\f_F$ is a morphism of functors, the left square commutes. Formula (\ref{eqtn_property_of_f_F}) implies the commutativity of the right square.
	Note that $\eta FF^*$ is an isomorphism, hence the diagram implies that to show $\varepsilon {FF^*}\circ FF^!\eta \circ  F\f_F = \textrm{id}_{FF^*}$ is equivalent to checking that $FF^*\eta$ is inverse to $\eta FF^*$. This is implied by the standard commutative diagram:
	\[
	\xymatrix{& FF^*FF^* \ar[d]|{F\varepsilon' F^*} &\\
		FF^* \ar[ur]^{\eta FF^*} \ar[r]_{\textrm{id}} & FF^* & FF^* \ar[ul]_{FF^*\eta} \ar[l]^{\textrm{id}}, }
	\]
	where $\varepsilon'\colon FF^* \to \Id$ is the $F^*\dashv F$ adjunction counit, because $\varepsilon'$ is an isomorphism.
\end{proof}
Assume now that $F = H G$, for functors $G\colon \cC_0 \to \cC_1$, $H\colon \cC_1 \to \cC_2$, which admit fully faithful left and right adjoint functors. Then functors $F^*$, $F^!$, left and right adjoin to $F$ are also fully faithful. Denoting by $\eta_F \colon \Id \to FF^*$, $\varepsilon_F\colon FF^!\to \Id$, $\eta_G \colon \Id \to GG^*$, $\varepsilon_G\colon GG^! \to \Id$, $\eta_H \colon \Id \to HH^*$, $\varepsilon_H \colon HH^! \to \Id$ the units and counits, we have presentations for $\eta_F$ and $\varepsilon_F$ as the composites:
\begin{equation}\label{eqtn_consist_of_unit_and_coun}
\eta_F\colon \textrm{Id} \xrightarrow{\eta_H} H\,H^* \xrightarrow{H\eta_G{H^*}} H\,G\,G^*\,H^*,\quad \varepsilon_F \colon H\,G\,G^!\,H^! \xrightarrow{H\varepsilon_G{H^!}} H\,H^! \xrightarrow{\varepsilon_H} \textrm{Id}.
\end{equation}

\begin{PROP}\label{prop_map_f_g}
	Let $F = H G$, where 
$G$ and $H$ admit fully faithful left and right adjoints. Then 
	\begin{enumerate}
		\item[(i)] $\f_H$ is decomposed as $\f_H \colon H^* \xrightarrow{\eta_G{H^*}} G\,G^*\,H^* \xrightarrow{G\f_F} G\,G^!\,H^! \xrightarrow{\varepsilon_G{H^!}} H^!$,
		\item[(ii)] we have decomposition $\f_F \colon F^* \simeq G^*\, H^* \xrightarrow{\f_G{H^*}} G^! \, H^* \xrightarrow{G^!\f_H} G^!\, H^! \simeq F^!$,
		\item[(iii)] we have decomposition $\f_F\colon F^* \simeq G^*\, H^* \xrightarrow{G^*\f_{H}}G^* H^! \xrightarrow{\f_G{H^!}} G^! \, H^! \simeq F^!$.
	\end{enumerate} 
\end{PROP}
\begin{proof}
	We have, by formulas (\ref{eqtn_property_of_f_F}) and (\ref{eqtn_consist_of_unit_and_coun}): 
	\begin{equation}\label{eqtn_long_cal}
	\begin{aligned}
	\varepsilon_H \circ H((\varepsilon_G{H^!}) \circ (G\f_F ) \circ (\eta_G {H^*}))\circ \eta_H = & \varepsilon_H \circ (H\varepsilon_G{H^!}) \circ (H\,G\f_F ) \circ (H\eta_G {H^*})\circ \eta_H \\
	  =& \varepsilon_F \circ (F\f_F ) \circ \eta_F = \textrm{Id},
	\end{aligned} 
	\end{equation}
	Since the composite $\Hom(H^*, H^!) \xrightarrow{H} \Hom(H\,H^*, H\,H^!) \xrightarrow{(-)\circ \eta_H} \Hom(\textrm{Id}, H\,H^!)$ is the adjunction isomorphism, 
	we have an injection
	\begin{equation}\label{eqtn_injections}
	\Hom(H^*, H^!) \hookrightarrow \Hom(H\,H^*, H\,H^!),
	\end{equation}
	hence $\f_H$ is uniquely determined by $H\f_H$. Since $\eta_H$ and $\varepsilon_H$ are isomorphisms, equalities (\ref{eqtn_property_of_f_F}) and (\ref{eqtn_long_cal}) imply (i).
	
	Now consider a diagram:
	\[
	\xymatrix{H\,G\,G^*\,H^* \ar[rrr]^{H\,G\f_G{H^*}} &&&  H\,G\,G^!\,H^* \ar[rr]^{H\,G\, G^!\,\f_H} \ar[d]^{H\varepsilon_G{H^*}} && H\,G\,G^!\,H^! \ar[d]^{H\varepsilon_G{H^!}} \\
		H\,H^* \ar[rrr]^{\textrm{Id}} \ar[u]^{H\eta_G{H^*}} &&& H\,H^* \ar[rr]^{H\f_H} && H\,H^!  }
	\]  
	Its left square commutes by formula (\ref{eqtn_property_of_f_F}) for $\f_G$. Commutativity of the right square follows from the functoriality of $\varepsilon_G$. Together with (\ref{eqtn_consist_of_unit_and_coun}) this implies
	\begin{align*}
	\varepsilon_F \circ F\,(G^! \f_H \circ \f_G H^*) \circ \eta_F &= \varepsilon_H \circ H \varepsilon_G H^! \circ HGG^!\f_H \circ HG\f_GH^* \circ H\eta_GH^* \circ \eta_H \\
	&= \varepsilon_H \circ H \f_H \circ \eta_H = \textrm{Id}.
	\end{align*}
	The statement (ii) follows again from (\ref{eqtn_property_of_f_F}) and (\ref{eqtn_injections}) for $F$. Similarly for (iii).
\end{proof}

\bibliographystyle{alpha}
\bibliography{../../ref}

\def\cprime{$'$} \def\cprime{$'$}
\begin{thebibliography}{BVdB03}

\bibitem[BB15]{BodBon}
A.~Bodzenta and A.~Bondal.
\newblock Flops and spherical functors.
\newblock {\em arXiv:1511.00665 [math.AG]}, 2015.

\bibitem[BBD82]{BBD}
A.~A. Be{\u\i}linson, J.~Bernstein, and P.~Deligne.
\newblock Faisceaux pervers.
\newblock In {\em Analysis and topology on singular spaces, {I} ({L}uminy,
  1981)}, volume 100 of {\em Ast\'erisque}, pages 5--171. Soc. Math. France,
  Paris, 1982.

\bibitem[Bir37]{Birk}
G.~Birkhoff.
\newblock Rings of sets.
\newblock {\em Duke Math. J.}, 3(3):443--454, 1937.

\bibitem[BK89]{BK1}
A.~I. Bondal and M.~M. Kapranov.
\newblock Representable functors, {S}erre functors, and reconstructions.
\newblock {\em Izv. Akad. Nauk SSSR Ser. Mat.}, 53(6):1183--1205, 1337, 1989.

\bibitem[BK90]{BK}
A.~I. Bondal and M.~M. Kapranov.
\newblock Framed triangulated categories.
\newblock {\em Mat. Sb.}, 181(5):669--683, 1990.

\bibitem[BO95]{BonOrl}
A.~I. Bondal and D.~O. Orlov.
\newblock Semiorthogonal decompositions for algebraic varieties.
\newblock {\em arXiv:alg-geom/9506012}, 1995.

\bibitem[BO02]{BO}
A.~I. Bondal and D.~O. Orlov.
\newblock Derived categories of coherent sheaves.
\newblock In {\em Proceedings of the {I}nternational {C}ongress of
  {M}athematicians, {V}ol. {II} ({B}eijing, 2002)}, pages 47--56, Beijing,
  2002. Higher Ed. Press.

\bibitem[Bon89]{B}
A.~I. Bondal.
\newblock Representations of associative algebras and coherent sheaves.
\newblock {\em Izv. Akad. Nauk SSSR Ser. Mat.}, 53(1):25--44, 1989.

\bibitem[Bon13]{Bon2}
A.~I. Bondal.
\newblock Operations on {$t$}-structures and perverse coherent sheaves.
\newblock {\em Izv. Ross. Akad. Nauk Ser. Mat.}, 77(4):5--30, 2013.

\bibitem[Bri99]{Br6}
T.~Bridgeland.
\newblock Equivalences of triangulated categories and {F}ourier-{M}ukai
  transforms.
\newblock {\em Bull. London Math. Soc.}, 31(1):25--34, 1999.

\bibitem[Bri02]{Br1}
T.~Bridgeland.
\newblock Flops and derived categories.
\newblock {\em Invent. Math.}, 147(3):613--632, 2002.

\bibitem[BVdB03]{BvdB}
A.~I. Bondal and M.~Van~den Bergh.
\newblock Generators and representability of functors in commutative and
  noncommutative geometry.
\newblock {\em Mosc. Math. J.}, 3(1):1--36, 258, 2003.

\bibitem[Dan80]{Dan}
V.~I. Danilov.
\newblock Decomposition of some birational morphisms.
\newblock {\em Izv. Akad. Nauk SSSR Ser. Mat.}, 44(2):465--477, 480, 1980.

\bibitem[Gro67]{EGAIV}
A.~Grothendieck.
\newblock \'{E}l\'ements de g\'eom\'etrie alg\'ebrique. {IV}. \'{E}tude locale
  des sch\'emas et des morphismes de sch\'emas {IV}.
\newblock {\em Inst. Hautes \'Etudes Sci. Publ. Math.}, (32):361, 1967.

\bibitem[Kaw02]{Kaw}
Y.~Kawamata.
\newblock {$D$}-equivalence and {$K$}-equivalence.
\newblock {\em J. Differential Geom.}, 61(1):147--171, 2002.

\bibitem[KPS17]{KruPloSos}
A.~Krug, D.~Ploog, and P.~Sosna.
\newblock Derived categories of resolutions of cyclic quotient singularities.
\newblock {\em arXiv:1701.01331 [math.AG]}, 2017.

\bibitem[Laz04]{Laz}
R.~Lazarsfeld.
\newblock {\em Positivity in algebraic geometry. {I}}, volume~48 of {\em
  Ergebnisse der Mathematik und ihrer Grenzgebiete. 3. Folge. A Series of
  Modern Surveys in Mathematics [Results in Mathematics and Related Areas. 3rd
  Series. A Series of Modern Surveys in Mathematics]}.
\newblock Springer-Verlag, Berlin, 2004.
\newblock Classical setting: line bundles and linear series.

\bibitem[Lur]{Lur}
J.~Lurie.
\newblock Derived {A}lgebraic {G}eometry {XII}: {P}roper {M}orphisms,
  {C}ompletions, and the {G}rothendieck {E}xistence {T}heorem.
\newblock {\em preprint available at http://www.math.harvard.edu/~lurie/}.

\bibitem[Orl92]{Orl}
D.~O. Orlov.
\newblock Projective bundles, monoidal transformations, and derived categories
  of coherent sheaves.
\newblock {\em Izv. Ross. Akad. Nauk Ser. Mat.}, 56(4):852--862, 1992.

\bibitem[PV14]{PsaVit}
C.~Psaroudakis and J.~Vit{\'o}ria.
\newblock Recollements of module categories.
\newblock {\em Appl. Categ. Structures}, 22(4):579--593, 2014.

\bibitem[Ric10]{Riche}
S.~Riche.
\newblock Koszul duality and modular representations of semisimple {L}ie
  algebras.
\newblock {\em Duke Math. J.}, 154(1):31--134, 2010.

\bibitem[Rou08]{Rou2}
R.~Rouquier.
\newblock Dimensions of triangulated categories.
\newblock {\em J. K-Theory}, 1(2):193--256, 2008.

\bibitem[{\v S}VdB16]{SpVdB}
\v{S}. {\v S}penko and M.~Van~den Bergh.
\newblock Semi-orthogonal decomposition of {GIT} quotient stack.
\newblock {\em arXiv:1603.02858}, 2016.

\bibitem[Tod14]{Toda}
Yu. Toda.
\newblock Stability conditions and birational geometry of projective surfaces.
\newblock {\em Compos. Math.}, 150(10):1755--1788, 2014.

\bibitem[To{\"e}12]{Toen1}
B.~To{\"e}n.
\newblock Derived {A}zumaya algebras and generators for twisted derived
  categories.
\newblock {\em Invent. Math.}, 189(3):581--652, 2012.

\bibitem[TV08]{ToeVaq}
B.~To{\"e}n and M.~Vaqui{\'e}.
\newblock Alg\'ebrisation des vari\'et\'es analytiques complexes et
  cat\'egories d\'eriv\'ees.
\newblock {\em Math. Ann.}, 342(4):789--831, 2008.

\bibitem[VdB04]{VdB}
M.~Van~den Bergh.
\newblock Three-dimensional flops and noncommutative rings.
\newblock {\em Duke Math. J.}, 122(3):423--455, 2004.

\end{thebibliography}
\end{document}